\documentclass[12pt]{article}
\usepackage{graphicx}
\usepackage[centering, margin={0.8in, 0.5in}, includeheadfoot]{geometry}

\usepackage[utf8x]{inputenc}
\usepackage{amssymb,pdfsync,epstopdf,verbatim,gensymb,amsmath,amsthm}
\usepackage{nccmath, mathtools}
\usepackage{textcomp}
\usepackage{amsmath}
\usepackage{float,pbox}
\usepackage{subfig}
\usepackage{amsfonts}
\usepackage{epsfig,color, hyperref}

\usepackage{appendix}

\def\R {{\mathbb R}}

\def\H01{{H_0^1(\Omega)}}
\def\L2{{L^2(\Omega)}}

\usepackage{bmpsize}
\newtheorem{theorem}{Theorem}[section] 

\newtheorem{lemma}{Lemma}[section]

\newtheorem{proposition}{Proposition}
\newtheorem{algorithm}{Algorithm}[section]
\DeclareMathOperator*{\argmin}{arg\,min}

\newcommand{\Beq}{\begin{equation}}
\newcommand{\Eeq}{\end{equation}}
\newcommand{\beq}{\begin{equation*}}
\newcommand{\eeq}{\end{equation*}}
\newcommand{\bal}{\begin{align}}
\newcommand{\eal}{\end{align}}

\renewcommand{\L}{\langle}

\newcommand{\bp}{\begin{prob}}
\newcommand{\ep}{\end{prob}}
\newcommand{\bpr}{\begin{proof}}
\newcommand{\epr}{\end{proof}}

\newcommand{\tred}[1]{{\color{red}{#1}}}

\newcommand{\bel}[1]{\begin{equation}\label{#1}}
\newcommand{\ee}{\end{equation}}

\graphicspath{{./figures/}}

\author{Anwesa Dey \thanks{anwesary@math.utah.edu, Department of Mathematics, University of Utah, John A. Widtsoe Building 
155 S 1400 E, Salt Lake City, 84112, Utah, USA. } \and Alfio Borz\`i \thanks{alfio.borzi@mathematik.uni-wuerzburg.de, Institut f\"ur Mathematik, Universit\"at W\"urzburg, Emil-Fischer-Strasse 30, 97074 W\"urzburg, Germany.} \and Souvik Roy \thanks{souvik.roy@uta.edu, Department of Mathematics, University of Texas at Arlington, 411 S. Nedderman Drive, Arlington, 76019, Texas, USA. }}
\date{}

\begin{document}

\title{\bf A high contrast and resolution reconstruction algorithm in quantitative photoacoustic tomography\thanks{Partially supported by the BMBF-Project iDeLIVER : Intelligent MR Diagnosis of the Liver by Linking Model and Data-driven Processes, US National Science Foundation Grant No: 2309491, and  University of Texas at Arlington, Research Enhancement Program Grant No: 2022-605.}}

\maketitle
\begin{abstract}
{A framework for reconstruction of optical diffusion and absorption coefficients in quantitative
photoacoustic tomography is presented. This framework
is based on a Tikhonov-type functional with a
regularization term promoting sparsity of the absorption
coefficient and a prior involving a Kubelka-Munk
absorption-diffusion relation that allows to obtain superior reconstructions.
The reconstruction problem is formulated as the minimization
of this functional subject to the differential constraint given by a
photon-propagation model. The solution of this problem is obtained
by a fast and robust sequential quadratic hamiltonian algorithm
based on the Pontryagin maximum principle. Results of several numerical
experiments demonstrate that the proposed computational strategy is able to
obtain reconstructions of the optical coefficients with high contrast and resolution
for a wide variety of objects.}

\end{abstract}



\section{Introduction}
Photoacoustic tomography (PAT) is an emerging hybrid imaging technique that combines two imaging modalities: optical tomography, with a large contrast of optical parameters in an object, and ultrasound imaging, which provides high resolution images. The PAT technique is based on the photoacoustic effect where an object is exposed to a short-pulse optical radiation that propagates through it and is, in part, partially absorbed. This phenomenon leads to heating and thermal expansion of the region where the propagation occurs, so that acoustic pulses are emitted that travel back to the boundary of the object, 
where they can be measured { \cite{Wang2009,Xia2014}}. {From these measured acoustic pulses, one can obtain the initial pressure distribution and, subsequently, optical coefficients like the absorption and diffusion, through the formulation of inverse problems.}

The reconstruction of the initial pressure distribution denoted with $\mathcal{H}(x),~ x \in \Omega \subset \mathbb{R}^n$, from the measured acoustic pulses on the boundary of the object $\partial\Omega$, can be formulated as an inverse source problem for the wave operator as follows:
{Let $p(x,t)$ and $\mathcal{C}(x)$ denote the pressure and the speed of 
the acoustic wave, respectively. Then $p$ satisfies the following equation
\begin{equation}\label{eq:acoustic_wave}
\begin{aligned}
&p_{tt} - \mathcal{C}^2(x) \, \Delta p = 0, \quad (x,t) \in \Omega\times[0,{T}]\\
& p(x,0) = \mathcal{H}(x), \quad x\in \Omega\\\
&\dfrac{\partial p}{\partial t}(x,0) = 0, \quad x \in \Omega\\
\end{aligned}
\end{equation}
where ${T}$ represents the time period of propagation of the acoustic waves starting with the initial pressure field $\mathcal{H}$. Then, the time-dependent acoustic wave signals of the form 
\begin{equation}\label{eq:acoustic_signal}
p(x,t), ~\dfrac{\partial{p}}{\partial n}(x,t)  \quad(x,t) \in \partial\Omega\times[0,{T}]
\end{equation}
are measured at the acquisition boundary $\partial\Omega$. 
The inverse problem is to recover $\mathcal{H}$ given the data in \eqref{eq:acoustic_signal}. This is also known as the classical PAT inverse problem \cite{Bal10a,Bal11,Bal12,Bal10,Cox09,Finch09}}

The aforementioned inverse problem has been very well studied; see, e.g.,  \cite{Cox09,Finch09}, and the solution to this problem gives the initial pressure distribution $\mathcal{H}$ profile in the entire domain $\Omega$. However, 
this procedure does not give information about the more specific properties of the object like the optical parameters of absorption and diffusion coefficient. {Thus, to reconstruct the optical parameters from the initial pressure distribution, a second inverse problem is solved, which is known as quantitative photoacoustic tomography (QPAT).} {One of the major applications of QPAT is to provide accurate images of the heart for the study of congenital heart diseases \cite{Gao2018}}. Other applications of QPAT include detecting cancerous tissues, thus facilitating biomedical studies \cite{Gao12,Sandall11}. {Notice that there are other optical imaging 
methods that have a structure similar to QPAT in the sense that these methods make use of the photoacoustic effect to determine specific internal features in an object.} For example, 
the fluorescence optical diffuse tomography (FDOT), which is used to detect fluorescing targets in tissues \cite{Ntz06,Schot11}.

However, in QPAT and similar modalities, the accurate reconstruction of the optical parameters poses several theoretical and computational challenges 
that are under investigation by many research groups. Our purpose is to contribute to this research work with a novel QPAT reconstruction framework 
with high resolution and high contrast, which is of paramount importance especially in the context of biomedical imaging for detection of cancerous tissues.

Mathematically, the initial pressure distribution $\mathcal{H}({x})$ is a function of the absorption coefficient $\sigma_a({x})$, of the intensity of radiation $u({x})$, and of the Gr\"{u}neisen coefficient $\Gamma({x})$, which measures the amount of ultrasound generated by the absorbed radiation at the point ${x}$. {In general, 
the intensity of radiation $u$ is a function of space and time and its 
evolution is governed by the radiative transfer equation (RTE) (see e.g., \cite[Section 3.2]{Arridge1999}, \cite[Chapter 7]{Ishimaru1978}, \cite{Sandall11}). However, assuming the use of short light pulses such that the radiation propagates for a very short time at the scale of small acoustic wave lengths, we have that the scattering is significantly larger than absorption (see \cite[Chapter 9] {Ishimaru1978}, \cite[Section 3.3]{Arridge1999}). Thus, the photons are scattered almost uniformly in all directions leading to an almost uniform angular distribution. This phenomenon is prevalent inside biological tissues \cite{Farrell1992,Sandall11} and blood \cite[Chapter 9]{Ishimaru1978}. In this case, it has been shown in \cite[Sec 3.3]{Arridge1999} that the photon current satisfies a steady state assumption, which results in the diffusion approximation of the RTE equation. Therefore, in this setting, the intensity of radiation  $u({x})$ is modelled by the following diffusion equation \cite{Bal11,Bal12,Bal10}:}
\[
\begin{aligned}
-\nabla\cdot D({x}) \, \nabla u({x}) + \sigma_a({x}) \, u({x}) &= 0, \qquad \mbox{ in } \Omega , \,\\
u({x}) &= g({x}), \quad \mbox{ on } \partial\Omega ,\\
\end{aligned}
\]
where $g$ is the known intensity of the radiation that is applied at the boundary of the object.
In this framework, the initial pressure distribution function is given by
\[
\mathcal{H}({x})=\Gamma({x}) \, \sigma_a({x}) \, u({x}),\quad {x} \in \Omega .
\]
The QPAT inverse problem requires to recover the absorption coefficient $\sigma_a({x})$ and the diffusion coefficient $D({x})$, given $\mathcal{H}$ and $g$.

There are several works related to theoretical and computational frameworks for reconstructions in QPAT; see, e.g.,\cite{Bal11,Ban08,Gao15,Halt2018,Yao10}. In the work \cite{Bal10}, it has been shown that, if $\Gamma$ is given, then 
the coefficients $(D,\sigma_a)$ can be uniquely and stably reconstructed based on two different sets of initial pressure distribution data generated by two well-chosen boundary conditions $g_1$ and $g_2$ for two copies of the diffusion equation 
given above. Also in 
\cite{Bal10}, we find a method for reconstruction that relies on the solution of a set of transport and elliptic equations. Further, in \cite{Bal11,Bal12} it has been shown that all three parameters $(D,\sigma_a,\Gamma)$ cannot be reconstructed simultaneously, irrespective of the number of illuminations unless one considers the QPAT inversion problem at different wavelengths. The correspondence between the QPAT inversion problem and parameter estimation with PDE's was illustrated in \cite{Bal13}. Other works on QPAT include: the multiple-source framework in \cite{Zemp10}, and its analysis \cite{Shao11}, a PDE-optimization framework using the radiative transfer equation \cite{Abd05}, the recovery of absorption coefficient of a small absorber \cite{Ammari11}, the reconstruction of absorption and scattering coefficients \cite{Tar12}, and the analysis of the linearised QPAT problem \cite{Halt2018}.  

In addition to the theoretical investigation of the QPAT problem, there are several numerical reconstruction algorithms for solving related inversion problems.  In \cite{Cox06}, the authors develop an iterative method using the radiation map to reconstruct the optical parameters. A non-linear model-based inversion scheme was formulated in \cite{Lau10} for determining chromophore concentrations in 2D. In \cite{Sara13}, a gradient-based scheme for QPAT using the radiative transfer equation is discussed. A new iterative reconstruction algorithm using a different form of the radiation function and Robin-boundary data was proposed in \cite{Ren13}. It 
appears from the simulation results in these works that existing algorithms for QPAT have difficulties in reconstructing accurately both $D$ and $\sigma_a$. Furthermore,  these algorithms do not perform well to reconstruct the optical coefficients with high contrast, which is an inherent property of tissues, thus limiting their range of applicability in biomedical problems. 
 
One of the major reason for the poor reconstruction of $D$ is due to the fact that the initial pressure distribution data does not directly contain information on $D$. 
On the other hand, it has been shown that there is a relation between $D$ and $\sigma_a$ and the reduced scattering coefficient $\sigma_s$, which is given by \cite{Aronson99,Durian98}
\begin{equation}\label{eq:D_sigma}
D = \dfrac{1}{3(\sigma_a+\sigma_s)} .
\end{equation}
This fact shows a deficiency in the use of the diffusion approximation for 
reconstructing $D$, which would require to employ the 
radiative transfer (or transport) equation (RTE) \cite{Ishimaru1978} as 
the adequate model for light transport in tissues. Therefore, although 
the main point in photo acoustic is a thermodynamics phenomenon that explains how absorbed energy is transformed in a pressure wave, also scattering 
is relevant for a successful reconstruction of the diffusion coefficient, 
as it is illustrated by results in, e.g., \cite{Tar12}. 

{However, in some applications a proportionality between 
$\sigma_a$ and $\sigma_s$ appears. In particular, 
the following relation between $\sigma_a$ and $\sigma_s$ has for long been described by the Kubelka-Munk theory in the diffusive 
regime \cite{Kubel31,Hecht76,Wei2003,Dzim2012,Roy2012,Fanjul2020}}. We have 
\begin{equation} \label{eq:KM}
\sigma_s = \dfrac{2R_{\infty}}{(1-R_\infty)^2}\sigma_a,
\end{equation}
where $R_\infty$ is the diffuse reflectance of the biological object. In the near infrared wavelength region, $\R_\infty$ has been experimentally found to be lying somewhere in the range of 0.5-1.25 (see Figure 4 in \cite{Fanjul2020}). 
Furthermore, a similar linear relation between $\sigma_a $ and $\sigma_s$ has also been found to be experimentally true for the 
tissues of brain, breast, heart, lungs, bone, prostate and tumor; see, e.g.,  \cite{Dimo05,Sandall11}, and Table \ref{table:relation}, where we report measured values of the absorption and reduced scattering coefficients for various human tissues. 

\begin{table}[H]
\centering
\caption{Values of the absorption and the reduced scattering coefficients for various tissues at the near-infrared wavelengths \cite{Sandall11}.}
\begin{tabular}{|c| c| c| c|}
\hline
Tissue & Wavelength (nm) &$\sigma_a$ (cm$^{-1}$) & $\sigma_s$ (cm$^{-1}$)\\ [0.5ex]
\hline
Brain  & 760  & 0.11-0.17& 4.0-10.5 \\
Breast  & 760  & 0.031-0.1& 8.3-12.0 \\
Heart  & 661  & 0.12-0.18& 5.22-90.80 \\
Lungs & 661  & 0.49-0.88& 21.14-22.52 \\
Bone  & 760  & 0.07-0.09& 11.9-14.1 \\
Prostate  & 732  & 0.09-0.72& 3.37-29.8 \\
Tumor  & 795  & 0.068-0.102& 12.4-13.1 \\[1ex]
\hline
\end{tabular}
\label{table:relation}
\end{table} 

These results suggest that the reconstruction of the diffusion 
coefficient in the framework of the diffusion approximation can be improved 
if we include our qualitative prior knowledge represented by 
\eqref{eq:D_sigma} and \eqref{eq:KM} in the formulation of our inverse 
problem.

For this purpose, based on the results given in Table \ref{table:relation} and supported by the Kubelka-Munk relation, we construct our prior based on the following tentative relation between the absorption and the reduced scattering coefficients
{\begin{equation}\label{eq:rel}
\sigma_s \approx c~ \sigma_a.
\end{equation}
This choice is based on results of measurements 
indicating that the value of $c$ varies among different biological tissues in the approximate range $(30, 800)$}. Notice that in our approach we interpret \eqref{eq:rel} as a mean approximate relation between $\sigma_a$ and $\sigma_s$ in the tissues 
 (with the value of $c$ arbitrarily chosen in a given interval determined by the 
 application) and aim at using this information for better reconstruction of the QPAT coefficients. 

We demonstrate that the consistent use of \eqref{eq:D_sigma} and \eqref{eq:rel} embedded as a prior in our nonsmooth optimization problem results in a novel reconstruction strategy that allows to obtain reconstructions of $D$ and $\sigma_a$ with high resolution and high contrast. However, it is important 
to emphasize that the high-quality reconstructions that we obtain do not 
necessarily satisfy the relations \eqref{eq:D_sigma} and \eqref{eq:rel}. 
On the other hand, we show that our computational framework is certainly 
valid for biomedical applications with near-infrared light illumination of tissues in order to accurately detect the cancerous ones inside the organs mentioned in Table \ref{table:relation}, where the Kubelka-Munk relation \eqref{eq:rel} holds approximately true.

In our problem formulation, we consider the minimization of a 
Tikhonov-type least squares functional with regularization terms including the relation between $D$ and $\sigma_a$ mentioned above. Further, we 
include terms that promote sparsity in the reconstruction, in order to accommodate the fact 
that, in biomedical application, different tissues appear distributed as almost non-overlapping bounded regions of the domain under consideration. However, 
while the introduction of nonsmooth penalization terms in the functional 
considerably contribute to the ability of our method to provide 
sharp reconstructions as in \cite{Ade2018,Roy2018}, they require the development of sophisticated 
optimization techniques that apply in the nonsmooth case. We refer 
to \cite{Stadler2007} for a work that pioneered the use of $L^1$ penalization 
terms to promote sparsity in PDE optimization and the development of 
semi-smooth Newton (SSN) methods to solve the corresponding problems. 
On the other hand, the implementation of a SSN scheme requires the application of the Hessian of the optimization problem and, for this purpose,
in addition to solving the model equation and its adjoint, one needs
to solve other two linearized equations. This fact and further implementation issues make 
the development of the SSN scheme difficult to practitioners. 

To overcome these difficulties, we extend the sequential quadratic hamiltonian (SQH) algorithm, proposed in \cite{Breit18,Breit19a,Breit19} for nonsmooth PDE optimal control problems, to solve our inverse problem. This algorithm 
is based on the Pontryagin maximum principle (PMP), which is well-known 
in the theory of optimal control of dynamical systems. 
Our focus on the SQH algorithm is motivated 
by the fact that it is easy to implement, it appears very efficient and robust, 
and it is wellposed as an iterative scheme \cite{Breit18,Breit19a}. A specific feature 
of the iterative SQH method is that it relies on a pointwise optimization 
procedure. 

Several numerical experiments are performed 
to demonstrate the computational effectiveness of our methodology. The results of these experiments show that our QPAT inverse problem, 
which includes a prior relation on $D$ and $\sigma_a$, and 
the SQH algorithm solving this problem, result in a 
fast and accurate reconstruction procedure. These are the main novelties 
of our work.

This paper is organized as follows. In the next section, we discuss our 
optimization formulation of the QPAT inverse problem, and 
illustrate our use of a relation between $D$ and $\sigma$. Further, we 
discuss $L^2$ and $L^1$ penalization terms. 
In Section \ref{sec:theory},  we provide theoretical 
consideration on solutions to our minimization problem and discuss their characterization in the PMP framework. Section \ref{sec:numerical_scheme} is devoted to illustrating all details of our 
numerical optimization procedure that includes wellposedness of 
the SQH algorithm and 
the numerical approximation of our PDE governing model and of its 
optimization adjoint. In Section \ref{sec:results}, results of numerical experiments 
are presented that successfully validate our QPAT framework. For this purpose, 
different test cases with synthetic data are considered that allow to 
investigate the role of the optimization parameters and the ability of our 
computational framework to provide reconstructions of the 
coefficients $D$ and $\sigma$ with high contrast and resolution, 
also in the case of data affected by noise. A section of conclusions completes 
this work.

\section{Formulation of a QPAT optimization problem}
We start discussing the governing model of our QPAT inverse problem. 
For this purpose, we specify the dimension $n=2$, and refer 
to the set of coordinates $(x,y)$. 
Let $u(x,y)$ be the total number of photons at a point $(x,y)$ inside an object represented by the domain $\Omega$. 
In our discussion, $\Omega$ is either 
a rectangle or a convex subset of $\R^2$ with boundary $\partial \Omega$ 
of class $C^2$. We assume that $u$ satisfies the following photon-propagation problem
\begin{equation}\label{eq:DA}
\begin{aligned}
-\nabla\cdot (D(x,y) \, \nabla u(x,y)) + \sigma_a(x,y) \, u(x,y) &= 0, \qquad (x,y) \in \Omega\\
u(x,y) &=  g(x,y), \quad (x,y) \in \partial\Omega , 
\end{aligned}
\end{equation}
where $D(x,y)> 0$ represents the diffusion coefficient at $(x,y)$. The non-negative absorption coefficient $\sigma_a$ is given by 
\[
\sigma_a(x,y) = \sigma(x,y)+\sigma_b  ,
\] 
where $\sigma_b$ is a known background absorption coefficient and $\sigma$ 
represents the deviation from this value in correspondence of organs and sites of tumor. {We remark that even though $\sigma$ is a signed function, however later on,  we will impose bounds on $\sigma$ such that $\sigma_a >0$}. Further, the function $g$ represents the known intensity of radiation at the boundary of the object. For the reconstruction of $D$ and $\sigma_a$, we require the knowledge, by measurements, of the optical energy given by 
\begin{equation}\label{eq:optical_energy}
\mathcal{H}(x,y,\sigma,u_i) = \Gamma \, (\sigma(x,y)+\sigma_b) \, u_i(x,y), \qquad i=1,2 ,
\end{equation}
for two sets of boundary radiation functions $g_i$, where $\Gamma$ is the known Gr\"{u}neisen coefficient {that is assumed to be a known constant.}

Our QPAT inverse problem consists in obtaining the unknown pair $(D,\sigma)$, given $H$ and $g$. For this purpose, we formulate a minimization problem with the following cost functional
\begin{equation}\label{eq:cost_functional}
\begin{aligned}
J(D,\sigma,u_1,u_2) &= \frac{\alpha}{2}\sum_{i=1}^2\int_\Omega(\mathcal{H}(x,y,\sigma,u_i) - G_i^{\delta}(x,y))^2~dxdy+\frac{\xi_1}{2}\int_\Omega \sigma^2(x,y)~dxdy\\
&+ \frac{\xi_2}{2}\int_\Omega \left(D(x,y)-\frac{1}{3c~(\sigma(x,y)+\sigma_b)}\right)^2~dxdy
+\gamma \int_\Omega |\sigma(x,y)|~dxdy,
\end{aligned}
\end{equation}
where $G_i^{\delta} \in L^\infty(\Omega)$ is a given noisy measurement function of the initial pressure distribution corresponding to the boundary radiation functions $g_i$, $i=1,2$. 
Therefore the first term in the objective functional $J$ represents a least-squares data fit of the initial pressure distribution measurements. We assume that the background value of the absorption coefficient $\sigma_b$ is known and therefore $\sigma$ equals zero in the background. In this setting, it is reasonable to introduce a $L^2$ and $L^1$ regularization of $\sigma$, which result in the second and last terms in \eqref{eq:cost_functional}. In particular, the $L^1$ term promotes sparsity of $\sigma$, thus enhancing contrast. 

One novelty of our work is the prior given by the third term in \eqref{eq:cost_functional}. To explain this term, we 
recall \eqref{eq:D_sigma} and the related discussion. Hence, by {assuming $c\in [30,800]$, we can make the following approximation 
$$
3 \, (\sigma_a+ \sigma_s) \approx 3 \, (\sigma_a+ c\sigma_a ) 
\approx 3c\sigma_a.
$$}
We remark that even though we will be choosing $c\in [30,800]$ in our numerical experiments, the choice of $c$ is not essential. Thus, we employ the following function as 
prior for the diffusion coefficient $D$
{\begin{equation}\label{eq:D_rel}
 D(x,y)  \approx  \bar{D} :=\frac{1}{3c~(\sigma(x,y)+\sigma_b)}.
\end{equation}
Therefore the third term in $J$ measures the deviation of $D$ in the tissues and tumors from $\bar D$. }

Our QPAT optimization problem is stated as follows
\begin{equation}\label{eq:min_problem}
\begin{aligned}
&\min_{D,\sigma \in D_{ad} \times \Sigma_{ad}}~ J(D,\sigma,u_1,u_2) \\
\mbox{s.t.} \qquad & \mathcal{L}(u_1,D,\sigma,g_1) = 0,\\
& \mathcal{L}(u_2,D,\sigma,g_2) = 0,\\
\end{aligned} 
\end{equation}
(`s.t.' means `subject to') where $ \mathcal{L}(u_i,D,\sigma,g_i) = 0$, $i = 1,2$, denote \eqref{eq:DA} for the boundary radiation functions $g_i$, $i = 1,2$, 
respectively, and our aim is to recover $D$ and $\sigma$ in the following admissible sets
{\begin{equation}\label{eq:admissible_set}
\begin{aligned}
&D_{ad} = \lbrace D \in L^2(\Omega): 0< D_l \leq D(x,y) \leq D_r, \mbox{ a.e. in  } \Omega \rbrace,\\
&\Sigma_{ad} = \lbrace \sigma \in L^2(\Omega): \sigma_l \leq \sigma(x,y) \leq \sigma_r, \mbox{ a.e. in  } \Omega, \rbrace,
\end{aligned}
\end{equation}
where $\sigma_l$ is chosen such that the total absorption coefficient $\sigma_a =\sigma+\sigma_b >0$, e.g., one can choose $\sigma_l = -\sigma_b + \sigma_{\epsilon}$, with $\sigma_{\epsilon} $ a small positive number.}

\section{Analysis of the QPAT inverse problem}
\label{sec:theory}

We remark that our diffusion model \eqref{eq:DA}, with the given 
bounds on the coefficients and Dirichlet boundary conditions, is a well-studied problem and 
we refer to \cite{GilbargTrudinger1998,Lady68} for two classical 
references. Specifically, one can verify that with $D\in D_{ad}$ 
and $\sigma\in \Sigma_{ad}$ our operator is uniformly elliptic 
and the boundedness conditions on the coefficients required in, e.g., 
\cite{GilbargTrudinger1998} are satisfied. Then, assuming 
that the Dirichlet data $g$ corresponds to the trace of a $H^1(\Omega)$ function, 
the following proposition holds whose proof can be found e.g., in \cite{Lady68}. 

\begin{proposition}\label{th:existence_forward}
 Let $D\in D_{ad}$, $\sigma\in \Sigma_{ad}$. Then there exists a unique solution $u \in H^1(\Omega)$ to \eqref{eq:DA}.
 \end{proposition}

Additional regularity of the solution, that is, 
$u \in H^1(\Omega) \cap H^2(\Omega)$,  
is attained at the interior of $\Omega$ if $D$ is assumed Lipschitz 
continuous, and this regularity holds globally if $\partial \Omega$ is of 
class $C^2$; see \cite{GilbargTrudinger1998}. 
Further results can be found in, e.g., \cite{Alberti2018}. 
In particular, we refer to \cite{Cianchi2015,Mingione2006} 
for the problem of establishing boundedness of the gradient of 
solutions to the diffusion model.

With Proposition \ref{th:existence_forward}, we have that the map $(D,\sigma) \mapsto u=u(D,\sigma)$, where $u$ 
is the solution to \eqref{eq:DA} for given 
$(D,\sigma) \in D_{ad}\times\Sigma_{ad} \subset L^\infty(\Omega) \times L^\infty(\Omega)$, is well defined. Furthermore, this map is continuous as shown 
below; see also \cite{kuchment}.

\begin{lemma}\label{sqh20}
Let $\Omega \subset \R^2$ and consider any two pairs $(u_1,D_1,\sigma_1), \, (u_2,D_2,\sigma_2) \in H^1(\Omega) \times D_{ad}  \times \Sigma_{ad}$ 
that satisfy \eqref{eq:DA} with $g$ the trace 
 of a $H^1(\Omega)$ function. Then the following estimate holds 
$$
\| u_1 - u_2 \|_{H^1 (\Omega)} 
\leq C_1 \, \big( \| D_1 - D_2 \|_{L^\infty(\Omega)} + \| \sigma_1 - \sigma_2 \|_{L^\infty(\Omega)} \big) .
$$

Furthermore, if for the given pairs the gradient of $u_1$ (respectively $u_2$), is essentially bounded, then the following estimate holds 
$$
\| u_1 - u_2 \|_{H^1 (\Omega)} 
\leq C_2 \, \big( \| D_1 - D_2 \|_{L^2(\Omega)} + \| \sigma_1 - \sigma_2 \|_{L^2(\Omega)} \big) .
$$

\end{lemma}
\begin{proof}
The two pairs satisfy both the diffusion equation in weak form. We have
$$
\left(D_i \, \nabla u_i, \nabla v \right ) + \left(\tilde{\sigma}_i \, u_i, v \right ) =0, 
\qquad v \in H_0^1(\Omega) ,
$$
where $\tilde{\sigma}_i = \sigma_i + \sigma_b$ and $u_i = g$ on the boundary, $i=1,2$. 

Therefore we have 
$$
\left(D_1 \, \nabla u_1, \nabla v \right ) + \left(\tilde{\sigma}_1 \, u_1, v \right )  
= \left(D_2 \, \nabla u_2, \nabla v \right ) + \left(\tilde{\sigma}_2 \, u_2, v \right ) . 
$$
Thus, by algebraic manipulation, we obtain 
$$
\left( D_1 \, \left[ \nabla u_1 - \nabla u_2 \right] , \nabla v \right ) 
+ \left(\tilde{\sigma}_1 \,  \left[  u_1 - u_2 \right] , v \right ) 
 = \left(\left[ D_2 - D_1 \right] \, \nabla u_2, \nabla v \right ) 
+ \left(\left[ \sigma_2 - \sigma_1 \right] \, u_2, v \right ) . 
$$
In particular, taking $v=u_1-u_2$, we have 
\begin{align*}
& \left( D_1 \, \nabla \left[ u_1 - u_2 \right] , \nabla \left[  u_1 - u_2 \right] \right ) 
+ \left(\tilde{\sigma}_1 \,  \left[  u_1 - u_2 \right] , \left[  u_1 - u_2 \right] \right ) \\
&=\left(\left[ D_2 - D_1 \right] \, \nabla u_2, \nabla \left[  u_1 - u_2 \right] \right ) 
+ \left(\left[ \sigma_2 - \sigma_1 \right] \, u_2, \left[  u_1 - u_2 \right] \right ) . 
\end{align*}
Since $D_i \in D_{ad}$ and $\sigma_i \in \Sigma_{ad}$, we have 
$$
c_1 \, \| u_1 - u_2 \|_{H^1(\Omega)}^2  
\le \left(\left[ D_2 - D_1 \right] \, \nabla u_2, \nabla \left[  u_1 - u_2 \right] \right ) 
+ \left(\left[ \sigma_2 - \sigma_1 \right] \, u_2, \left[  u_1 - u_2 \right] \right ) . 
$$
where $c_1 = \min \{D_l, \sigma_l + \sigma_b \}$. Hence, considering 
that both $D$ and $\sigma$ are bounded, we have the estimate 
$$
c_1 \, \| u_1 - u_2 \|_{H^1(\Omega)}  \le c_2 \, \big(  \| D_1 - D_2 \|_{L^\infty(\Omega)} + \| \sigma_1 - \sigma_2 \|_{L^\infty(\Omega)}   \big) ,
$$
where $c_2=\max \{ \|u_2\|_{L^2(\Omega)} , \|\nabla u_2\|_{L^2(\Omega)} \}$. 
Hence, the first statement of the lemma is proved with $C_1=c_2/c_1$. 

Further, by the assumption that $\| \nabla u_2 \| _{L^{\infty}(\Omega)} \le M$ 
(resp. $\| \nabla u_1 \| _{L^{\infty}(\Omega)} \le M$),  
for some positive constant $M$, we have 
\begin{align*}
c_1 \, \| u_1 - u_2 \|_{H^1(\Omega)}^2  
& \le  \| \nabla u_2 \|_{L^\infty(\Omega)}    \, \| D_2 - D_1 \|_{L^2(\Omega)} 
\, \|\nabla \left[  u_1 - u_2 \right] \|_{L^2(\Omega)} \\
&+ \| u_2 \|_{L^\infty(\Omega)}    \, \| \sigma_2 - \sigma_1 \|_{L^2(\Omega)} 
\, \|  u_1 - u_2 \|_{L^2(\Omega)}   \\
& \le \big( \| \nabla u_2 \|_{L^\infty(\Omega)}    \, \| D_2 - D_1 \|_{L^2(\Omega)} 
+ \| u_2 \|_{L^\infty(\Omega)}    \, \| \sigma_2 - \sigma_1 \|_{L^2(\Omega)} 
 \big) \, \|  u_1 - u_2 \|_{H^1(\Omega)}   
\end{align*}
Therefore the claim of the lemma follows from the following result 
$$
c_1 \, \| u_1 - u_2 \|_{H^1(\Omega)}  \le c_3 
\, \big(  \| D_1 - D_2 \|_{L^2(\Omega)} + \| \sigma_1 - \sigma_2 \|_{L^2(\Omega)}   \big) ,
$$
where $c_3=\max \{ \|u_2\|_{L^\infty(\Omega)} , \|\nabla u_2\|_{L^\infty(\Omega)} \}$. Thus, the second statement is proved with $C_2=c_3/c_1$. 
\end{proof}

With this preparation, we can introduce the following reduced cost functional 
\begin{equation}\label{eq:reduced_cost}
\hat{J}(D,\sigma) = J(D,\sigma,u_1(D,\sigma),u_2(D,\sigma)).
\end{equation}
Notice that, with this functional, our QPAT optimization problem can 
be equivalently formulated as the following unconstrained 
minimization problem
\begin{equation}
\min_{(D,\sigma) \in D_{ad}\times\Sigma_{ad}} \hat{J}(D,\sigma) .
\label{extraMin}
\end{equation}

The proof of existence of a minimizer 
for \eqref{extraMin} by using the classical direct method 
is a delicate issue since for the minimizing 
sequences $\{ D^k \} \subset L^\infty(\Omega)$ we do not have 
the strong convergence required to prove fulfilment of the 
differential constraint \eqref{eq:DA} in the limit of the 
minimizing process; see, e.g., Chapter VI of \cite{BanksKunisch1989} for a related framework. However, this would be possible assuming a compact set of admissible controls, e.g., 
{$\tilde D_{ad} = D_{ad} \cap V_l$ where 
$V_l:=\{ v \in H^{1}(\Omega)\, | \, \|v\|_{H^{1}\left(\Omega\right)} \le l \}$ where $l$ is a given positive constant}; see, e.g., \cite{LeitaoTeixeira2015} for a 
discussion on this setting. Nevertheless, since these are sufficient and 
not necessary conditions for existence of an optimal pair $(D^*,\sigma^*)$, 
we prefer to work with the 
admissible sets defined in \eqref{eq:admissible_set} that appear 
convenient in applications, and assume existence 
of the optimal optical coefficients  $(D^*,\sigma^*)$. Our 
focus is the construction of a method that efficiently 
computes these coefficients. 

Now, concerning the characterization of solutions to our QPAT optimization problem, we can use subdifferential calculus to derive 
first-order necessary optimality conditions \cite{Clarke1983,Ulbrich2011}. 
In particular, if $\hat{J}(D,\sigma) $ is Fr\'echet differentiable, then 
the optimal pair $(D^*,\sigma^*)$ must satisfy 
$$
\left\langle \nabla \hat{J}(D^*,\sigma^*), (D,\sigma)-(D^*,\sigma^*) \right\rangle \ge 0 , \qquad  (D,\sigma) \in D_{ad}\times \Sigma_{ad} ,
$$
where $ \langle \cdot, \cdot\rangle$ denotes the 
duality product in $L^2$. 

Alternatively, we can use the Pontryagin maximum principle that can be 
proven in our case provided that the gradient of the adjoint variable 
is essentially bounded. However, we can also apply a linearized 
version of the PMP as discussed in \cite{Sumin1989}. In both 
cases, the starting point of the PMP approach is the construction of the Hamiltonian-Pontryagin (HP) function $H:\mathbb{R}^2 \times [D_l,D_r]\times[\sigma_l,\sigma_r]\times \mathbb{R}^4$ that is defined as follows (see \cite{Breit18,Breit19a,Breit19})
\begin{equation}\label{eq:Hamiltonian1}
\begin{aligned}
&H(x,y,D(x,y),\sigma(x),u_1(x,y),u_2(x,y),q_1(x,y),q_2(x,y))\\
 &= \frac{\alpha}{2}\sum_{i=1}^2(\mathcal{H}(x,y,D(x,y),\sigma(x,y),u_i(x,y)) - G_i^{\delta}(x,y))^2+\frac{\xi_1}{2}\sigma^2(x,y)\\
&+ \frac{\xi_2}{2} \left(D(x,y)-\bar{D}(x,y) \right)^2+\gamma  |\sigma(x)|\\
&+\sum_{i=1}^2D(x,y) \nabla u_i(x,y)\cdot \nabla q_i(x,y)+ \sum_{i=1}^2(\sigma(x,y)+\sigma_b) \, u_i(x,y) \, q_i(x,y),
\end{aligned}
\end{equation}
{where $q_i,~ i=1,2$ solve the following adjoint equations, which are derived from the Lagrangian and Hamiltonian framework \cite{Alfio2012}}
\begin{equation}\label{eq:DA_adj}
\begin{aligned}
-\nabla\cdot (D(x,y) \, \nabla q_i(x,y))) + (\sigma(x,y)+\sigma_b) \, q_i(x,y) &= -\alpha\Gamma~(\sigma(x,y)+\sigma_b)  \, [\mathcal{H}(x,y,\sigma,u_i)-G_i^\delta], ~ \mbox{ in } \Omega\\
q_i(x,y) &=  0, \qquad \mbox{ on } \partial\Omega,
\end{aligned}
\end{equation}

Notice that similar consideration, as for the diffusion model, concerning 
existence and regularity of solutions apply to these elliptic problems.

Then subject to appropriate conditions the PMP characterization of a minimum of the minimization problem \eqref{eq:min_problem} is given by the following theorem (expressed 
in the equivalent form of a minimum principle)
\begin{theorem}\label{th:PMP}
Let $(D^*,\sigma^*,u_1^*,u_2^*)$ be a solution to \eqref{eq:min_problem}. Then
\begin{equation}\label{eq:PMP1}
\begin{aligned}
&H(x,y,D^*(x,y),\sigma^*(x,y),u_1^*(x,y),u_2^*(x,y),q_1^*(x,y),q_2^*(x,y))\\
 &=  \mathop{\min_{(v_1,v_2)~\in [D_l,D_r]\times [\sigma_l,\sigma_r]}} H(x,y,v_1,v_2,u_1^*(x,y),u_2^*(x,y),q_1^*(x,y),q_2^*(x,y)),
\end{aligned}
\end{equation}
for almost all $(x,y) \in \Omega$, and $u_i^*,~i=1,2$ and $q_i^*,~i=1,2$ solve \eqref{eq:DA} and \eqref{eq:DA_adj}, respectively, with $D=D^*$ and $\sigma=\sigma^*$. 
\end{theorem}
 We can rewrite condition \eqref{eq:PMP1} as follows
\begin{equation}\label{eq:PMP}
(D^*(x,y),\sigma^*(x,y)) =  \mathop{\argmin_{(v_1,v_2)~\in [D_l,D_r]\times [\sigma_l,\sigma_r]}} H(x,y,v_1,v_2,u_1^*(x,y),u_2^*(x,y),q_1^*(x,y),q_2^*(x,y)) .
\end{equation}
We note that \eqref{eq:PMP} represents a pointwise optimality criteria at each $(x,y)$ and does not involve any derivatives with respect to the optimization variables $D$ 
and $\sigma$. In the next section, we discuss an efficient and robust numerical scheme 
based on \eqref{eq:PMP} that solves \eqref{eq:min_problem} (resp. \eqref{extraMin}) to obtain $D^*$ and $\sigma^*$.

\section{The SQH algorithm for solving the QPAT inverse problem}\label{sec:numerical_scheme}

The sequential quadratic hamiltonian (SQH) algorithm 
has been proposed in \cite{Breit18,Breit19}, as a new 
variant of the successive approximations methods; see \cite{Chernousko2007} 
for an early review. It includes an adaptive 
quadratic penalty of the updates that results in an 
augmented Hamiltonian as first proposed in \cite{SakawaShindo1980}. 
In this way, an efficient and robust iterative procedure is obtained that is able 
to solve nonsmooth optimization problems governed by PDEs. 

In our SQH implementation, the augmented HP function is given by  
\begin{equation}\label{eq:augmented_Hamiltonian}
H_\epsilon(x,y,D,\sigma ,\tilde{D},\tilde{\sigma},u_1 ,u_2 ,q_1 ,q_2 )=H(x,y,D ,\sigma ,u_1 ,u_2 ,q_1 ,q_2 ) + \epsilon \, [(D -\tilde{D} )^2+(\sigma -\tilde{\sigma} )^2] , 
\end{equation}
{where $\epsilon > 0$ is a penalization parameter that is chosen 
adaptively at each step of the SQH iteration process: a larger value of $\epsilon$ is chosen if a sufficient decrease of the functional $J$ is not observed while a smaller value of $\epsilon$ is chosen if $J$ decreases sufficiently.}  Specifically, if $\tilde{D}$ and 
$\tilde{\sigma}$ denote a previous approximation to the diffusion and 
absorption coefficients sought, then the purpose of quadratic term $\epsilon \,[ (D-\tilde{D})^2+(\sigma-\tilde{\sigma})^2] $ is to have the pointwise minimizer of 
$H_\epsilon$, and thus an update to $D$ and $\sigma$, that is 
close to the previous values $\tilde{D}$ and 
$\tilde{\sigma}$ as much as $\epsilon$ is large. We also remark that during a step of the optimization problem, that is, for a minimization 
sweep on all grid points of the $(x,y)$ mesh and, during this process, the values of $u$ and $q$ are those obtained in the previous iteration. The SQH algorithm is given as a pseudocode as follows. 

\begin{algorithm}[SQH algorithm]\label{eq:SQH}\
\begin{enumerate}

\item  Choose $\epsilon > 0$, $\kappa >0$, $\lambda > 1$, $\zeta \in (0,1)$, $\rho \in (0,\infty)$, $(D^0,\sigma^0 )\in D_{ad}\times\Sigma_{ad}$

\item Compute $u_i^0(D^0,\sigma^0)$ and $q_i^0(D^0,\sigma^0,u_i^0)$ for $i=1,2$. Set $k = 0$.
\item Find $(D,\sigma )\in D_{ad}\times\Sigma_{ad}$ such that 
\[
H_\epsilon(x,y,D,\sigma,D^k,\sigma^k,u_1^k,u_2^k,q_1^k,q_2^k) = \mathop{\min_{w \in [D_l,D_r]}}_{z \in [\sigma_l,\sigma_r]} H_\epsilon(x,y,w,z,D^k,\sigma^k,u_1^k,u_2^k,q_1^k,q_2^k)
\]
for all $(x,y)\in \Omega.$
\item Calculate $u_i(D,\sigma)~, i=1,2$, and set $\tau = \|D-D^k\|^2_{L^2(\Omega)}+\|\sigma-\sigma^k\|^2_{L^2(\Omega)}.$
\item If $J(D,\sigma, u_1,u_2) - J(D^k,\sigma^k, u_1^k, u_2^k) > -\rho \, \tau$, choose $\epsilon = \lambda \, \epsilon$,\\
else choose $\epsilon = \zeta \, \epsilon$, $u_i^{k+1} = u_i$, $D^{k+1} = D$, $\sigma^{k+1} = \sigma$, and calculate the 
adjoint variables $q_i^{k+1}$, $i=1,2$, corresponding to $D^{k+1}$, $\sigma^{k+1}$, $ u_i^{k+1}$, $i=1,2$.

\item Set $k = k+1$
\item If $\tau < \kappa$, STOP and return optimal $(D^k,\sigma^k)$ else go to Step 3.   
\item end
\end{enumerate}
\end{algorithm}
{The pointwise minimization in Step 3 of this algorithm is carried out by a direct search process, i.e., we divide the domain $[D_l,D_r] \times [\sigma_l,\sigma_r]$ into a two-dimensional discrete grid, evaluate $H_\epsilon$ on this grid, and then choose a grid point $(w,z)$ at which $H_\epsilon$ is minimum.} We remark that in Step 5 of this algorithm, if the inequality $J(D,\sigma, u_1,u_2) - J(D^k,\sigma^k, u_1^k, u_2^k) > -\rho \, \tau$ is true this means that no 
sufficient decrease of the value of the objective functional $J$ as been achieved. 
In this case, a larger value of $\epsilon$ is taken (since $\lambda > 1$) and the optimization procedure in Step 3. is repeated with the corresponding 
new augmented HP function. On the other hand, if the inequality above is 
false, the decrease has met the required criteria for reduction of the value of 
$J$, and the new values $D^{k+1} = D$, $\sigma^{k+1} = \sigma$ are 
obtained. Correspondingly, the updates $u_i^{k+1}$ and $q_i^{k+1} $, 
$i=1,2$, are computed. In this case, the value of $\epsilon$ is reduced 
by a factor $\zeta < 1$. 

The fundamental result required to guarantee wellposedness of the SQH 
algorithm applied to our QPAT inverse problem is to prove that, if not already 
at the optimal point, it is always possible 
to find a value of $\epsilon >0$ such that the minimization of the 
corresponding $H_\epsilon$ results in new values of $D$ and $\sigma$ 
for which the value of the functional $J$ is reduced. We present this result 
in Theorem \ref{sqh2} below and refer to \cite{Breit18,Breit19a} 
for the use of this theorem for further results on the convergence of the SQH 
algorithm in a finite number of steps.

\begin{theorem}\label{sqh2}
Let $(u_1,u_2,D,\sigma)$ and $(u_1^k,u_2^k,D^k,\sigma^k)$ be as in Algorithm \ref{eq:SQH}, $k\in\mathbb{N}_0$, and denote $\delta D:=D-D^k,~\delta \sigma:=\sigma-\sigma^k$, $\delta u_i:=u_i-u_i^k$, $ i=1,2$. If the assumptions of Lemma \ref{sqh20} are satisfied, then 
there is a $\theta >0$, independent of $\epsilon$, $k$, and 
$D^k$, $\sigma^k$, such that for the $\epsilon$ currently 
chosen by the SQH Algorithm \ref{eq:SQH}, it holds that 
$$
J(D,\sigma,u_1,u_2)-J(D^k,\sigma^k,u_1^k,u_2^k) 
< - \left( \epsilon- \theta \right) \,  
\big(  \| D - D^k \|_{L^2(\Omega)}^2 + \| \sigma - \sigma^k  \|_{L^2(\Omega)}^2 \big). 
$$ 
\end{theorem}

\begin{proof} See the Appendix. \end{proof}

\bigskip


For implementing Algorithm \ref{eq:SQH}, we need to solve the photon propagation equation \eqref{eq:DA} for different boundary conditions $g_i,~i=1,2$, and the corresponding adjoint equations \eqref{eq:DA_adj}. We use a cell-nodal finite-difference scheme, proposed in \cite{ascher_haber}, to discretize \eqref{eq:DA} and \eqref{eq:DA_adj}. To illustrate this scheme, we consider the generic form of the equations \eqref{eq:DA} and \eqref{eq:DA_adj} defined in the domain $\Omega=(a,b)^2 \subset\mathbb{R}^2$ as follows  
\begin{equation}\label{eq:elliptic}
\begin{aligned}
-\nabla\cdot({D(x,y)}\nabla{u(x,y)})+(\sigma(x,y)+\sigma_b) \, u(x,y)&=f & \qquad \mbox{ in }\Omega, \\
u(x,y)&=f_D(x,y)  & \qquad \mbox{ on }\partial\Omega ,\\
\end{aligned}
\end{equation}
where $f$ and $f_D$ are assumed continuous in $\Omega$ and on $\partial\Omega$, respectively. Consider a sequence of uniform grids 
$\lbrace\Omega_h\rbrace_{h>0}$ given by 
\begin{equation}\label{eq:domain}
\Omega_h = \lbrace{(x_1^i,x_2^j)\in\mathbb{R}^2:(x_1^i,x_2^j) = (a+ih,a+jh),~(i,j)\in
\lbrace{0,\hdots,N}\rbrace^2}\rbrace\cap\Omega,
\end{equation}
where $N$ represents the number of cells in each direction and $h = \dfrac{(b-a)}{N}$ is the mesh size. Then the corresponding cell-nodal discretization places the unknowns $u,D,\sigma$ at the nodal points $(x^i,y^j)$ of the grid that results in the following approximation of \eqref{eq:elliptic} at  $(x^i,y^j)$
\begin{equation}\label{eq:FD_scheme}
\begin{aligned}
&\dfrac{1}{h^2} \Bigg\lbrace({D_{i+1/2,j}}+{D_{i-1/2,j}}+{D_{i,j+1/2}}+{D_{i,j-1/2}})u_{i,j}\\
& -{D_{i+1/2,j}}u_{i+1,j}- {D_{i-1/2,j}}u_{i-1,j}- {D_{i,j+1/2}}u_{i,j+1}- {D_{i,j-1/2}}u_{i,j-1}  \Bigg\rbrace\\
&+(\sigma_{i,j}+\sigma_b)u_{i,j}=f_{i,j},\qquad 1 \leq i,j \leq N-1 .\\
\end{aligned}
\end{equation}
We denote $D_{i \pm 1,j}=D(x^i \pm h , y^j)$, 
$D_{i ,j \pm 1}=D(x^j , y^j \pm h )$, and 
$f_{i,j}=f(x^i,y^j)$. The required intermediate values 
of $D$ are computed as follows  
$$
{D_{i \pm 1/2,j}} = \dfrac{1}{2}\Bigg({D_{i\pm 1,j}}+{D_{i,j}} \Bigg) \quad \mbox{ and } \quad 
{D_{i ,j \pm 1/2}} = \dfrac{1}{2}\Bigg({D_{i,j \pm 1}}+{D_{i,j}} \Bigg) .
$$
The Dirichlet boundary data $f_D$ is included in the standard way in the right-hand side of the algebraic equation. We use one-sided finite difference discretizations to evaluate the right-hand sides of the adjoint equations. With this discretization scheme in \eqref{eq:FD_scheme}, we solve a matrix-vector equation to obtain the unknowns $u_{i,j}$.

\section{Numerical results}\label{sec:results}
We present results of numerical experiments to validate the ability of our 
QPAT optimization framework in the reconstruction of the diffusion coefficient $D$ 
and the absorption coefficient $\sigma$ profiles with high accuracy. These also demonstrate the efficiency and robustness of our 
SQH-QPAT methodology. We consider a widely varying set of phantoms for $D,\sigma$. In Test Case 1, we consider a disk phantom with different intensities for $D,\sigma$ and investigate reconstruction with the different regularization terms considered in our cost functional $J$. Next, we discuss two experiments with the heart lung phantom: one in which the relation \eqref{eq:D_rel} is satisfied between ${D}$ and $\sigma$, and the other in which \eqref{eq:D_rel} is not satisfied, making it a very generic set of phantoms for $D,\sigma$. We demonstrate that in both cases, by incorporating the aprior assumption \eqref{eq:D_rel} in the functional $J$, we can obtain high quality reconstructions. 
We conclude our set of experiments with the more challenging set of two Shepp-Logan phantoms:  the standard one, and another augmented with tumor structures. In the test cases with the heart and lung phantom and the Shepp-Logan phantom,  we also implement our method on data containing additive Gaussian noise in order to demonstrate the robustness of our framework.

{For the experiments below, we first remark that the units of the spatial variable $x$ is $cm$ and of the absorption coefficient is $cm^{-1}$. Starting with an experimental square domain of side length 5 cm,  we first scale and translate the diffusion equation \eqref{eq:DA}, such that the domain $\Omega = (-1,1)\times(-1,1)$, is discretized into 100 equally spaced points in both the coordinate directions. The corresponding scaling is 1 experimental unit corresponds to 2.5 cm. Thus, based on Table 1, an experimental range of 0-1 for $\sigma_a$ corresponds to 0-0.4 $cm^{-1}$. }The two boundary radiation functions are $g_1(x,y) = e^x,~ g_2(x,y) = e^y.$ {The weights of the functional $J$ given in \eqref{eq:cost_functional} are chosen as $\alpha=1,\xi_1=0.01,\xi_2=20,\gamma=0.01$ and $\Gamma = 1.0$. }The large value of $\xi_2$ is chosen to emphasize the need to satisfy the relation between the diffusion and the absorption coefficient, given by \eqref{eq:D_rel}, in the numerical experiments. {The value of $c$ in the Kubelka-Munk relation \eqref{eq:D_rel} is chosen to be 100/3, though as we have mentioned earlier, this choice is not essential. }{The initial value of $\epsilon = 10$} and the value of the stopping criteria for convergence is chosen as $\kappa = 10^{-6}$. To generate the data $G_i^\delta,~ i=1,2$, we first solve for $u_i$ in (\ref{eq:DA}) with given test values of $D,\sigma$ and boundary illumination data $g_i$ on a finer mesh with $N=400$ using the cell-nodal finite difference model described in Section \ref{sec:numerical_scheme}. We next compute $G_i^\delta$ on the finer mesh using the values of $\sigma,\sigma_b,u_i$ from \eqref{eq:optical_energy}. We finally restrict the resulting $G_i^\delta$ onto the coarser mesh with $N=100$ and take this as our given data. 

{To quantify the comparisons of the reconstructions obtained in terms of their contrast (function values) and resolution (discontinuities or edges), with and without noise in the data, we use the following quantitative figures of merit \cite{Kinahan1994}
\begin{equation}
\begin{aligned}
 &\mbox{Relative mean square error (RMSE) percentage}= \frac{\|f_{\mbox{rec}}-f_{\mbox{ex}}\|_2}{\|f_{\mbox{ex}}\|_2}*100 \%,\\
&\mbox{Peak signal-to-noise ratio (PSNR)} = 10\log_{10}\left(\frac{\max{f_{\mbox{ex}}}}{\|f_{\mbox{rec}}-f_{\mbox{ex}}\|^2_2}\right),\\ 
\end{aligned}
\end{equation}
where $f_{ex}$ is the exact phantom and $f_{rec}$ is the reconstructed phantom.}

\vspace{0.1cm}

\textbf{Test Case 1:}  we consider a phantom represented by a disk centered at $(0.25,0.25)$ and having radius 0.25. The value of $\sigma$ inside the disk is 1 and outside is 0. The background value $\sigma_b$ is chosen to be 0.16. The corresponding value of D is chosen as 0.003 inside the disk and 0.02 outside. The plots of the actual phantoms for $D$ and {the total absorption coefficient $\sigma_a$} are shown in Figure \ref{disk_exact}.

\tred{\begin{figure}[H]
\centering
\subfloat[Exact $D$]{\includegraphics[width=0.4\textwidth]{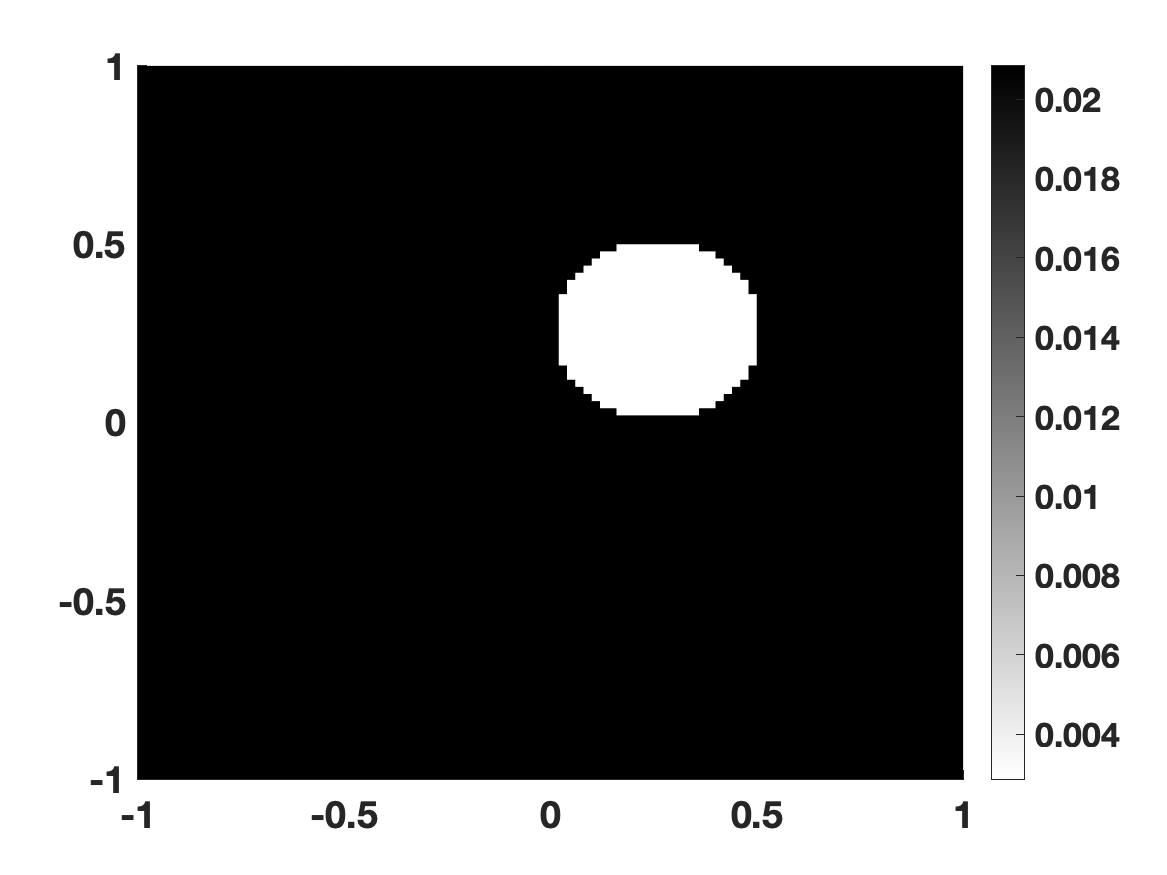}\label{D_disk_actual}}\hspace{10mm}
\subfloat[Exact $\sigma_a$]{\includegraphics[width=0.4\textwidth]{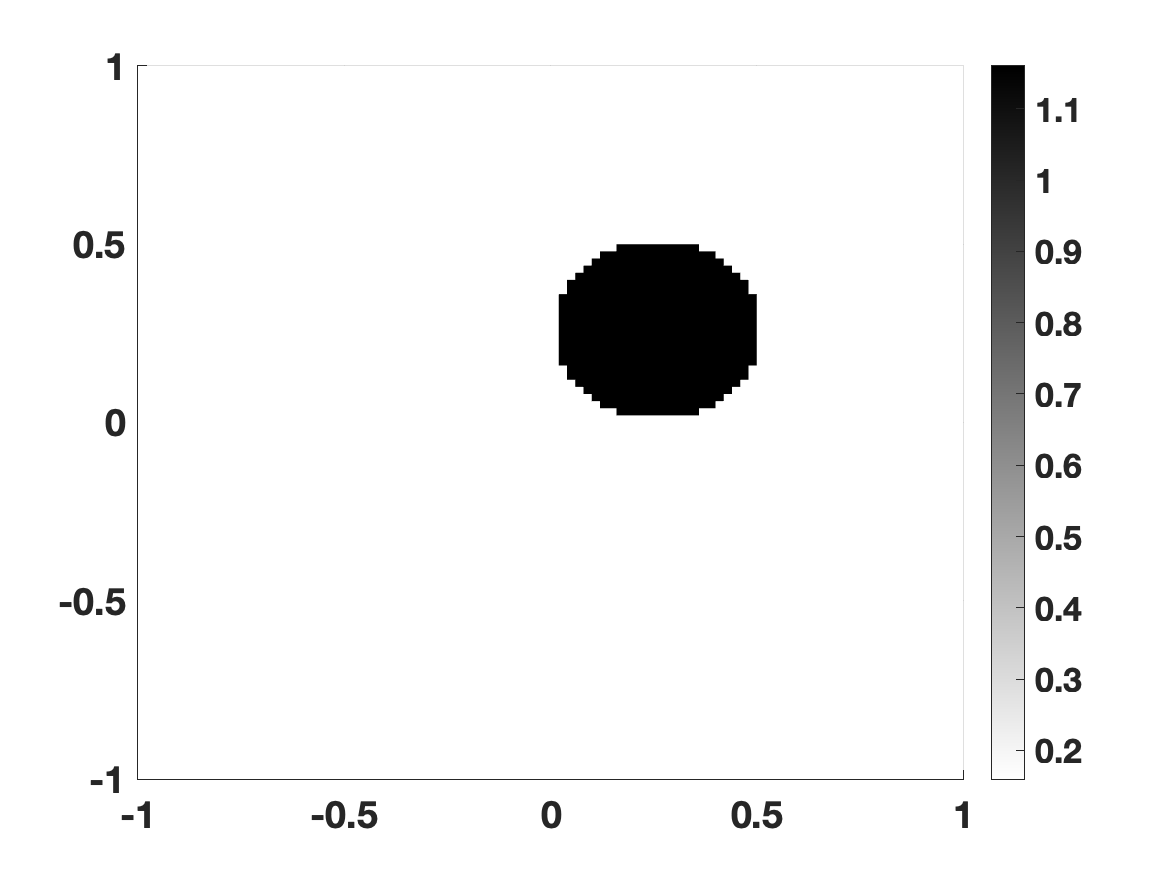}\label{sigma_disk_actual}}\\
\caption{Test Case 1: Exact $D$ and $\sigma$ represented by disk phantoms. }
    \label{disk_exact}
  \end{figure}
  }
  
In Figure \ref{disk1}, we present the reconstruction of $D$ and $\sigma_a$ using various choices of the regularization terms in $J$. {The first column 
of images depicts reconstructions of $D$ and $\sigma_a$ by setting the weight $\xi_2=0$ of the regularization term 
$ \left\|D-\bar{D}\right\|_2^2$, which is the term that measures the deviation of $D$ from the Kubelka-Munk approximation $\bar{D}$, given in \eqref{eq:D_rel}}. We see that the reconstruction of $D$ has a very poor resolution and contrast and in turn, this causes loss of contrast in the reconstruction of $\sigma_a$. {In the second column, we use our regularization term $ \left\|D-\bar{D}\right\|_2^2$ but setting the weight $\gamma=0$ of the sparsity that promotes the $L^1$ regularization of $\sigma$.} We now see a reconstructed $\sigma_a$ with a poor contrast with respect to the background $\sigma_b$. The reconstructed $D$ is now far more superior than in the previous cases, but there is still a significant loss of contrast.

\begin{figure}[H]
\centering
\subfloat[Reconstruction of $D$ with $\alpha=1,\xi_1=0.01,\xi_2=0,\gamma=0.01$]{\includegraphics[width=0.35\textwidth]{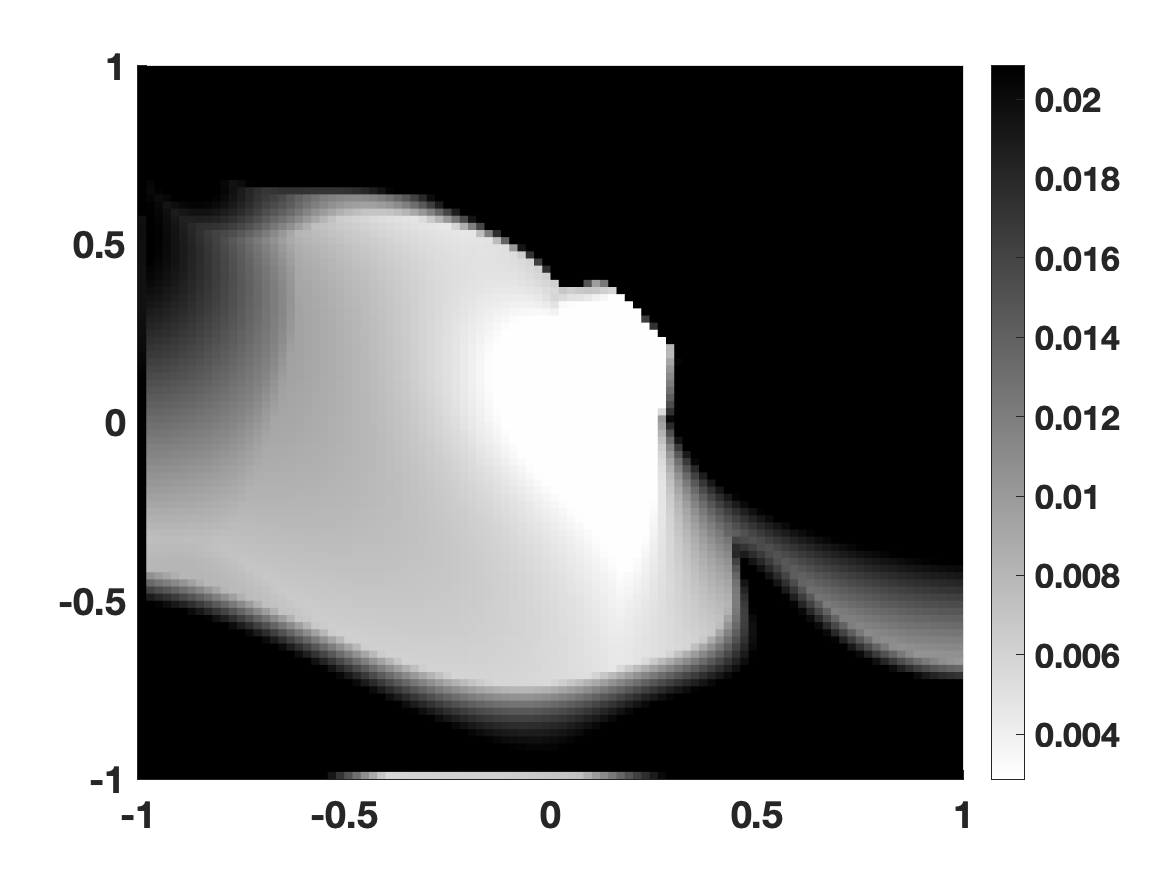}\label{D_disk_wo_rel_recon}}
\subfloat[Reconstruction of $D$ with $\alpha=1,\xi_1=0.01,\xi_2=20,\gamma=0$]{\includegraphics[width=0.35\textwidth]{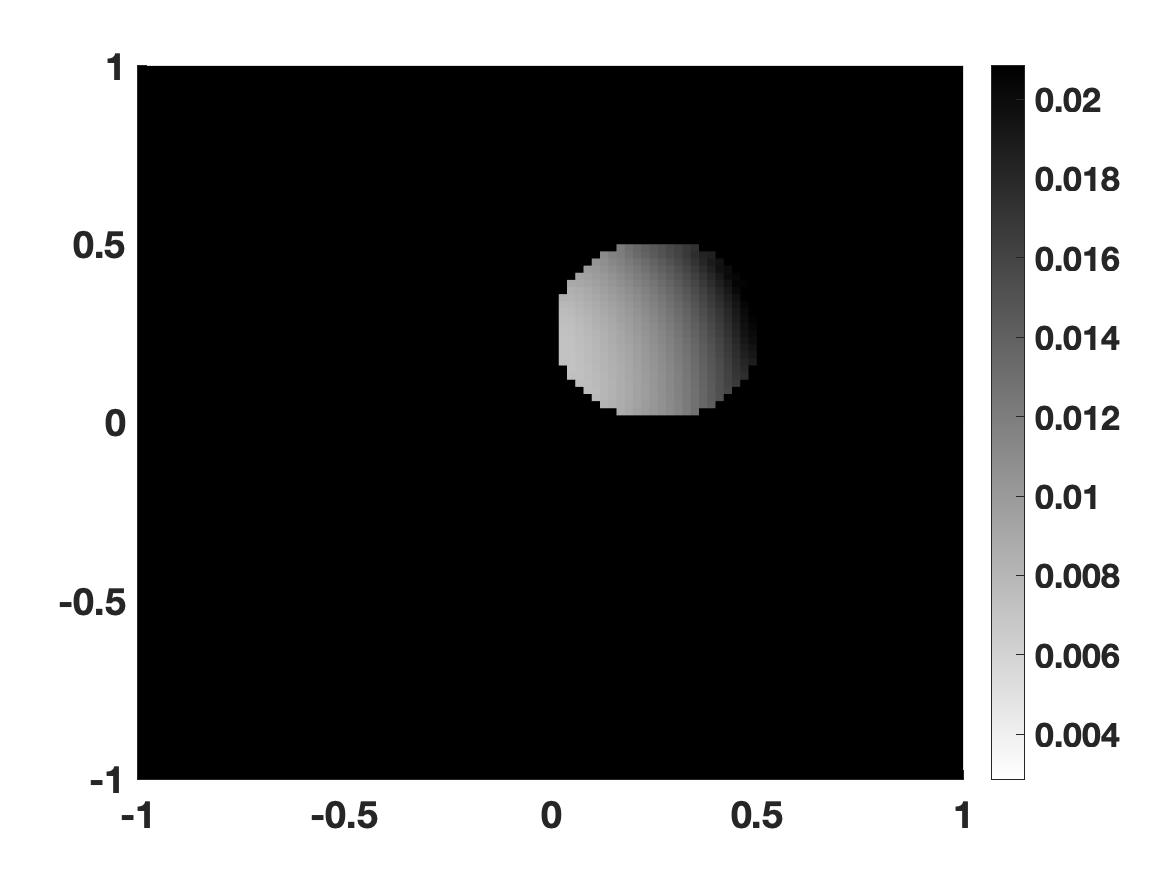}\label{D_disk_wo_L1_recon}}\\
\subfloat[Reconstruction of $\sigma_a$ with $\alpha=1,\xi_1=0.01,\xi_2=0,\gamma=0.01$]{\includegraphics[width=0.35\textwidth]{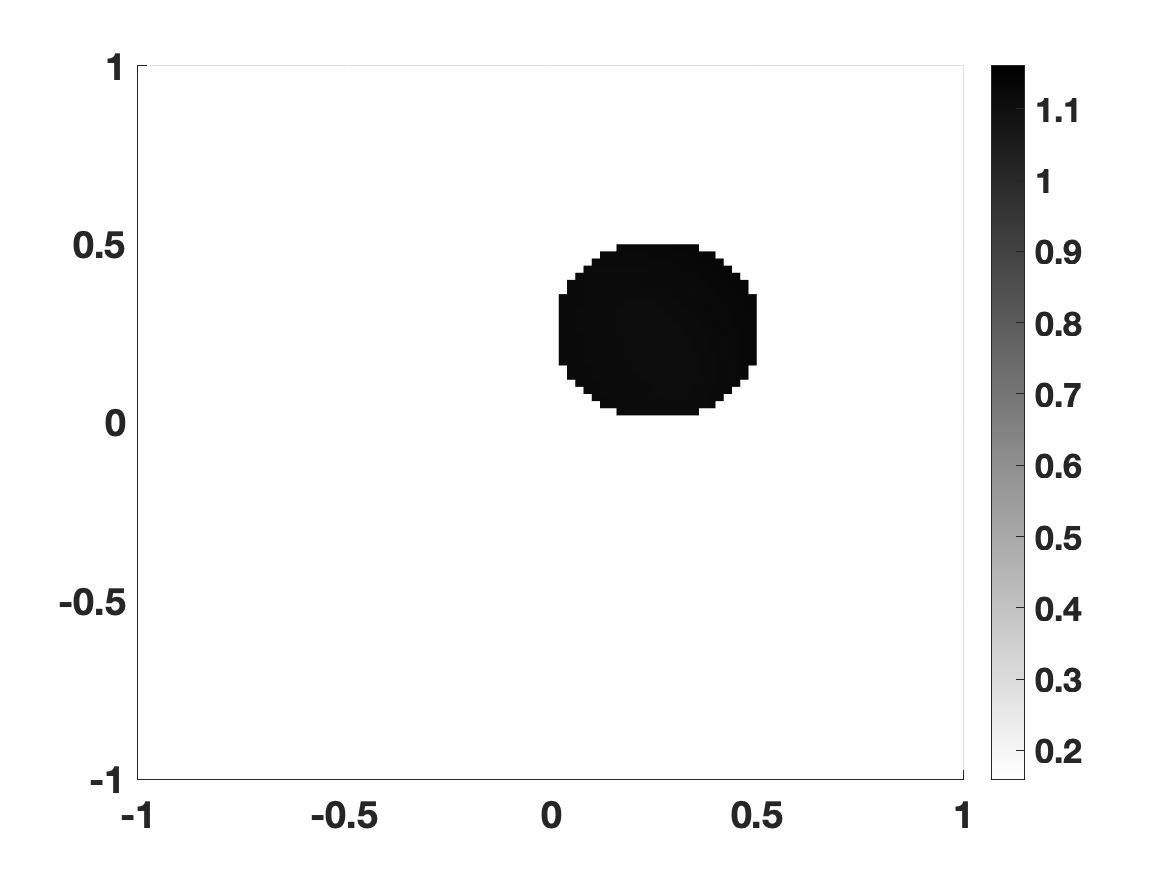}\label{sigma_disk_wo_rel_recon}}
\subfloat[Reconstruction of $\sigma_a$ with $\alpha=1,\xi_1=0.01,\xi_2=20,\gamma=0$]{\includegraphics[width=0.35\textwidth]{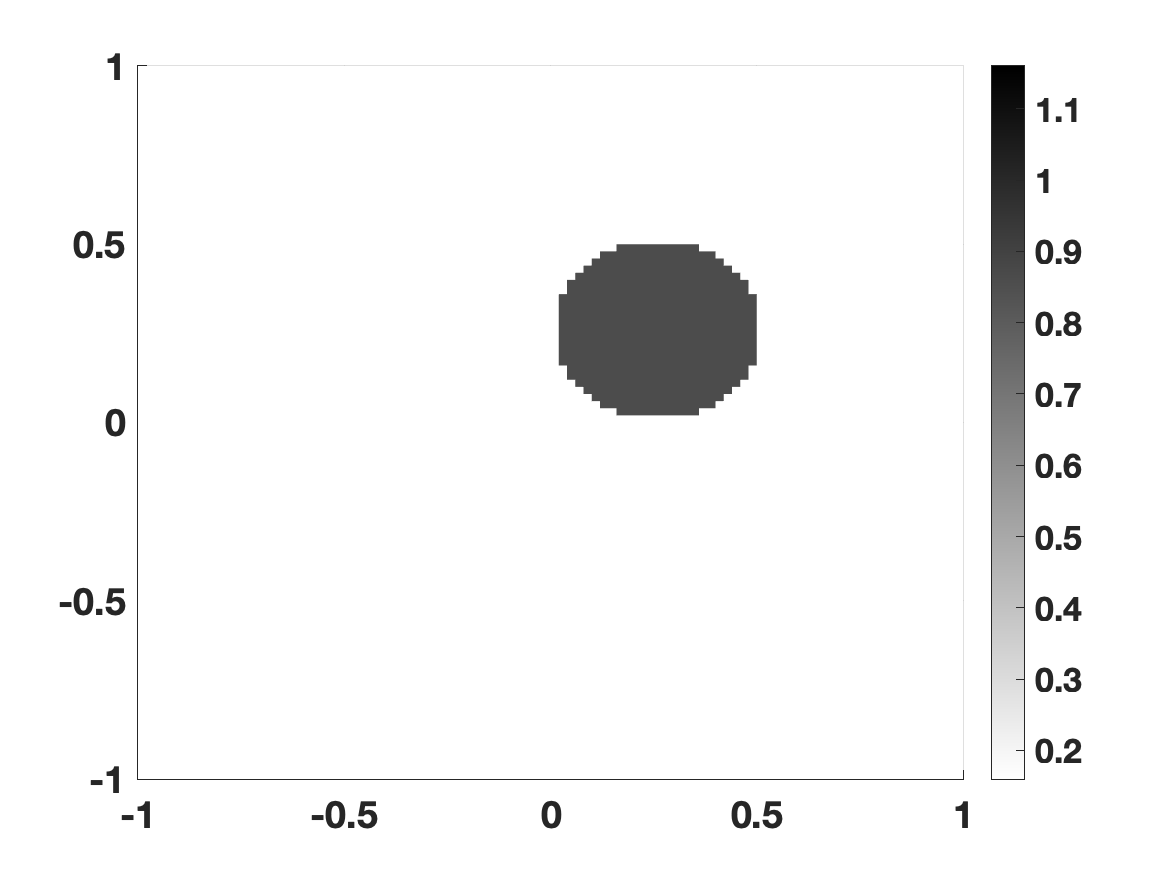}\label{sigma_disk_wo_L1_recon}}\\

\caption{Test Case 1: The reconstructed disk phantoms
    with different choices of the values of the regularization weights. }
    \label{disk1}
  \end{figure}
  
{Next, we show that in order to enhance contrast and resolution in the reconstruction, we need to have the $L^2$ regularization term $ \left\|D-\bar{D}\right\|_2^2$ for $D$ and the sparsity promoting $L^1$ regularization term for $\sigma$. In this setting, Figure \ref{disk2} shows the reconstructions of $D$ and $\sigma_a$ with our framework using the SQH algorithm. We see that the reconstructed functions are of high resolution and high contrast when compared to the actual phantoms in Figure \ref{disk_exact}.  We also compute the RMSE percentages and PSNR values for the reconstructions in Table \ref{table:disk}. 
We observe that without the Kubelka-Munk regularization term for $D$ and without the $L^1$ regularization term for $\sigma$, we have very high values of RMSE \% and low values of PSNR values. With those regularization terms, using the SQH algorithm, we obtain RMSE \% of around 1 and high PSNR values, reaffirming the superiority of the reconstructions.}
  
\begin{figure}[H]
\centering
\subfloat[Reconstruction of $D$ with $\alpha=1,\xi_1=0.01,\xi_2=20,\gamma=0.01$]{\includegraphics[width=0.35\textwidth]{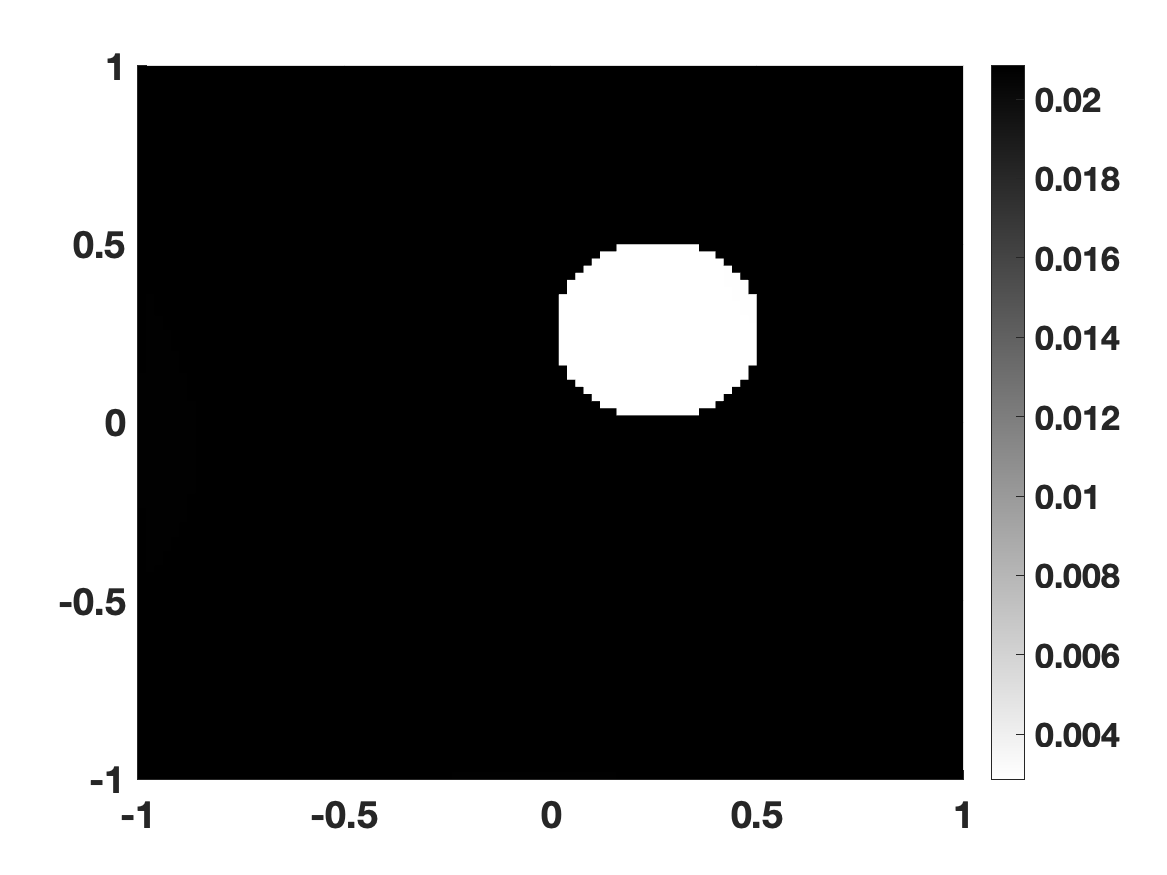}\label{D_disk_recon}}\hspace{10mm}
\subfloat[Reconstruction of $\sigma_a$ with $\alpha=1,\xi_1=0.01,\xi_2=20,\gamma=0.01$]{\includegraphics[width=0.35\textwidth]{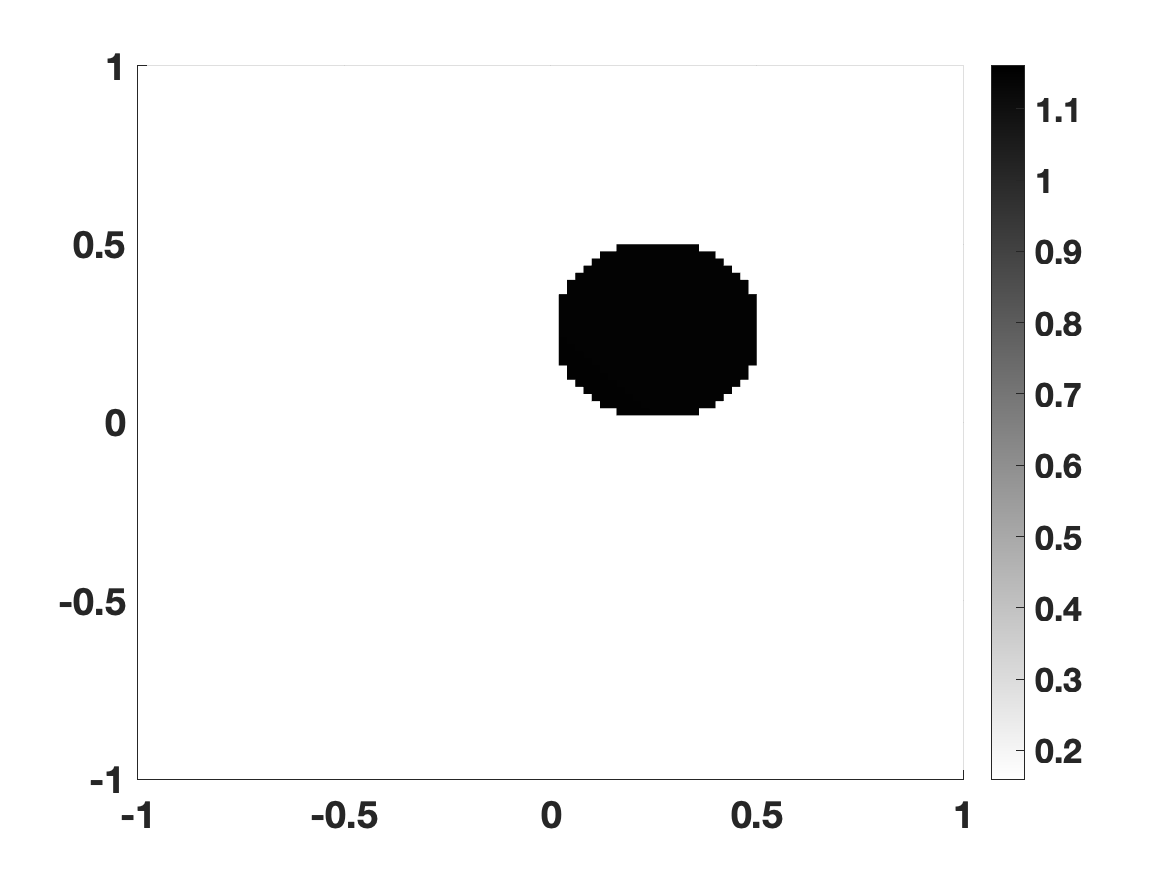}\label{sigma_disk_recon}}\\
    \caption{Test Case 1: The reconstructed disk phantoms using the weights $\alpha=1,\xi_1=0.01,\xi_2=20,\gamma=0.01$. }
    \label{disk2}
  \end{figure}
  
\begin{table}[H]
\centering
\begin{tabular}{|c|c|c|c|c|}
\hline
Method & RMSE \% ($D$) & RMSE \% ($\sigma_a$) & PSNR ($D$) & PSNR ($\sigma_a$) \\ [0.5ex]
\hline
Without Kubelka-Munk ($\xi_2=0$) & 225.7 &9.24 & 25.09 & 32.37\\
Without $L^1$ ($\gamma=0$)& 218 &30.47 & 27.06 & 23.51\\
\pbox{60mm}{With Kubelka-Munk\\ $(\xi_2 = 20) $ and $L^1$ ($\gamma=0.01)$} & 0.08 &1.41 & 94.33 & 50.29\\
[1ex]
\hline
\end{tabular}
\caption{Test Case 1: RMSE \% and PSNR values for reconstructions of the disk phantom with different regularization weights}
\label{table:disk}
\end{table}
 
\textbf{Test Case 2:} we demonstrate the robustness of our framework for reconstructing multiple objects. In this test case, we consider a heart lung phantom for $\sigma$. The underlying value of the phantom is 0 that is perturbed into two ellipses that represent the lungs with value 1, and a disk representing the heart with value 0.5. The value of $\sigma_b$ is chosen to be 0.03. The value of $D$ is chosen as 0.003 inside the ellipses, 0.006 inside the disk and 0.1 elsewhere. The plots of the exact and the reconstructed phantoms are shown in Figure \ref{heart_lung_1}.
 
\begin{figure}[H]
\centering
\subfloat[Exact $D$]{\includegraphics[width=0.35\textwidth]{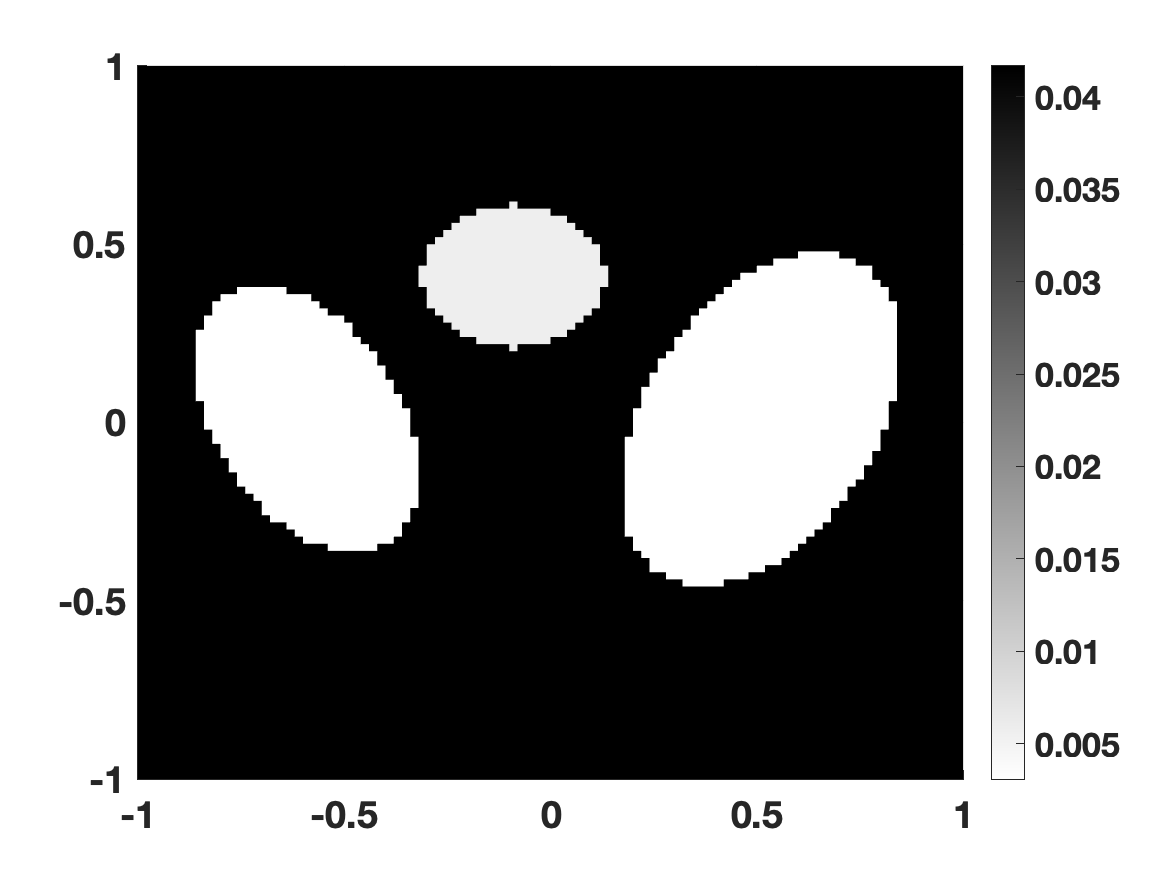}\label{D_heart_lung_actual}}
\subfloat[Exact $\sigma_a$]{\includegraphics[width=0.35\textwidth]{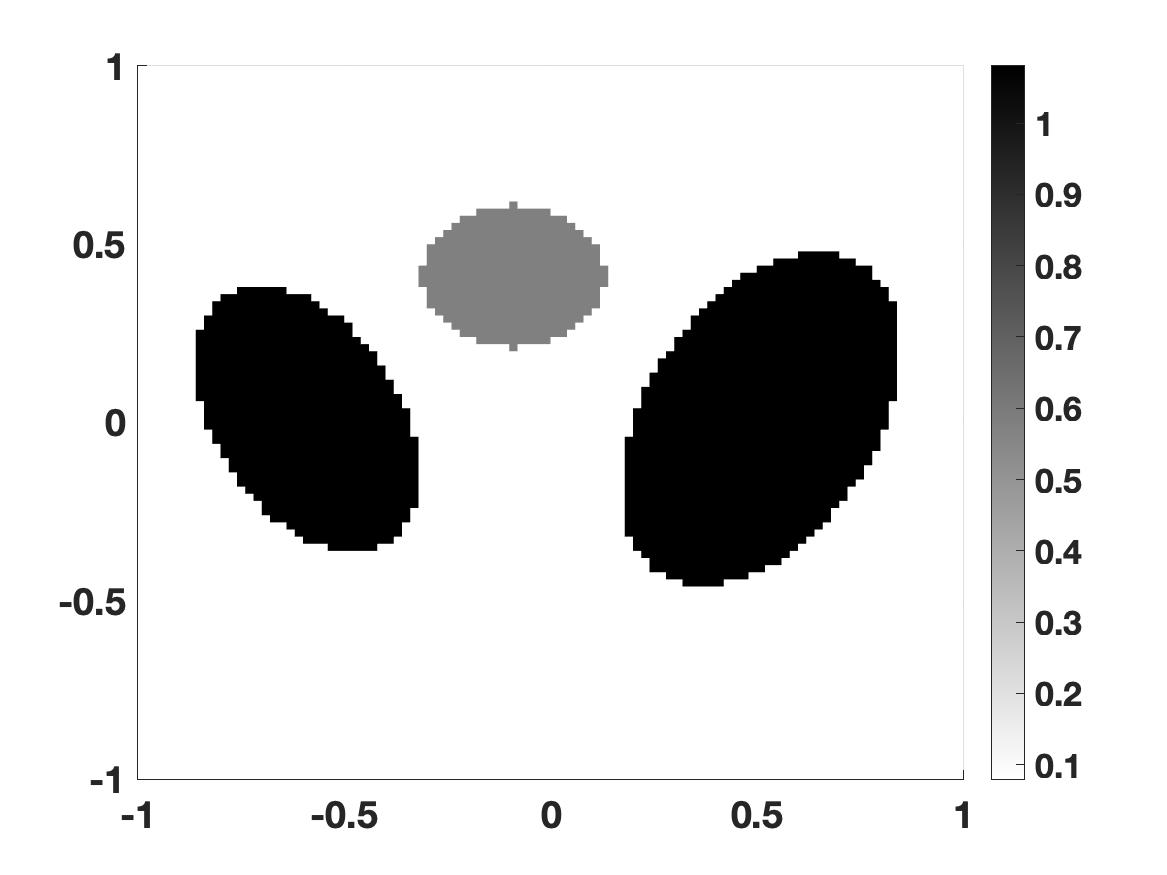}\label{sigma_heart_lung_actual}}\\

\subfloat[Reconstructed $D$]{\includegraphics[width=0.35\textwidth]{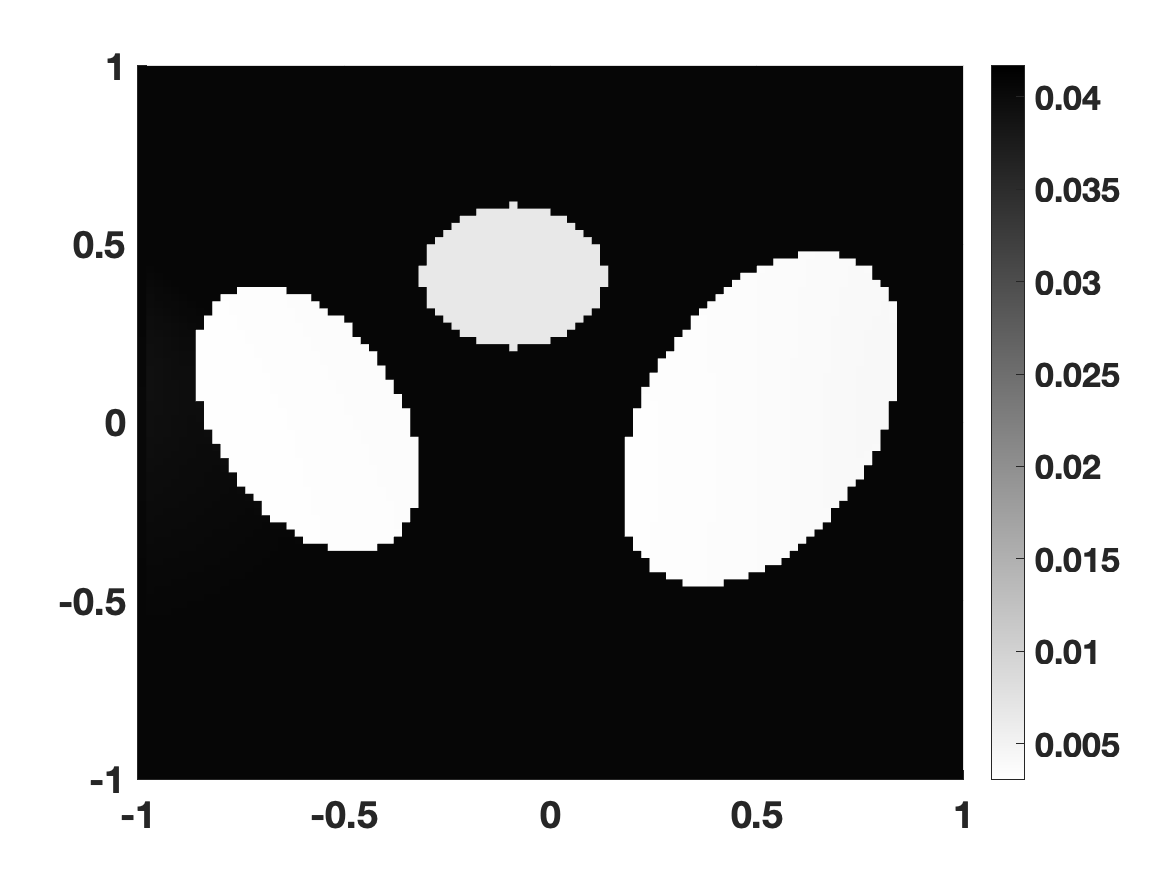}\label{D_heart_lung_recon}}
\subfloat[Reconstructed $D$ with 5\% noise]{\includegraphics[width=0.35\textwidth]{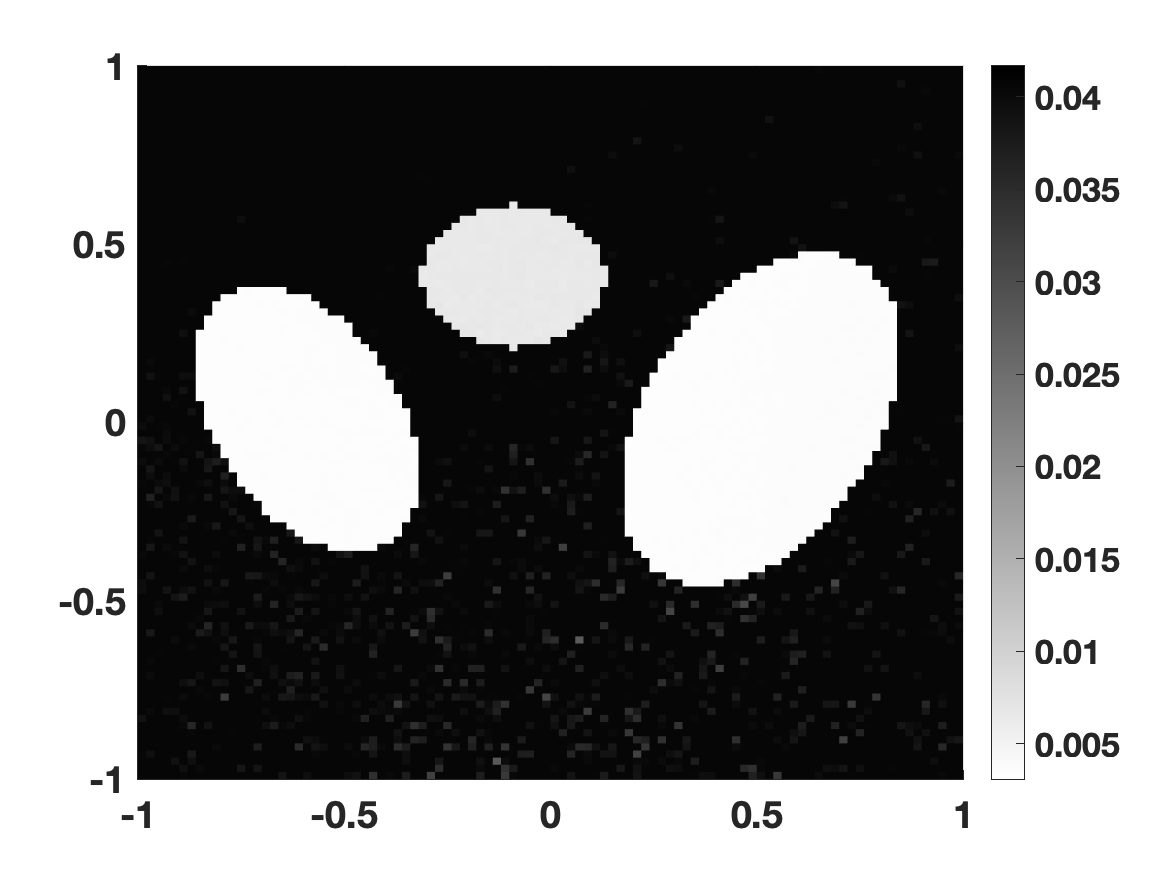}\label{D_heart_lung_recon_5}}
\subfloat[Reconstructed $D$ with 10\% noise]{\includegraphics[width=0.35\textwidth]{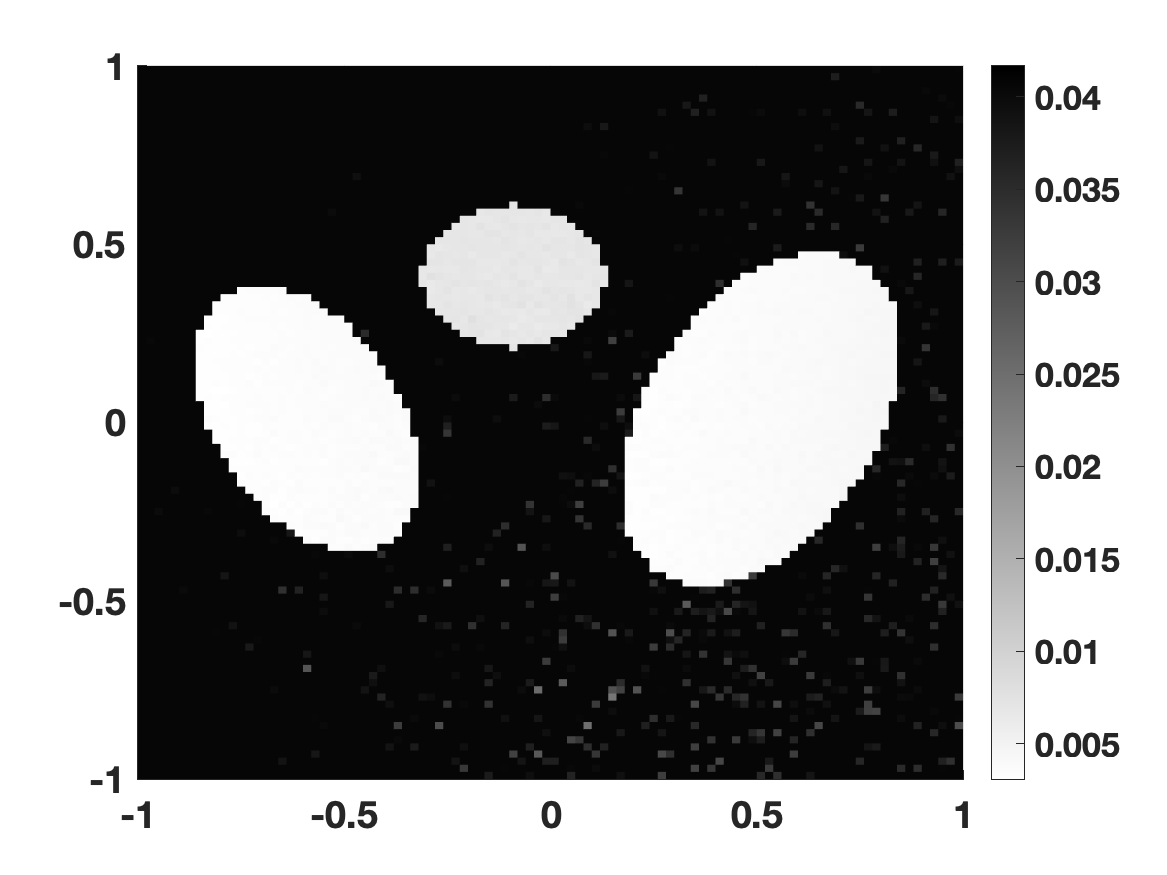}\label{D_heart_lung_recon_10}}\\

\subfloat[Reconstructed $\sigma_a$]{\includegraphics[width=0.35\textwidth]{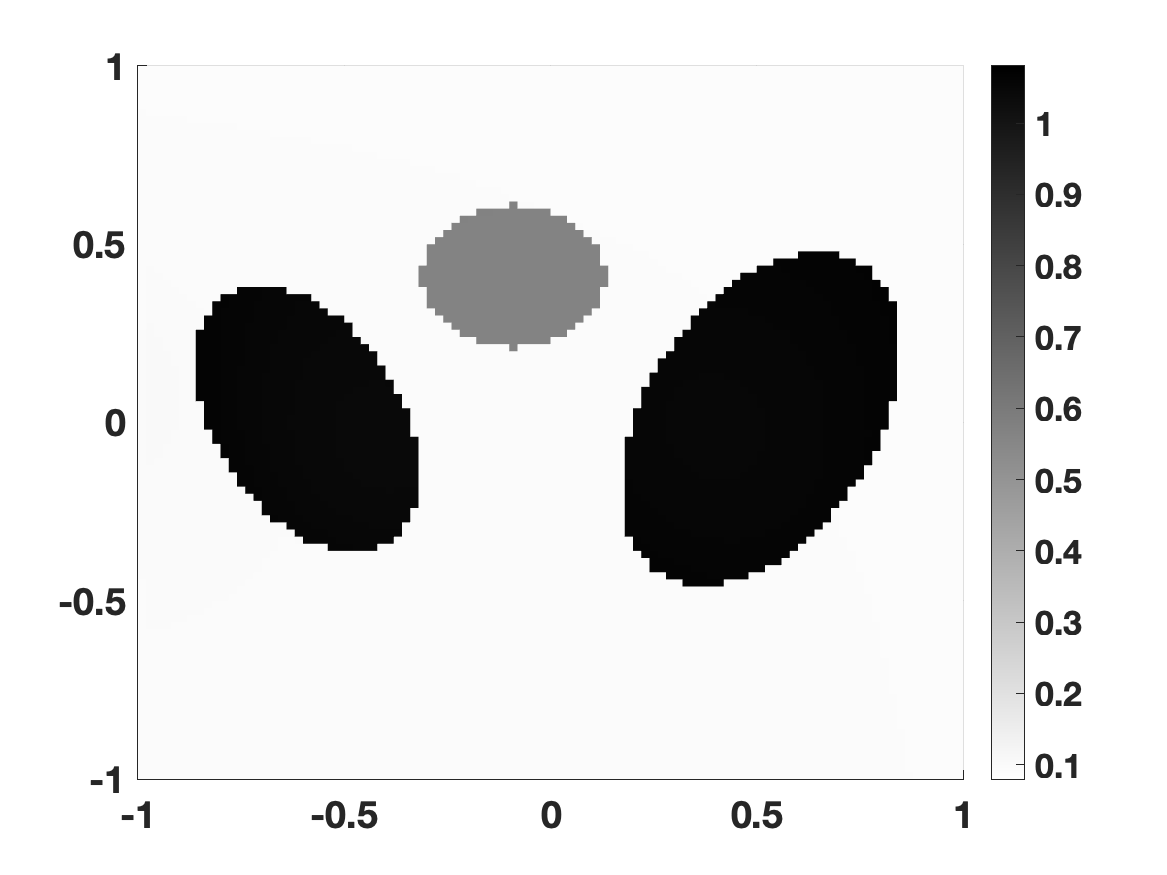}\label{sigma_heart_lung_recon}}
\subfloat[Reconstructed $\sigma_a$ with 5\% noise]{\includegraphics[width=0.35\textwidth]{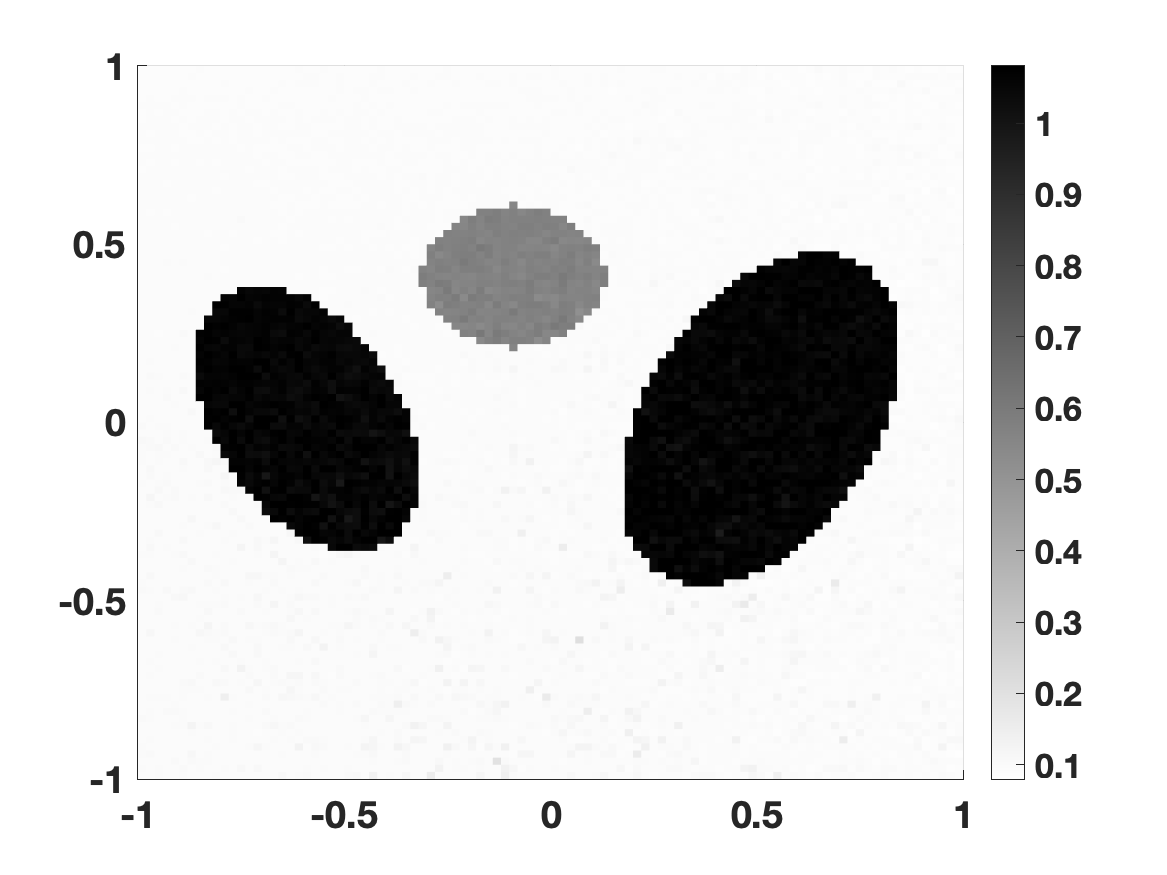}\label{sigma_heart_5}}
\subfloat[Reconstructed $\sigma_a$ with 10\% noise]{\includegraphics[width=0.35\textwidth]{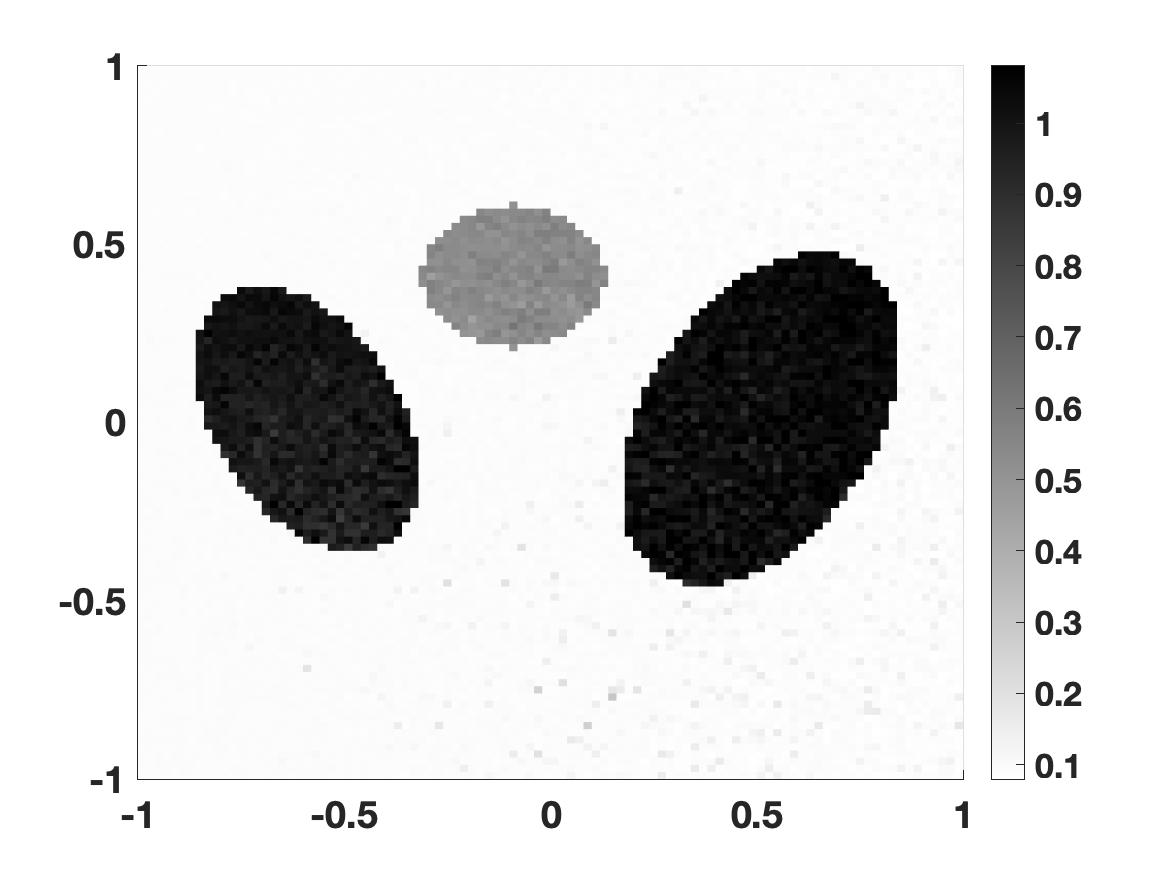}\label{sigma_heart_10}}\\

\caption{Test Case 2: The actual and the reconstructed phantoms }
    \label{heart_lung_1}
  \end{figure}
{We see from the reconstructed images of $D$ and $\sigma_a$ in Figures \ref{D_heart_lung_recon} and \ref{sigma_heart_lung_recon} that they are of high contrast and resolution. Now, in order to test the robustness of our method with noisy data, we introduce 5\% and 10\% multiplicative Gaussian noise to the interior data $\mathcal{H}$. 
We also modify the $L^2$ and $L^1$ regularization weights for $\sigma$, $\xi_1 = 0.1$ and $\gamma=0.1$. This ensures removal of artifacts from the reconstruction of $\sigma$ due to the noisy data, which in turn helps to remove some artifacts in the reconstruction of $D$. Figures \ref{D_heart_lung_recon_5}, \ref{D_heart_lung_recon_10}, \ref{sigma_heart_5} and \ref{sigma_heart_10} show the corresponding reconstructions. We see that the reconstructions of $D$ and $\sigma_a$ involve a few artifacts near the bottom, more pronounced in the case of 10\% noise. But overall, both reconstructions still preserve obtain high contrast and resolution near and at the regions of the heart and lungs. This is also validated analytically from the low RMSE \% and high PSNR values of the reconstructions presented in Table \ref{table:heart1}.}

\begin{table}[H]
\centering
\begin{tabular}{|c|c|c|c|c|}
\hline
Noise \% & RMSE \% ($D$) & RMSE \% ($\sigma_a$) & PSNR ($D$) & PSNR ($\sigma_a$) \\ [0.5ex]
\hline
0 & 2.32 &3.42 &59.96 &34.55\\
5&2.94 &3.46 &57.08 &33.89\\
10 &2.86 &7.56 &56.24 &28.05\\
[1ex]
\hline
\end{tabular}
\caption{Test Case 2: RMSE \% and PSNR values for reconstructions of the heart and lung phantom with different noise levels in the data}
\label{table:heart1}
\end{table}

\vspace{0.1cm}

\textbf{Test Case 3:} we investigate the robustness of our SQH-QPAT algorithm in reconstructing $D$ in the case where the Kubelka-Munk relation \eqref{eq:D_rel} is by far not satisfied; i.e. 
we have factors different from 100. For this test case, in regard to 
$\sigma$, we again consider the setting for heart and lung phantom as described in test Case 2. However, for the value of $D$, we choose $D=0.0006$ inside the right ellipse, $D=0.003$ inside the left ellipse, $D=0.009$ inside the disk and $D=0.1$ for the background. The value of $\sigma_b$ is again chosen to be 0.03.

\begin{figure}[H]
\centering
\subfloat[Exact $D$]{\includegraphics[width=0.35\textwidth]{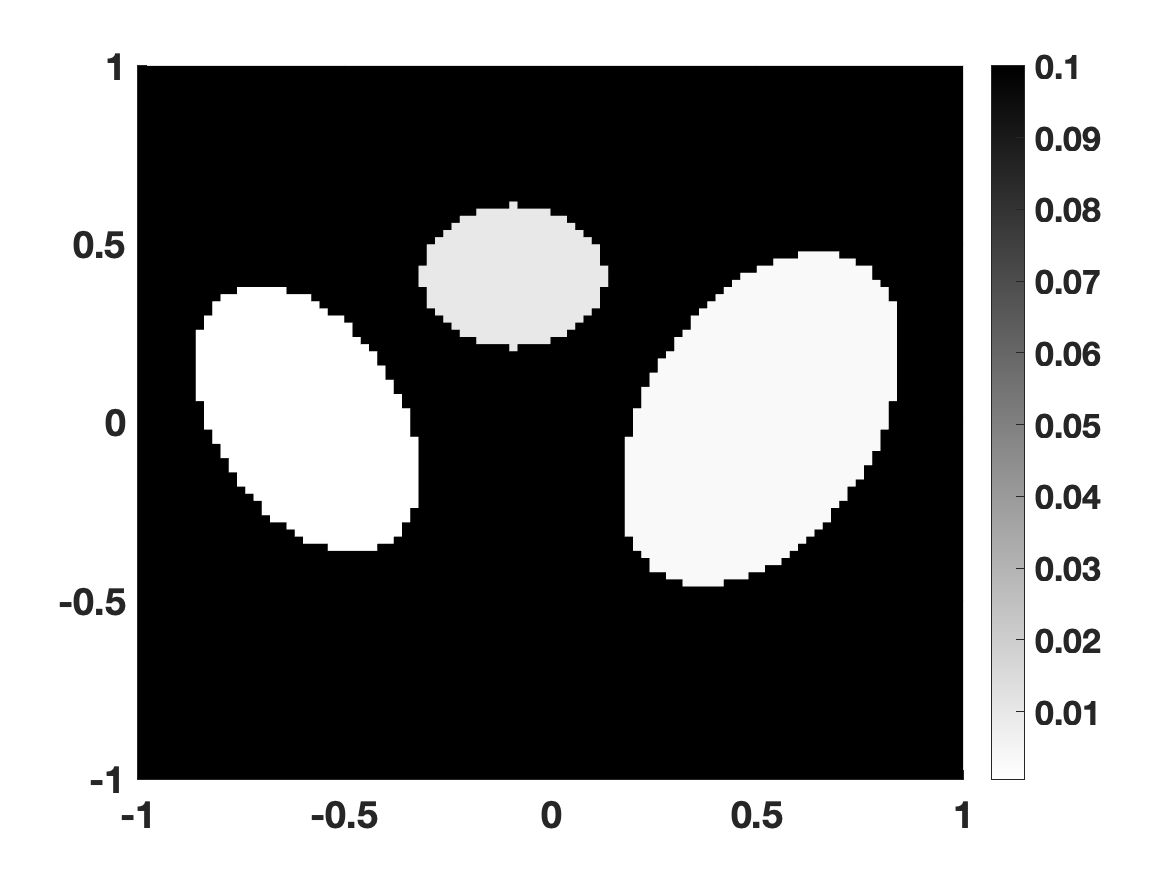}\label{D_heart_lung_actual2}}
\subfloat[Exact $\sigma_a$]{\includegraphics[width=0.35\textwidth]{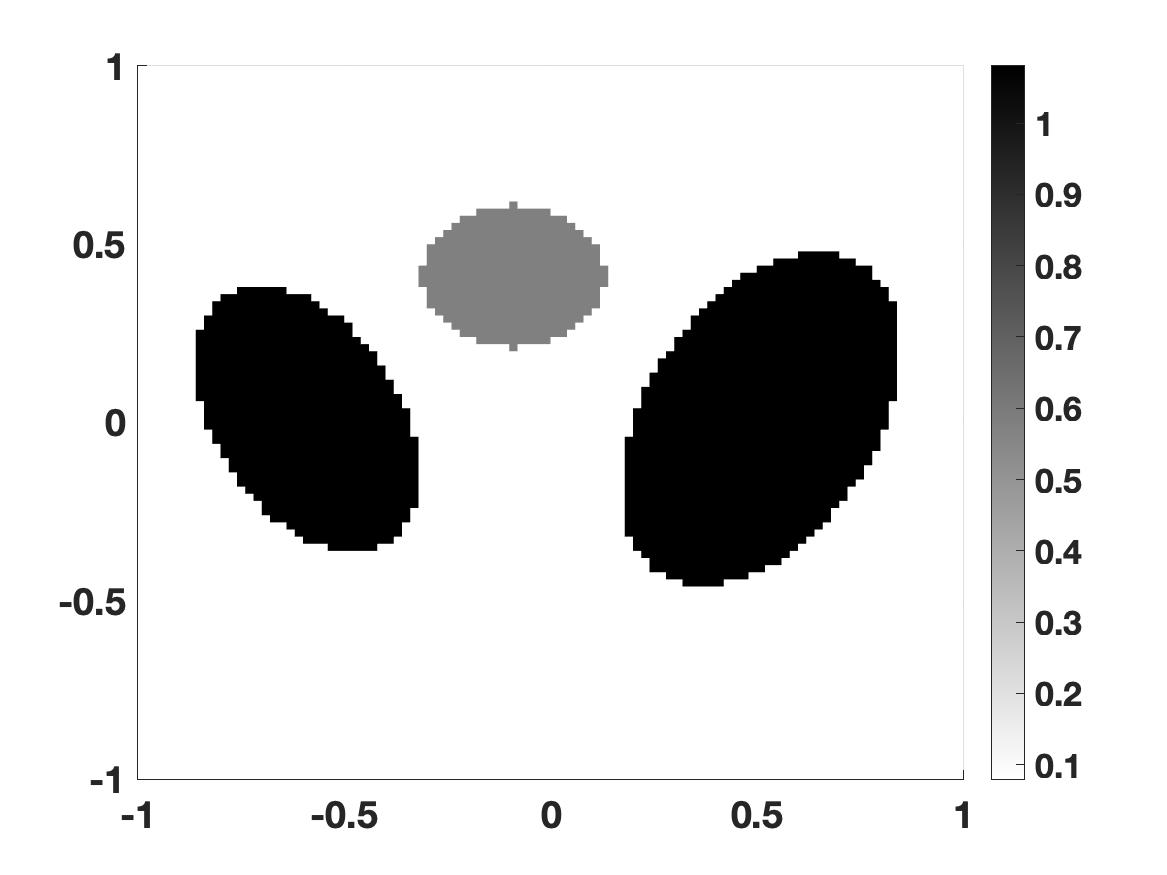}\label{sigma_heart_lung_actual2}}\\
\subfloat[Reconstructed $D$]{\includegraphics[width=0.35\textwidth]{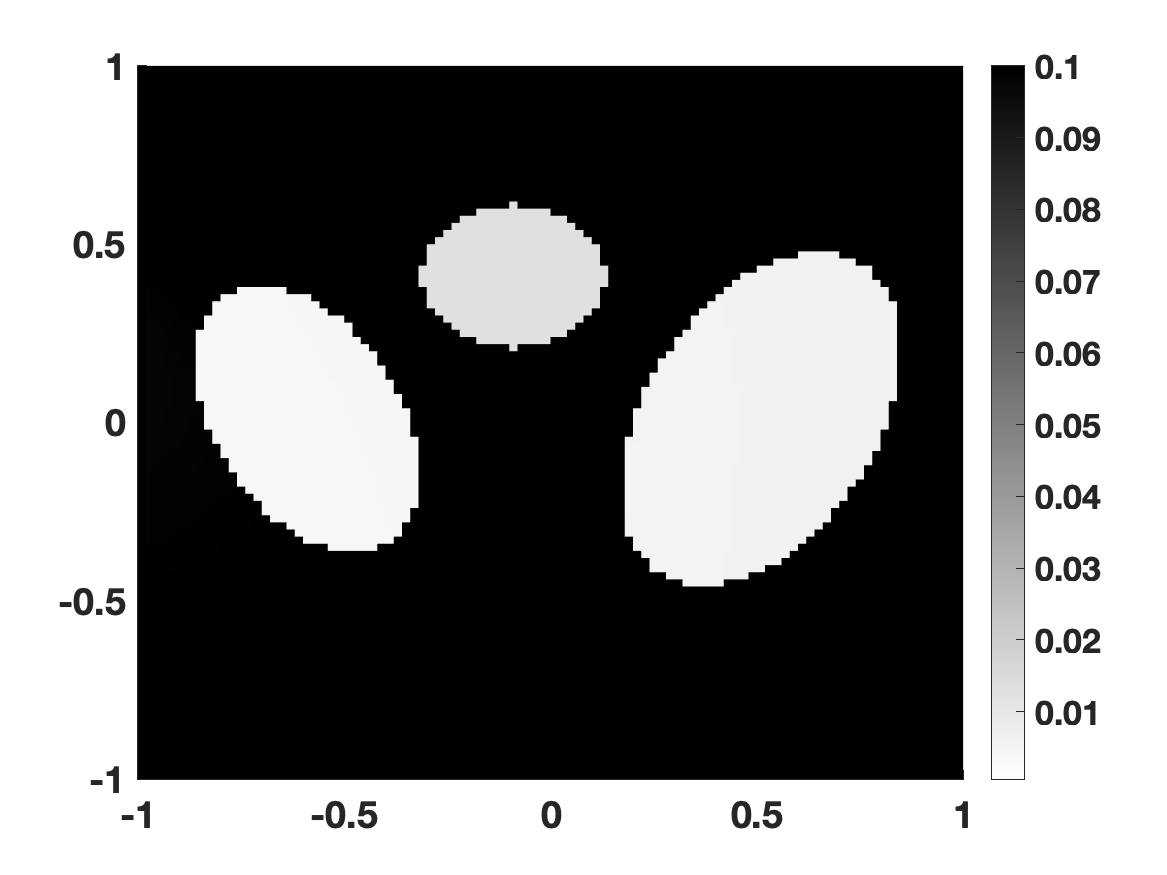}\label{D_heart_lung_recon2}}
\subfloat[Reconstructed $D$ with 5\% noise]{\includegraphics[width=0.35\textwidth]{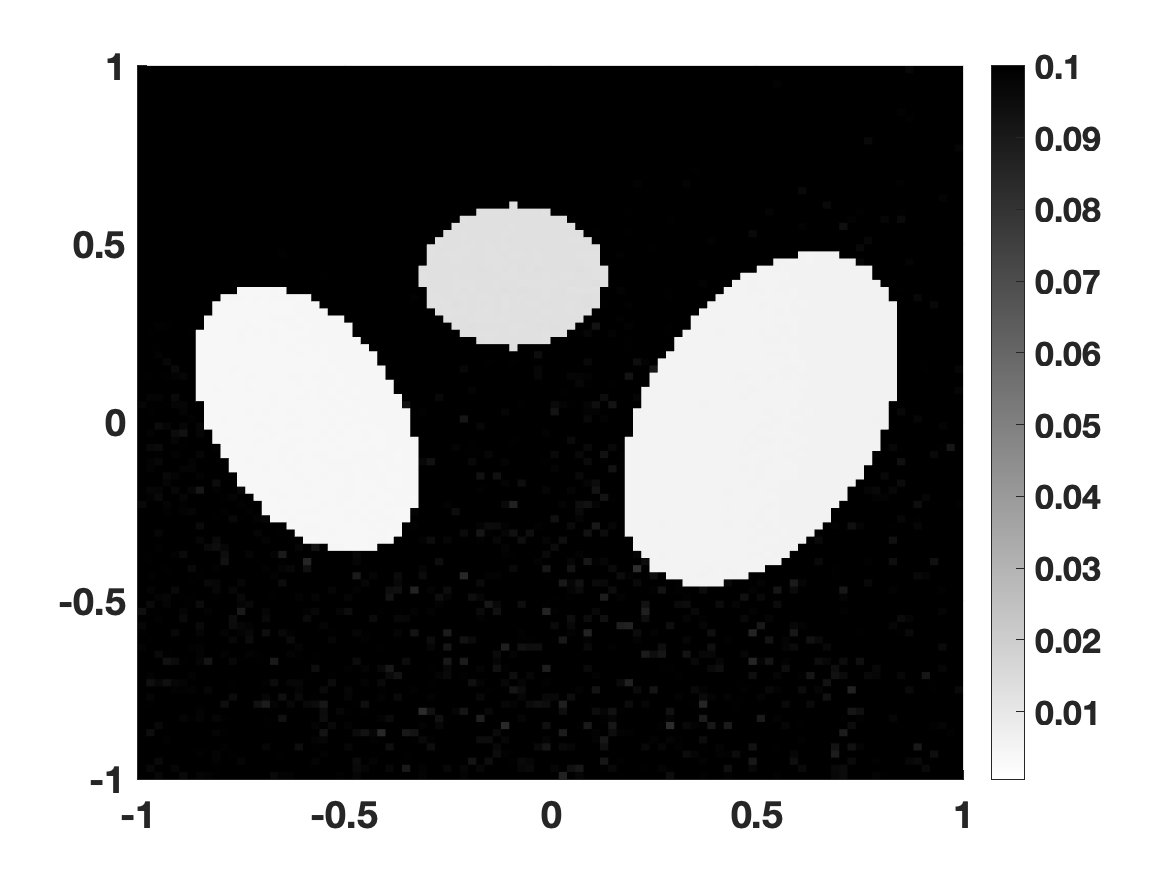}\label{D_heart_lung_recon2_5}}
\subfloat[Reconstructed $D$ with 10\% noise]{\includegraphics[width=0.35\textwidth]{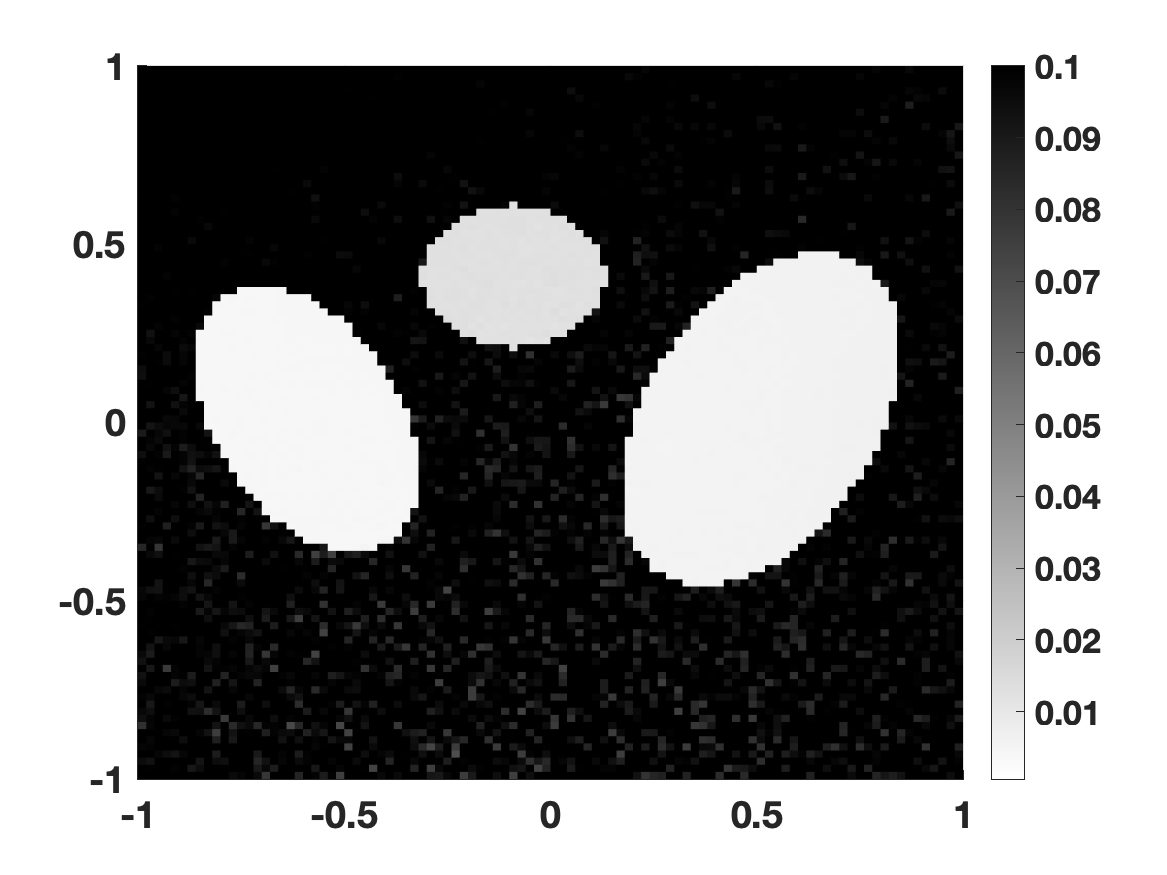}\label{D_heart_lung_recon2_10}}\\

\subfloat[Reconstructed $\sigma_a$]{\includegraphics[width=0.35\textwidth]{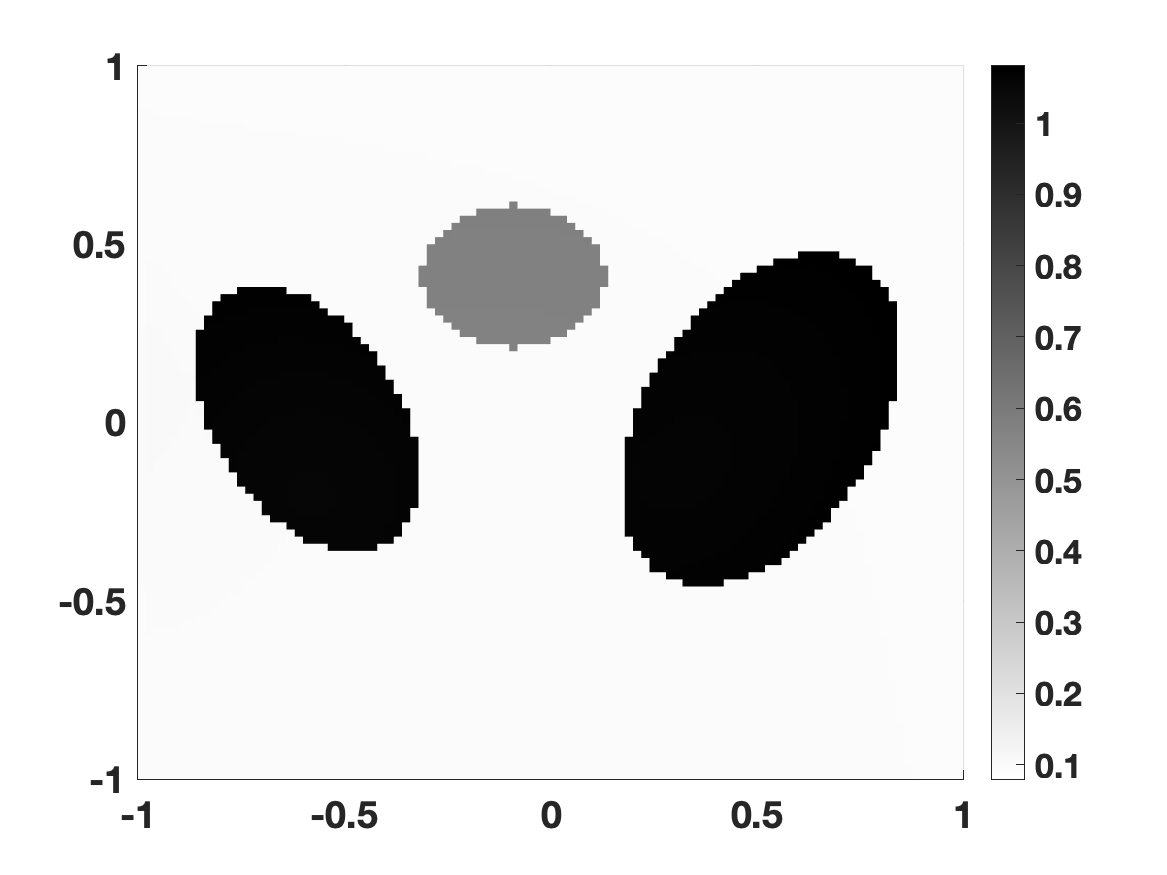}\label{sigma_heart_lung_recon2}}
\subfloat[Reconstructed $\sigma_a$ with 5\% noise]{\includegraphics[width=0.35\textwidth]{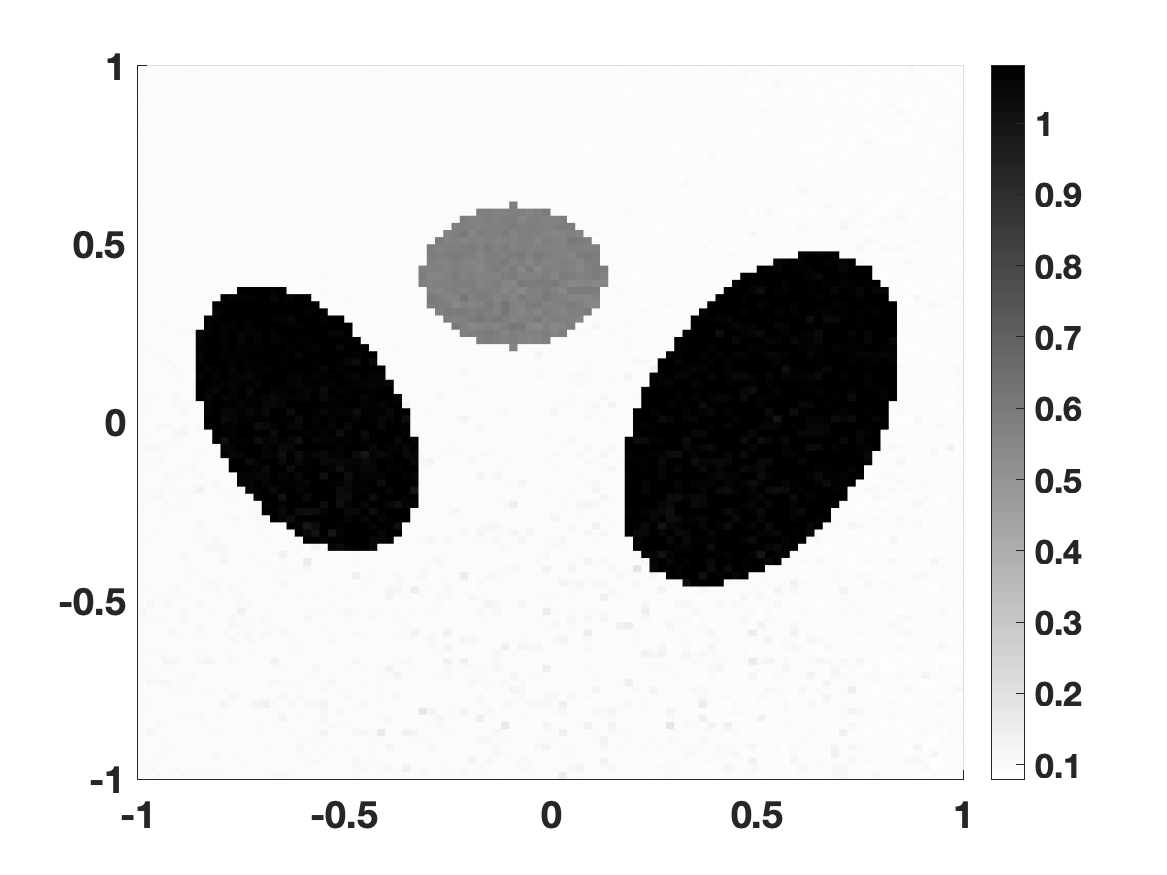}\label{sigma_heart_lung_recon2_5}}
\subfloat[Reconstructed $\sigma_a$ with 10\% noise]{\includegraphics[width=0.35\textwidth]{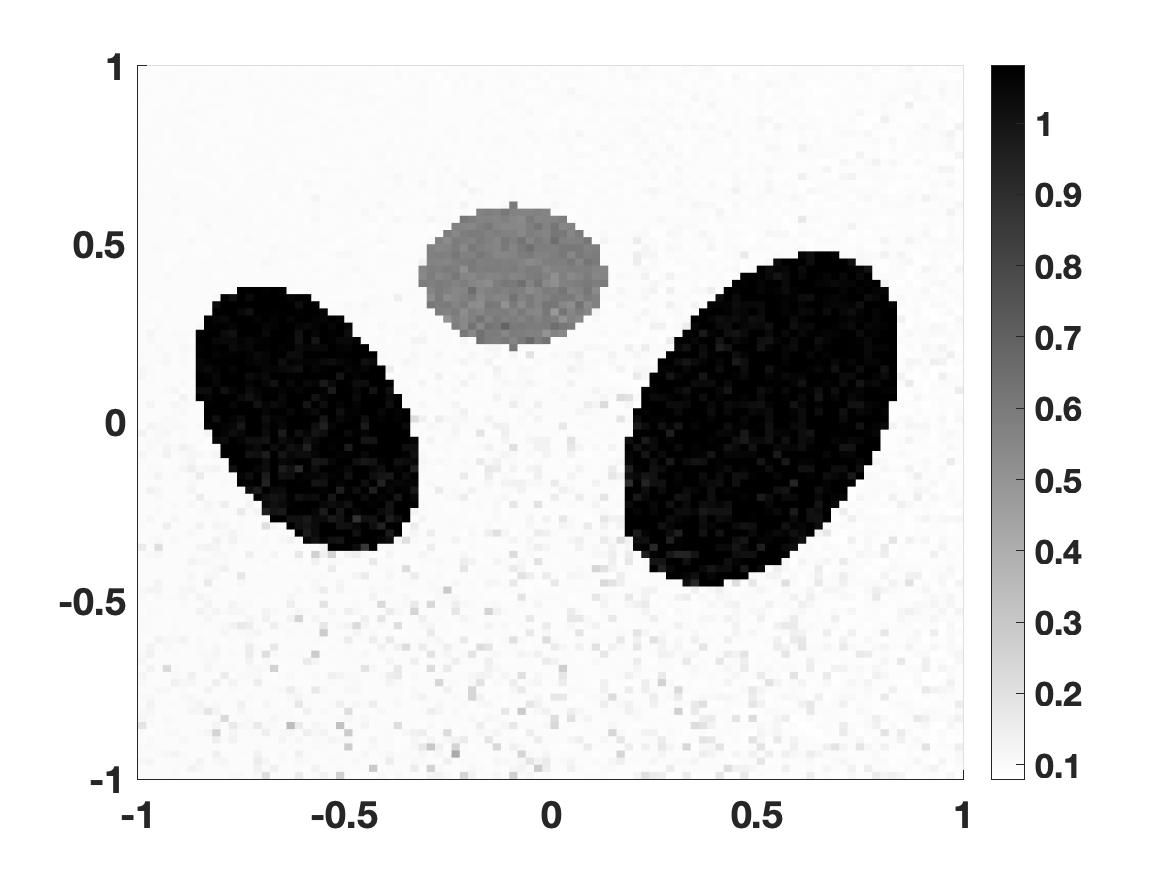}\label{sigma_heart_lung_recon2_10}}\\


\caption{Test Case 3: The actual and the reconstructed phantoms }
    \label{heart_lung_2}
  \end{figure}
  
{With this challenging setting, we perform our reconstructions of $D$ and $\sigma_a$ and 
depict the results in Figure \ref{heart_lung_2} with and without noise in the data. We again observe high contrast and resolution reconstructions of $D$ and $\sigma_a$, even in the presence of 5\% and 10\% noise in the data. The corresponding RMSE \% and PSNR values are presented in Table \ref{table:heart2}, which are similar to the previous test case.}
\begin{table}[H]
\centering
\begin{tabular}{|c|c|c|c|c|}
\hline
Noise \% & RMSE \% ($D$) & RMSE \% ($\sigma_a$) & PSNR ($D$) & PSNR ($\sigma_a$) \\ [0.5ex]
\hline
0 & 1.47 &3.25 &56.69 &35.88\\
5&1.43 &3.57 &54.18 &34.57\\
10 &2.13 &4.61 &47.70 &30.05\\
[1ex]
\hline
\end{tabular}
\caption{Test Case 3: RMSE \% and PSNR values for reconstructions of the second heart and lung phantom with different noise levels in the data}
\label{table:heart2}
\end{table}

\vspace{0.1cm}

\textbf{Test Case 4:} we perform another challenging test using the Shepp-Logan phantom that represents the structure of a human head, which comprises 
ellipses, disks of varying radii, and an elliptical annulus. In 
this case, $\sigma$ has the structure 
of Shepp-Logan phantom as given in \cite{Shepp1974}. The phantom is described as follows: The values of $\sigma$ and $D$, respectively, inside two big ellipses centered at (0.22,0) and (-0.22,0) are -0.3 and 0.02, inside the big disk centered at (0,0.35) are 0.5 and 0.003, inside the 2 small disks centered at (0,0.1) and (0,-0.1) are 0.5 and 0.003, inside the small ellipses centered at (-0.08,-0.605), (0,-0.606) and (0.06, -0.605) are 0.5 and 0.003, inside the elliptical annulus are 1 and 0.002, and elsewhere 0 and 0.006. The background $\sigma_b$ is chosen to be 0.5 in this case. 

\begin{figure}[H]
\centering
\subfloat[Exact $D$]{\includegraphics[width=0.35\textwidth]{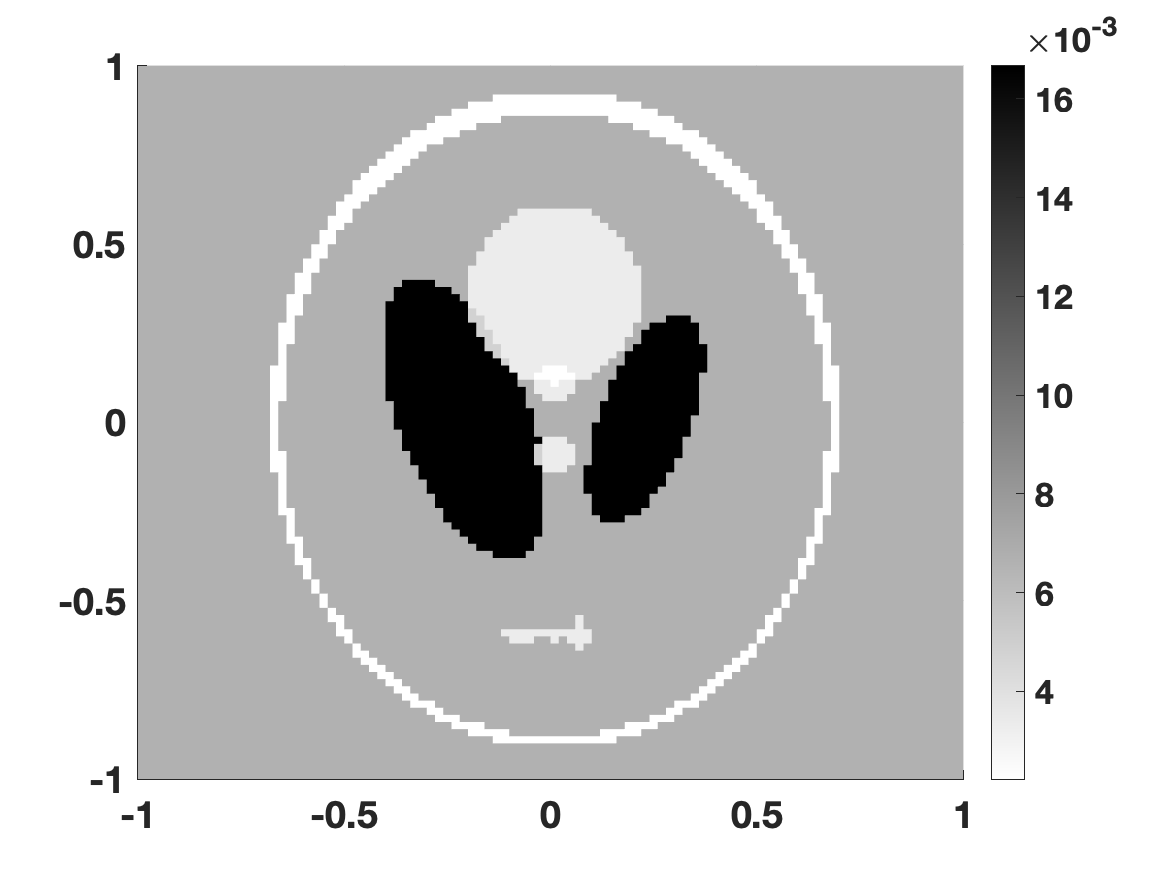}\label{D_shepp_actual}}
\subfloat[Exact $\sigma_a$]{\includegraphics[width=0.35\textwidth]{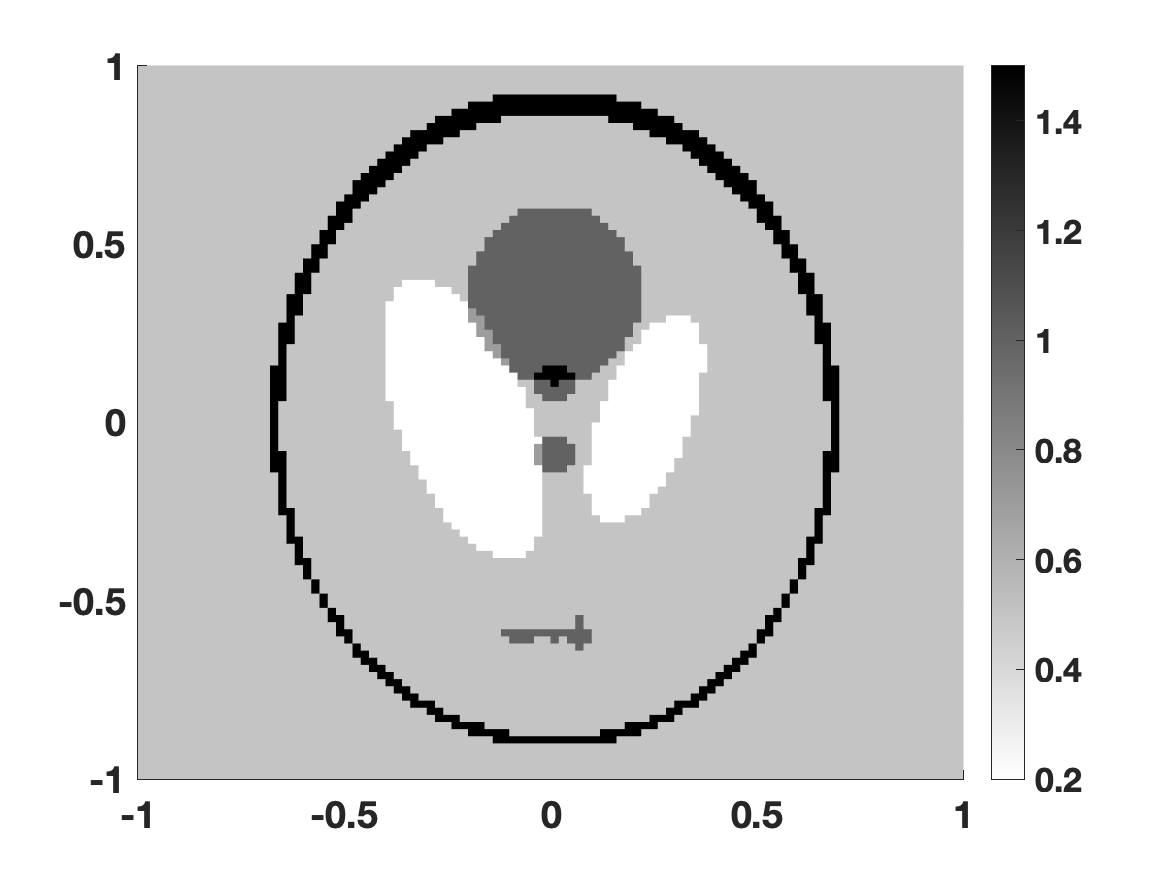}\label{sigma_shepp_actual}}\\
\subfloat[Reconstructed $D$]{\includegraphics[width=0.35\textwidth]{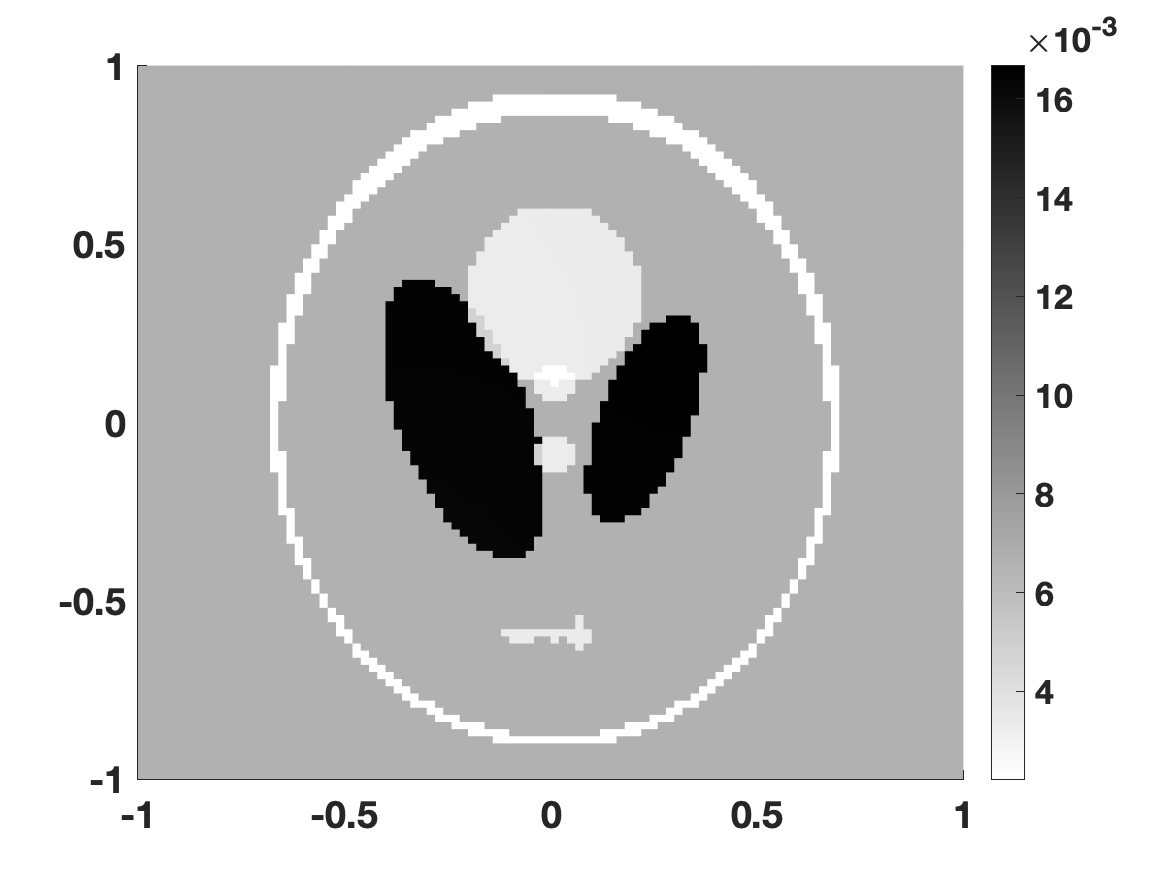}\label{D_shepp_recon}}
\subfloat[Reconstructed $D$ with 5\% noise]{\includegraphics[width=0.35\textwidth]{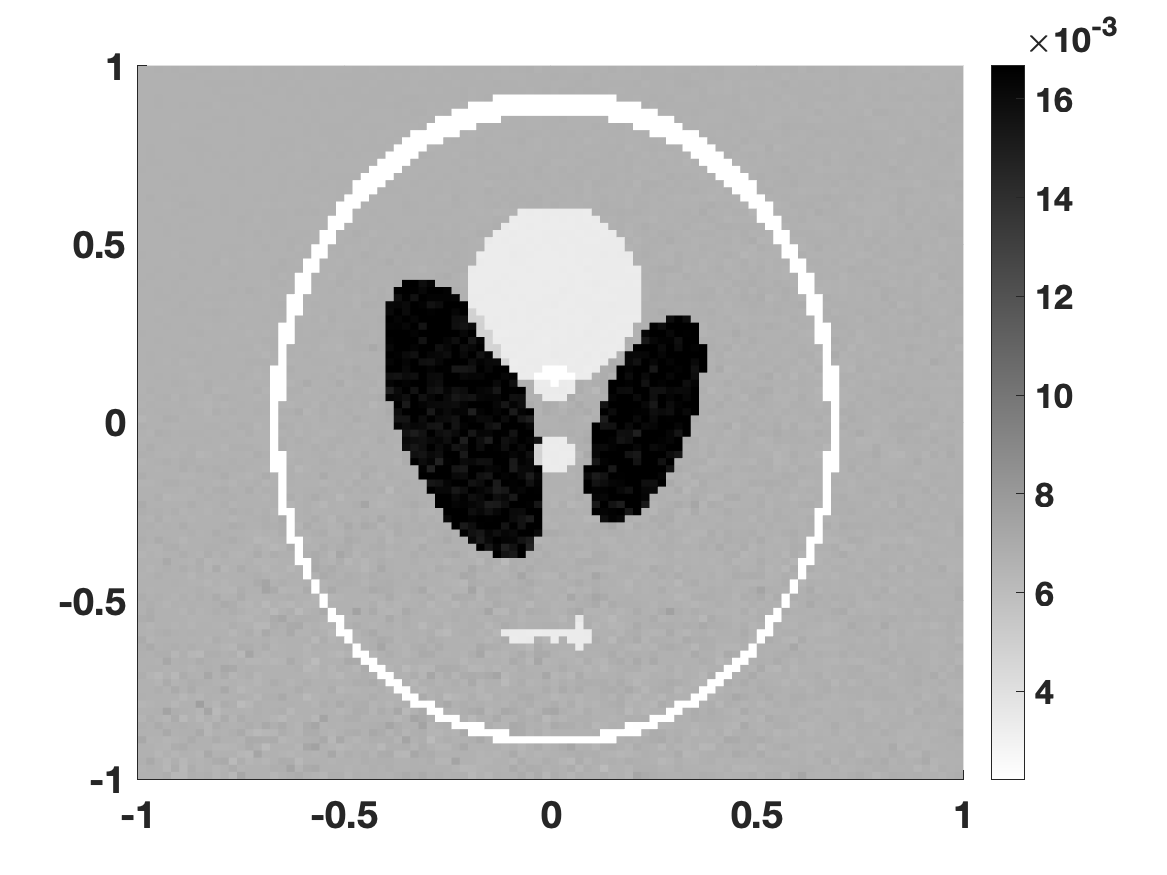}\label{D_shepp_recon_5}}
\subfloat[Reconstructed $D$ with 10\% noise]{\includegraphics[width=0.35\textwidth]{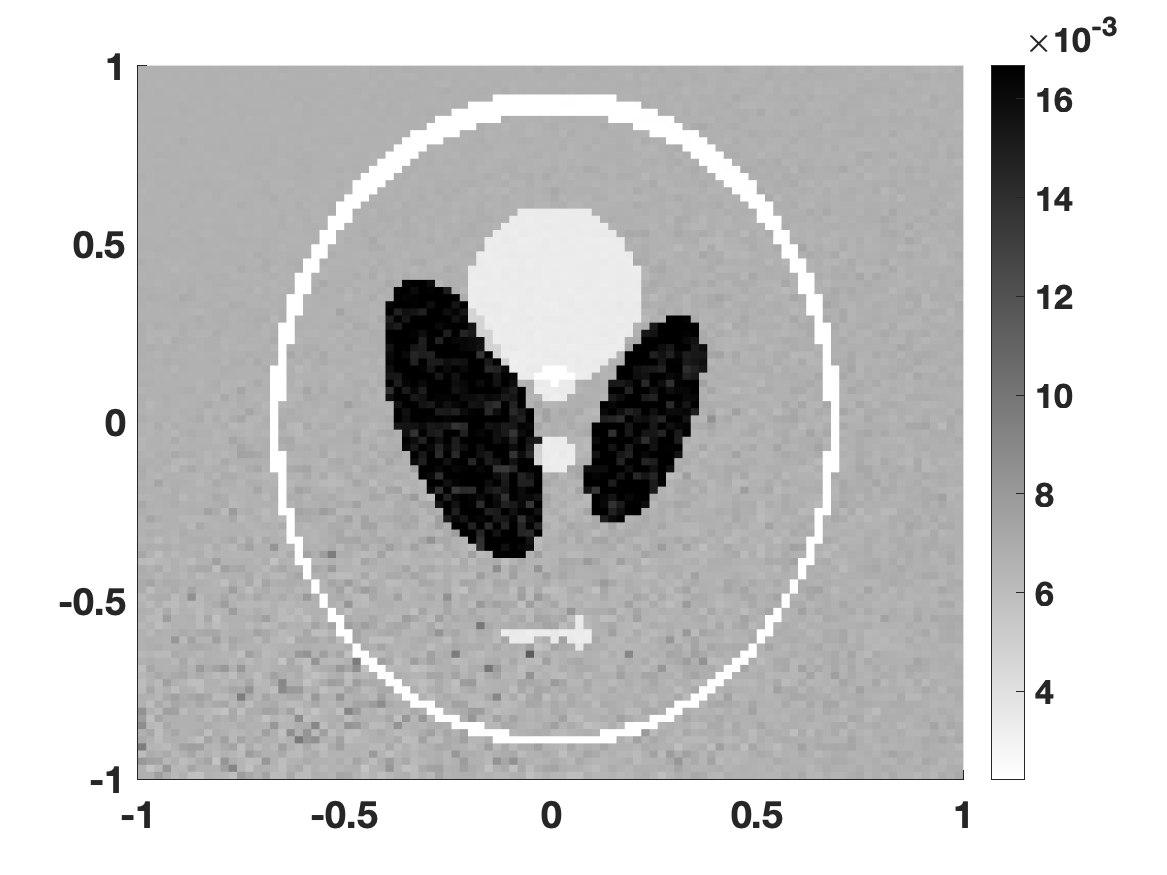}\label{D_shepp_recon_10}}\\

\subfloat[Reconstructed $\sigma_a$]{\includegraphics[width=0.35\textwidth]{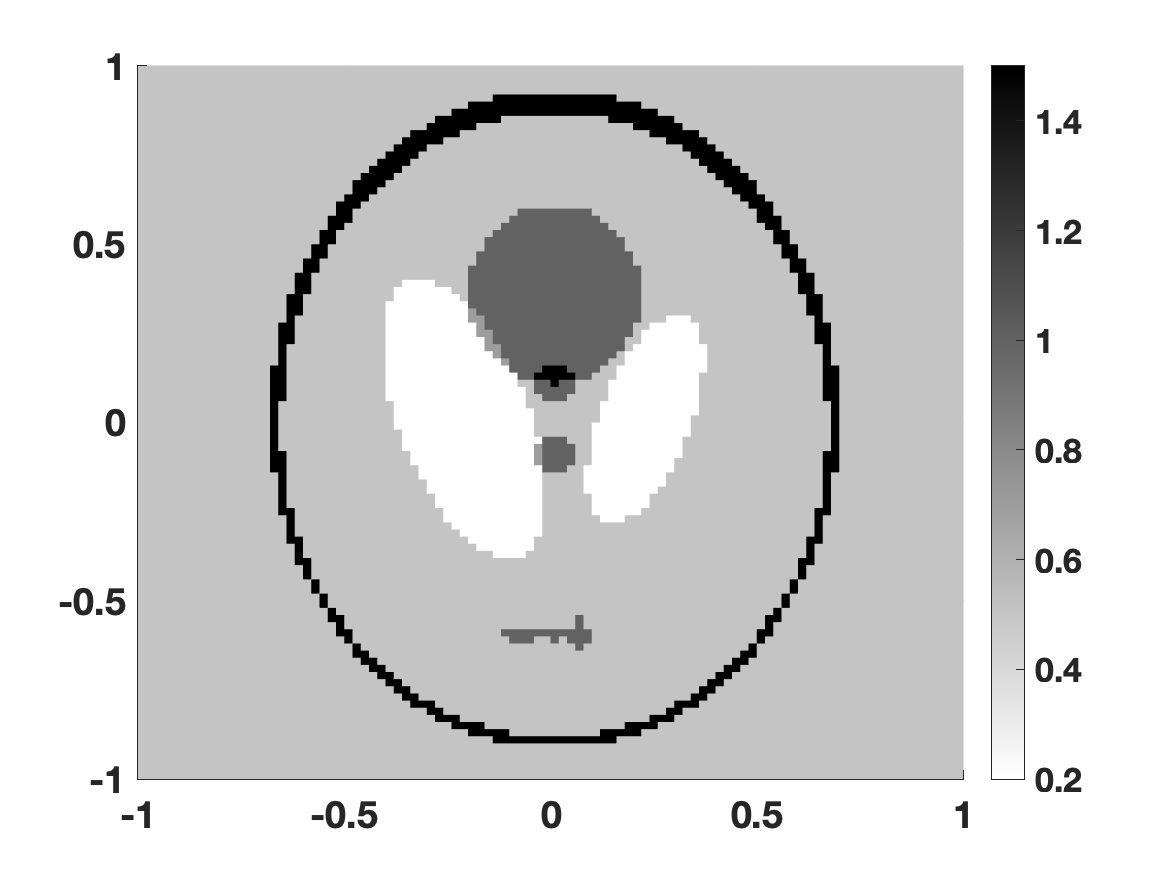}\label{sigma_shepp_recon}}
\subfloat[Reconstructed $\sigma_a$ with 5\% noise]{\includegraphics[width=0.35\textwidth]{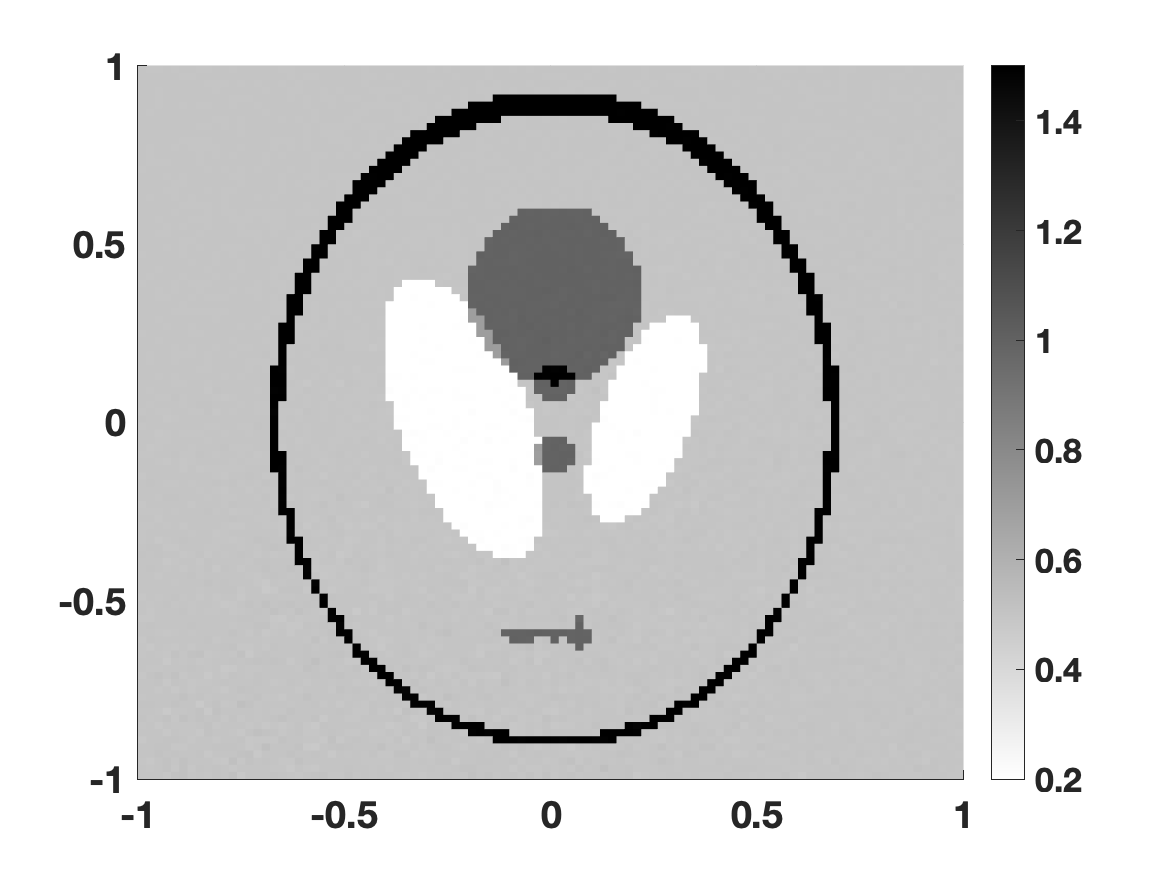}\label{sigma_shepp_recon_5}}
\subfloat[Reconstructed $\sigma_a$ with 10\% noise]{\includegraphics[width=0.35\textwidth]{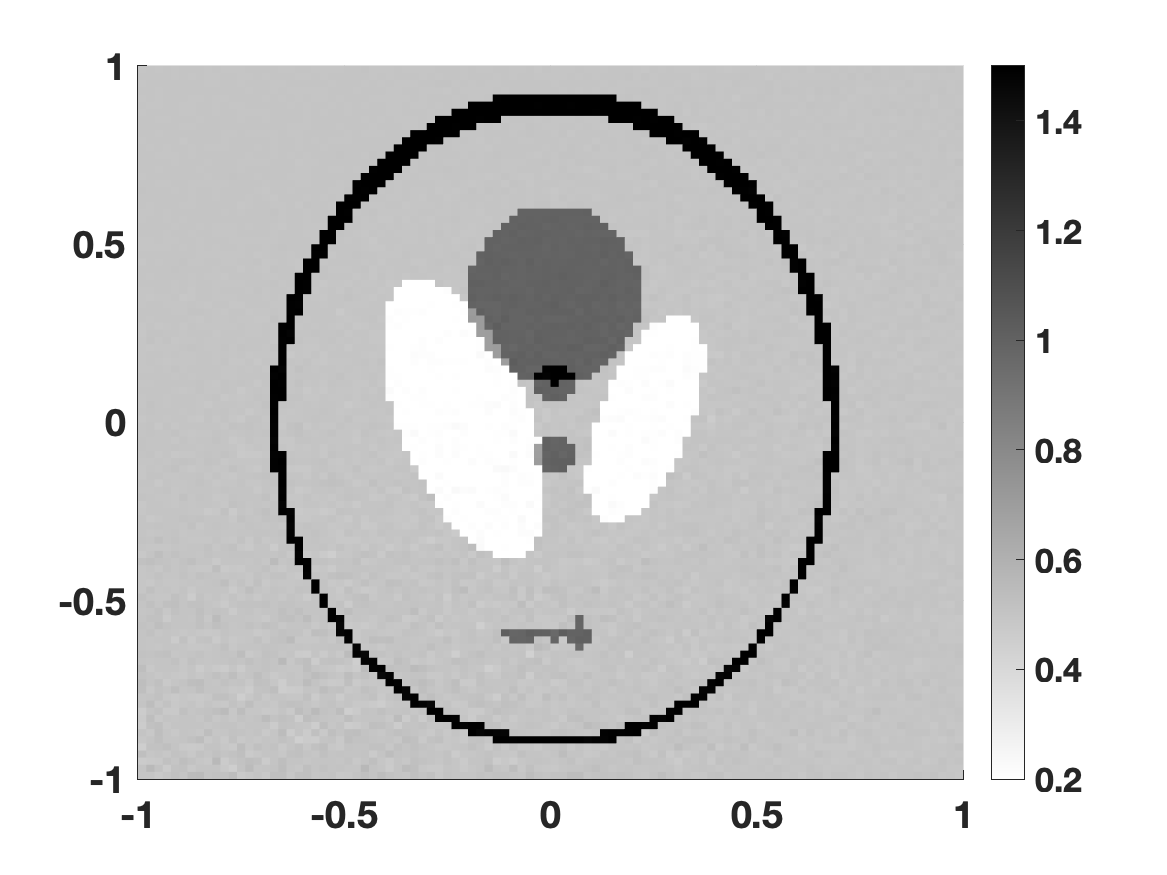}\label{sigma_shepp_recon_10}}\\
\caption{Test Case 4: The actual and the reconstructed phantoms }
    \label{shepp}
  \end{figure}
 
{For  test Case 4, Figure \ref{shepp} shows the exact and the reconstructed values of $D$ and $\sigma_a$ with high resolution and contrast. Figures \ref{D_shepp_recon_5}, \ref{D_shepp_recon_10} and \ref{sigma_shepp_recon_5}, \ref{sigma_shepp_recon_10} depict the reconstruction with 5\% and 10\% multiplicative Gaussian noise in the interior data, respectively. We obtain high quality reconstruction with our SQH-QPAT framework also in the case of objects with holes and inclusions. The corresponding RMSE \% and PSNR values are presented in Table \ref{table:shepp1}. We do observe that the RMSE \% is higher than in the case of the heart and lung phantom, which is expected since there are far more finer features in the Shepp-Logan phantom. However, the PSNR values are still at a comparable level to the previous test cases that indicate that the reconstructions have high quality.}

\begin{table}[H]
\centering
\begin{tabular}{|c|c|c|c|c|}
\hline
Noise \% & RMSE \% ($D$) & RMSE \% ($\sigma_a$) & PSNR ($D$) & PSNR ($\sigma_a$) \\ [0.5ex]
\hline
0 & 16.21 &6.54 &55.89 &38.57\\
5&16.36 &6.68 &55.80 &38.02\\
10 &16.58 &6.71 &55.47 &36.31\\
[1ex]
\hline
\end{tabular}
\caption{Test Case 4: RMSE \% and PSNR values for reconstructions of the Shepp-Logan phantom with different noise levels in the data}
\label{table:shepp1}
\end{table}

\vspace{0.1cm}
 
  \textbf{Test Case 5:} while in the previous test cases the phantoms to be reconstructed represented organs in a human being, in this test case, we present results of detecting a tumor along with organs using our SQH-QPAT framework. We consider $\sigma$ that represents the organs as the same Shepp-Logan phantom of test Case 4. The $\sigma$ for the tumors are represented by two disks: one of them having a large contrast with value 1.5 and the other having a contrast similar to the some of the organs with value 0.2. The corresponding value of $D$ is 0.001 and 0.004, respectively. The background value $\sigma_b = 0.5$. 

\begin{figure}[H]
\centering
\subfloat[Exact $D$]{\includegraphics[width=0.35\textwidth]{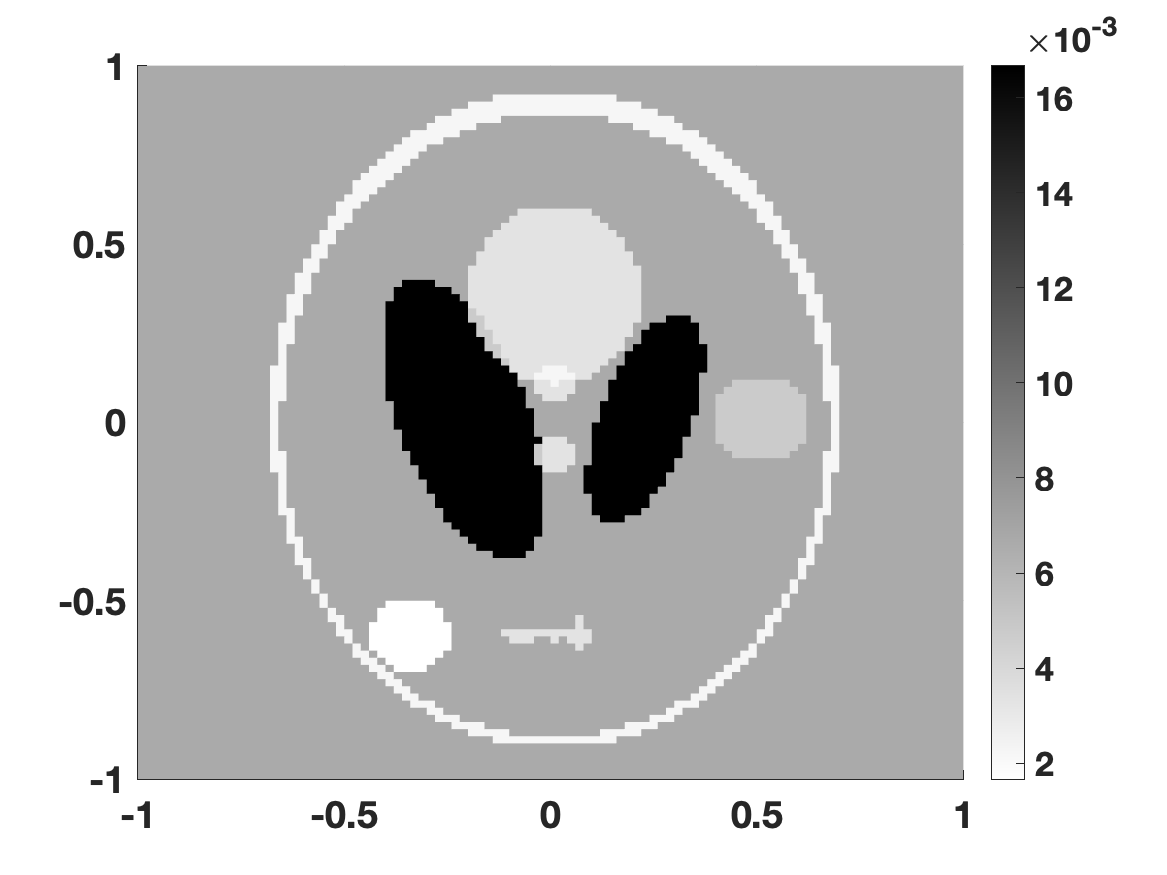}\label{D_shepp2_actual}}
\subfloat[Exact $\sigma_a$]{\includegraphics[width=0.35\textwidth]{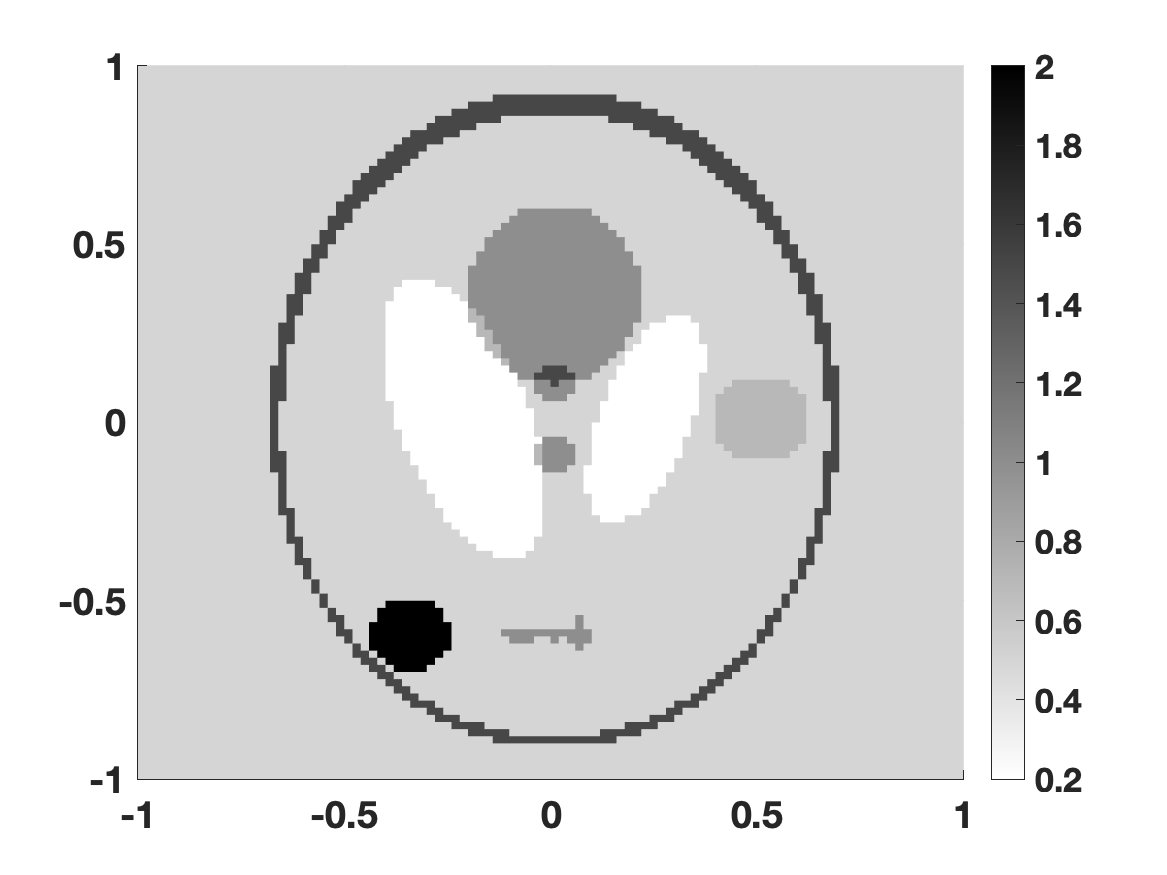}\label{sigma_shepp2_actual}}
\subfloat[Convergence history of SQH-QPAT iteration]{\includegraphics[width=0.35\textwidth]{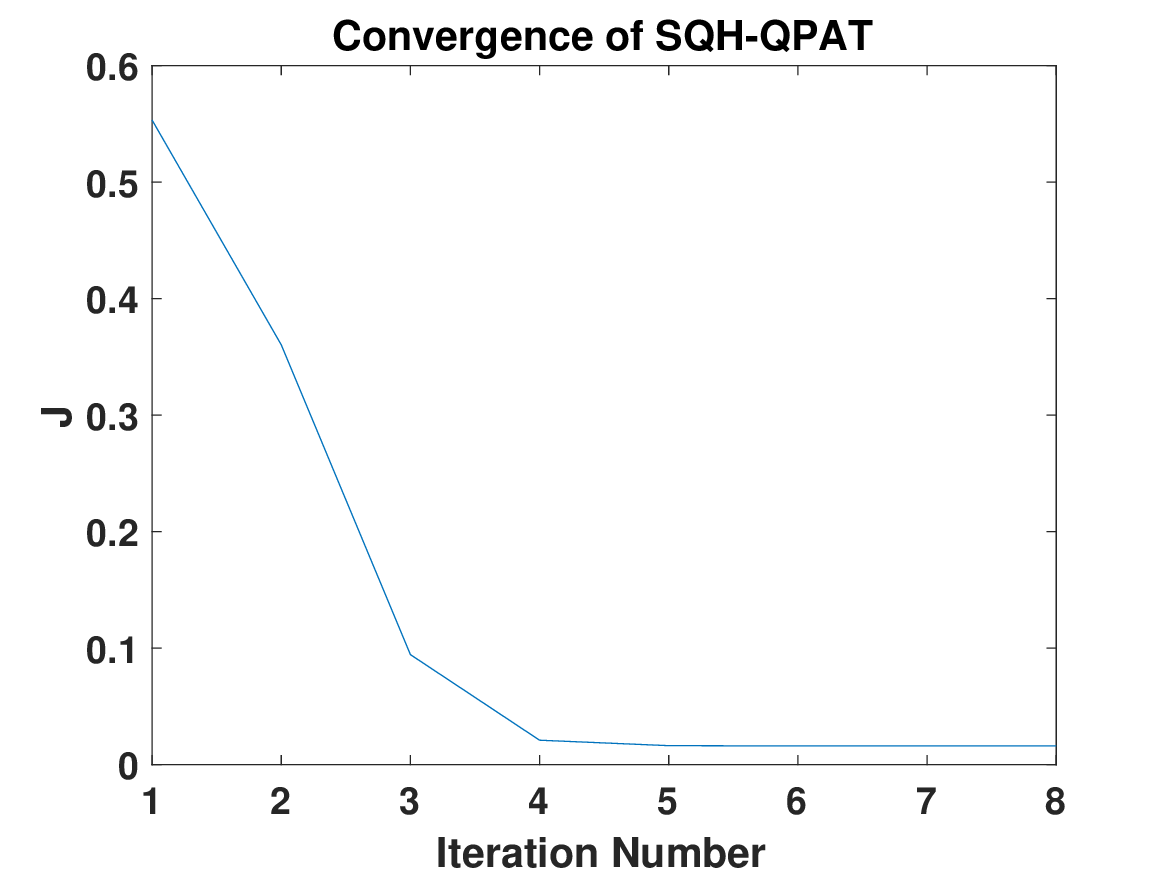}\label{convergence}}\\
\subfloat[Reconstructed $D$]{\includegraphics[width=0.35\textwidth]{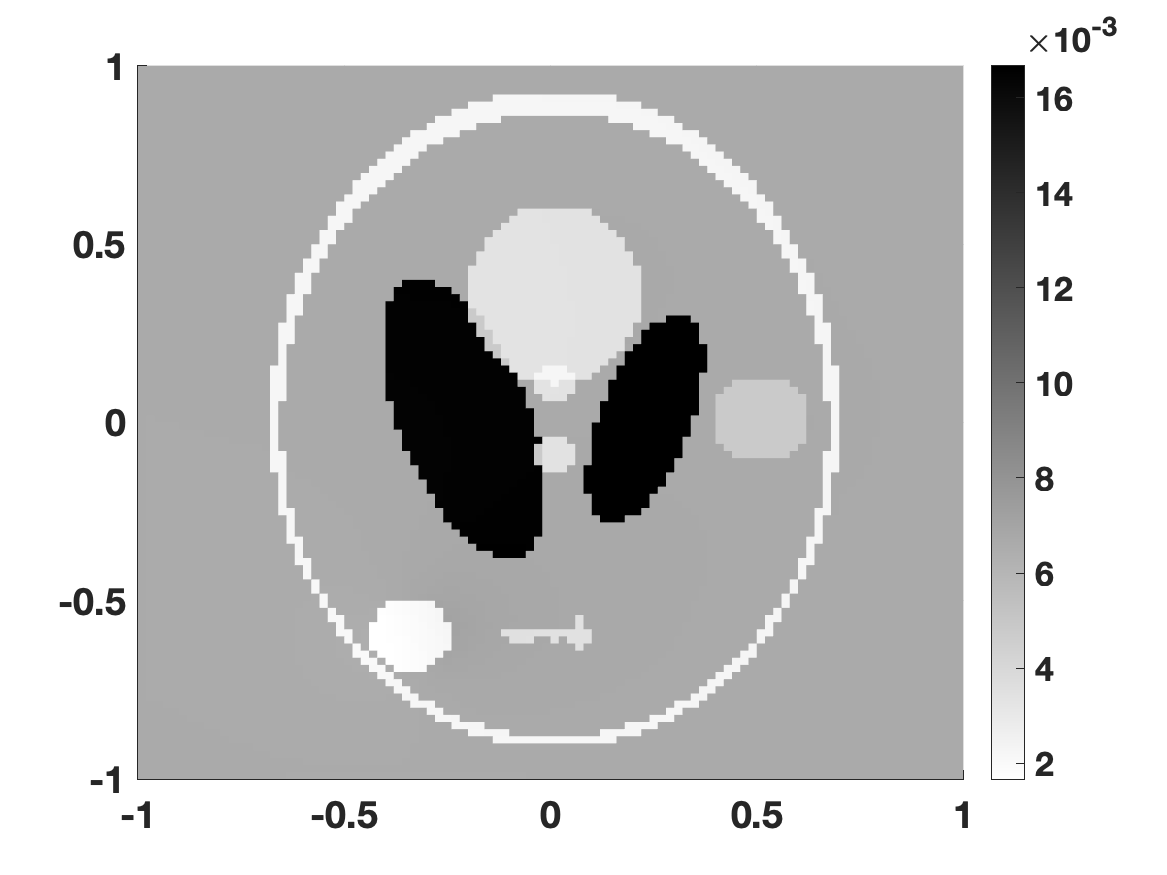}\label{D_shepp2_recon}}
\subfloat[Reconstructed $D$ with 5\% noise]{\includegraphics[width=0.35\textwidth]{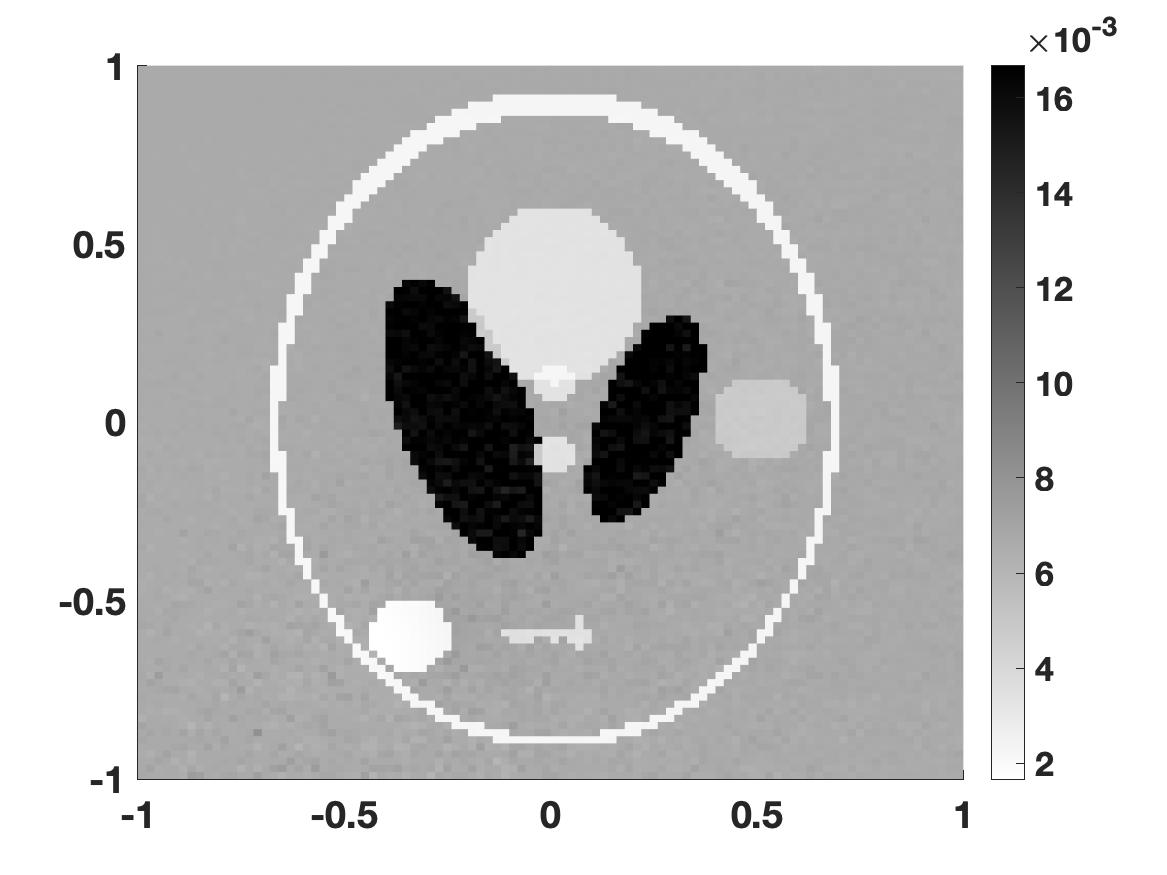}\label{D_shepp2_recon_5}}
\subfloat[Reconstructed $D$ with 10\% noise]{\includegraphics[width=0.35\textwidth]{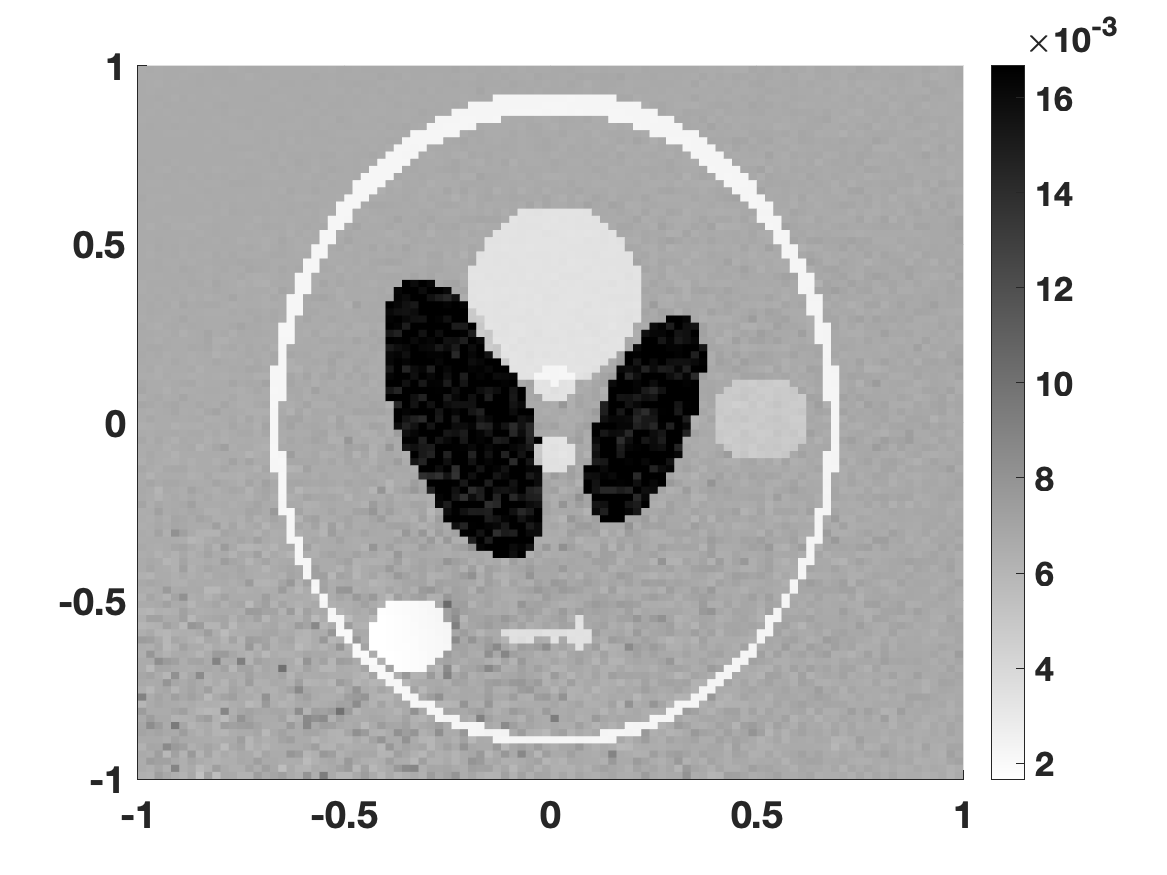}\label{D_shepp2_recon_10}}\\

\subfloat[Reconstructed $\sigma_a$]{\includegraphics[width=0.35\textwidth]{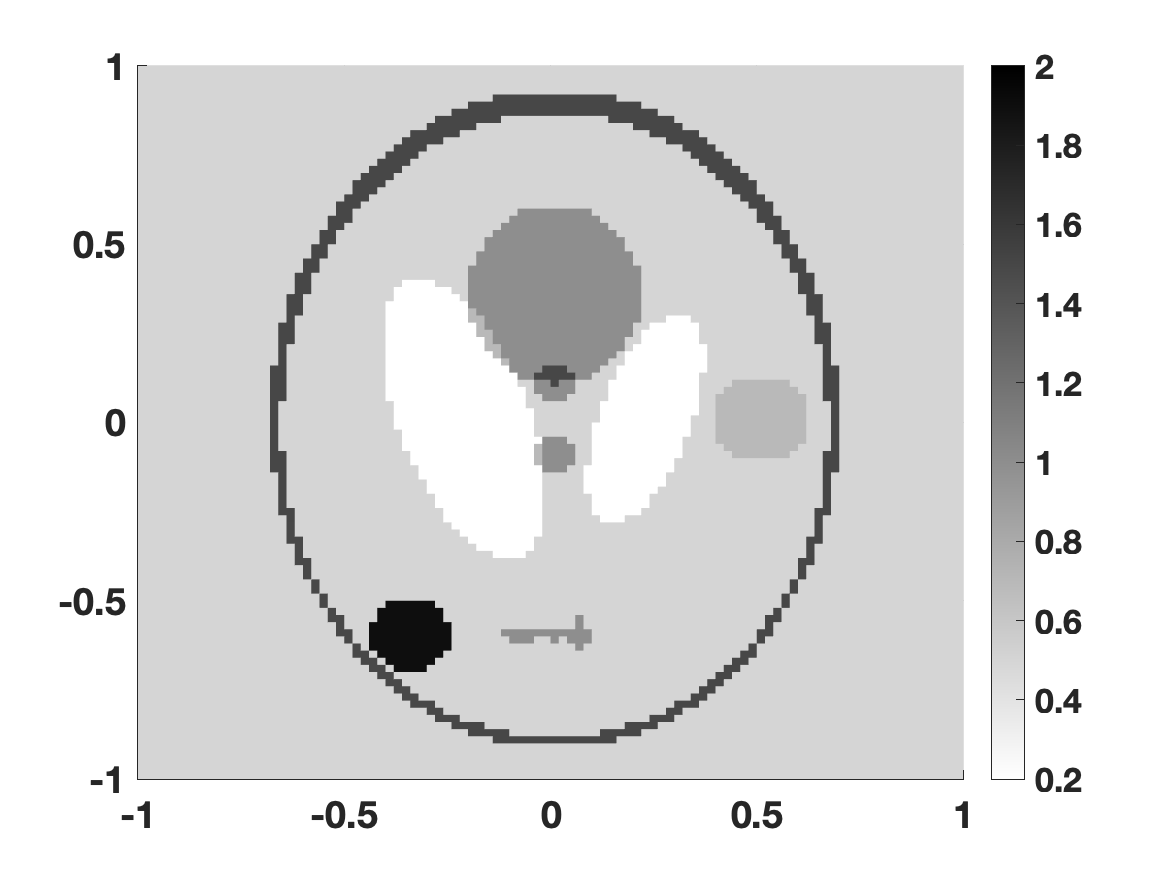}\label{sigma_shepp2_recon}}
\subfloat[Reconstructed $\sigma_a$ with 5\% noise]{\includegraphics[width=0.35\textwidth]{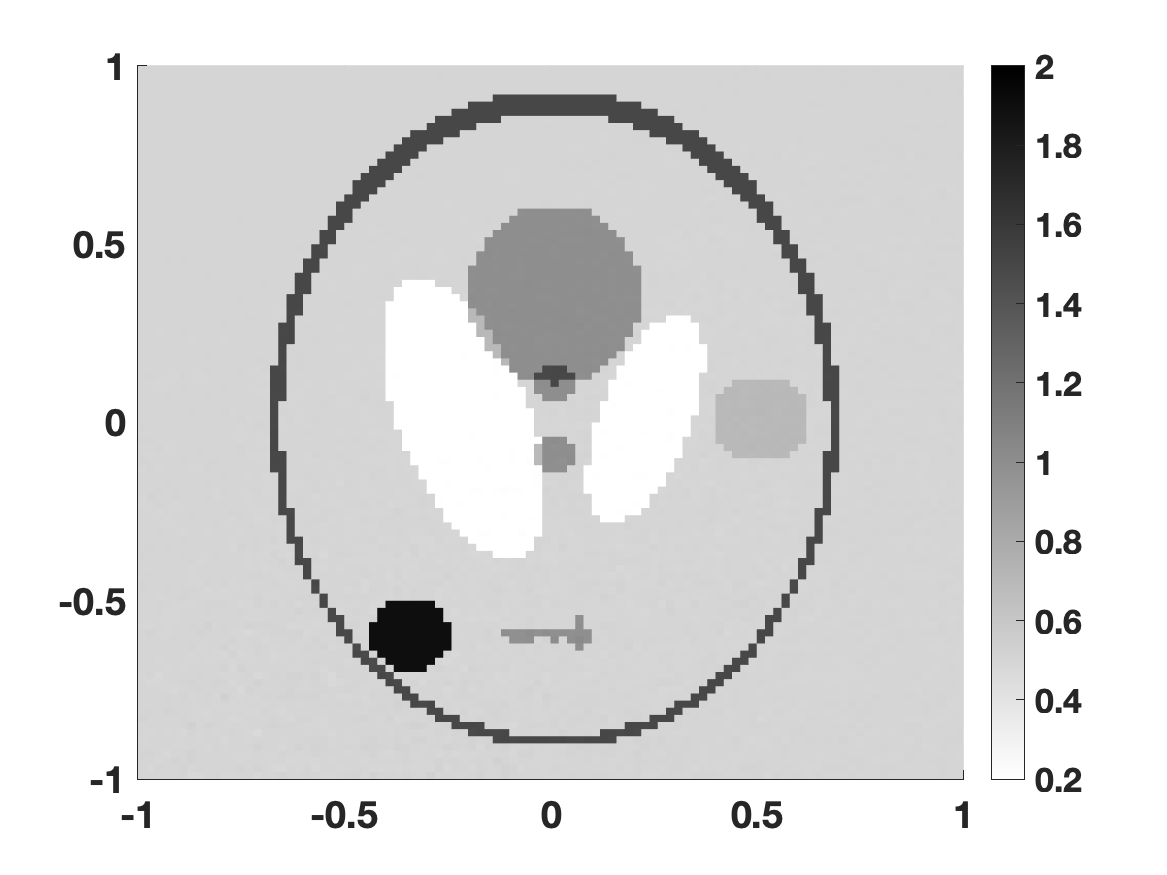}\label{sigma_shepp2_recon_5}}
\subfloat[Reconstructed $\sigma_a$ with 10\% noise]{\includegraphics[width=0.35\textwidth]{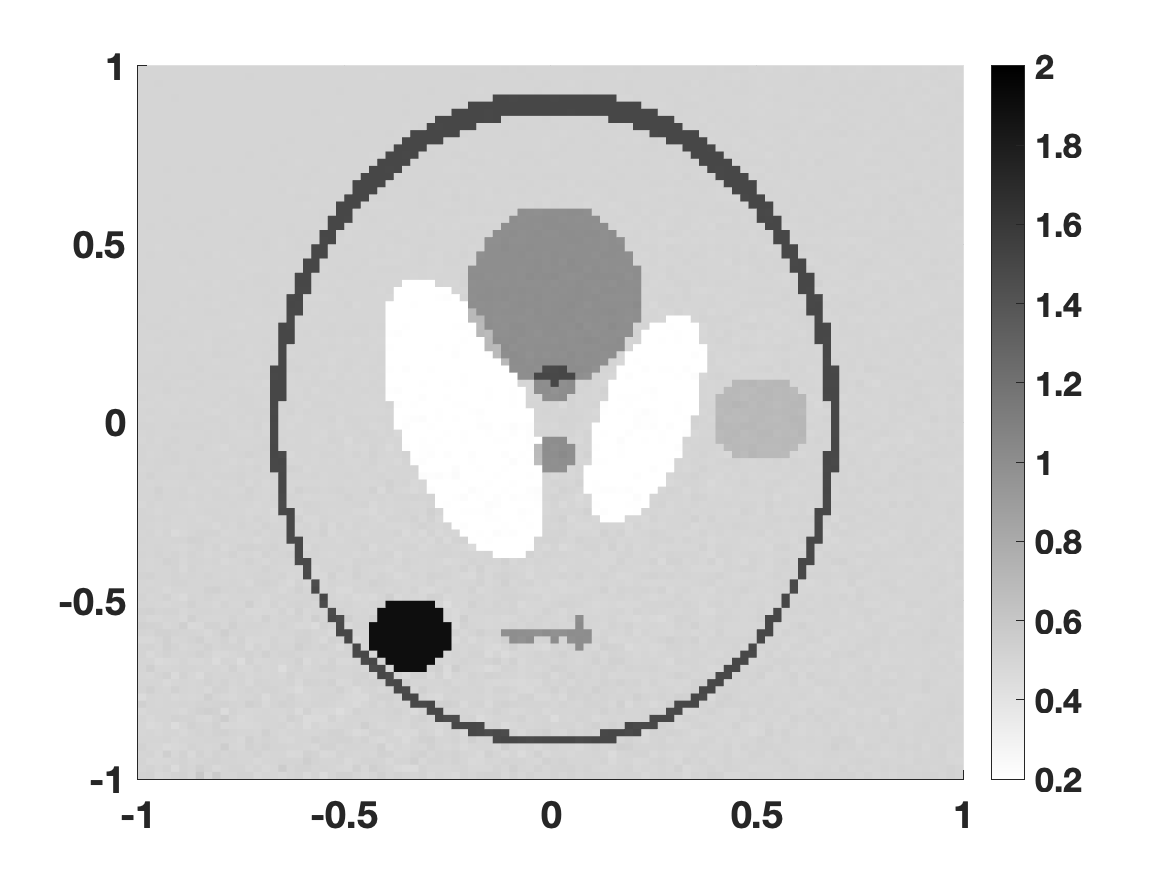}\label{sigma_shepp2_recon_10}}\\

\caption{Test Case 5: The actual and the reconstructed phantoms }
    \label{shepp2}
  \end{figure}
  
{For test Case 5, Figure \ref{shepp2} shows the exact and the reconstructed values of $D$ and $\sigma_a$ with and without noise in the interior data. We see that our SQH-QPAT framework is able to reconstruct very well the tumors alongside the organs, which further demonstrates the wide applicability of our reconstruction framework. Figure \ref{convergence} shows the convergence history of the SQH-QPAT scheme. Notice the strict monotonic decrease in the functional values, with stopping criteria fulfilled in 8 iterations. From Table \ref{table:shepp2},  we also observe similar RMSE \% and PSNR values, as in the previous test case, demonstrating high quality reconstructions.}

\begin{table}[H]
\centering
\begin{tabular}{|c|c|c|c|c|}
\hline
Noise \% & RMSE \% ($D$) & RMSE \% ($\sigma_a$) & PSNR ($D$) & PSNR ($\sigma_a$) \\ [0.5ex]
\hline
0 & 16.50 &7.19 &55.89 &36.62\\
5&16.69 &7.40 &55.79 &36.25\\
10 &16.90 &7.40 &55.45 &35.12\\
[1ex]
\hline
\end{tabular}
\caption{Test Case 5: RMSE \% and PSNR values for reconstructions of the Shepp-Logan phantom with different noise levels in the data}
\label{table:shepp2}
\end{table}

  \section{Conclusion}
  
A new optimization framework for reconstruction with high contrast and resolution of optical diffusion and absorption coefficients in quantitative photoacoustic tomography (QPAT) was presented. The resulting QPAT inverse problem was formulated as the minimization of a best-fit functional with a 
tentative prior relating the diffusion and absorption coefficients 
and $L^2$ and $L^1$ regularization terms, and
subject to the constraint of the steady photon-propagation equation. 

For the solution of this problem, a robust and fast sequential quadratic hamiltonian algorithm based on the Pontryagin maximum principle 
was developed and analysed. Results of numerical experiments 
were presented that demonstrate the ability of the proposed framework to 
efficiently compute reconstructions of the optical coefficients with high contrast and resolution.

\section*{Acknowledgments}
This work was partially supported by the German BMBF-Project iDeLIVER : Intelligent MR Diagnosis of the Liver by Linking Model and Data-Driven Processes, and by the US National Science Foundation Grant No: 2309491, and the University of Texas at Arlington, Research Enhancement Program Grant No: 2022-605.

\begin{appendices}
\section{Proof of Theorem \ref{sqh2}}

We have 
$$
H_\epsilon(x,y,D,\sigma,D^k,\sigma^k,u_1^k,u_2^k,q_1^k,q_2^k)\leq H_\epsilon(x,y,w,z,D^k,\sigma^k,u_1^k,u_2^k,q_1^k,q_2^k)
$$ 
for all $w \in [D_l,D_r],~z \in [\sigma_l,\sigma_r]$. Hence, it follows that 
\begin{equation}\label{sqh6_6}
\begin{aligned}
H_\epsilon(x,y,D,\sigma,D^k,\sigma^k,u_1^k,u_2^k,q_1^k,q_2^k)\leq &H_\epsilon(x,y,D^k,\sigma^k,D^k,\sigma^k,u_1^k,u_2^k,q_1^k,q_2^k)\\
=& H(x,y,D^k,\sigma^k,u_1^k,u_2^k,q_1^k,q_2^k)
\end{aligned}
\end{equation}
for all $(x,y)\in\Omega$. Let 
\[
\begin{aligned}
L(x,y,D,\sigma,u_1,u_2) = & \frac{\alpha}{2}\sum_{i=1}^2(\mathcal{H}(x,y,\sigma,u_i) - G_i^{\delta}(x,y))^2+\frac{\xi_1}{2}\sigma^2(x,y)\\
&+ \frac{\xi_2}{2}\left(D(x,y)-\frac{1}{100~(\sigma(x,y)+\sigma_b)}\right)^2
+\gamma  |\sigma(x,y)|.
\end{aligned}
\]
Then, we have
\begin{equation}\label{eq:1}
J(D,\sigma,u_1,u_2)-J(D^k,\sigma^k,u_1^k,u_2^k)=\int_{\Omega}L(x,y,D,\sigma,u_1,u_2)-L(x,y,D^k,\sigma^k,u_1^k,u_2^k)~dxdy.
\end{equation}
Adding and subtracting $D^k~\nabla q_i^k\cdot \nabla u_i+ (\sigma^k+\sigma_b) ~q_i^ku_i$ and 
$D^k~\nabla q_i^k\cdot \nabla u_i^k + (\sigma^k+\sigma_b) ~q_i^ku_i^k$ to the right-hand side of \eqref{eq:1}, we obtain
\[
\begin{aligned}
&J(D,\sigma,u_1,u_2)-J(D^k,\sigma^k,u_1^k,u_2^k)\\
=&\int_{\Omega}L(x,y,D,\sigma,u_1,u_2)~dxdy+\int_\Omega\sum_{i=1}^2 \left(D^k~\nabla q_i^k\cdot \nabla u_i+ (\sigma^k+\sigma_b) ~q_i^ku_i \right)~dxdy\\
&-\int_\Omega\sum_{i=1}^2 \left(D^k~\nabla q_i^k\cdot \nabla u_i+ (\sigma^k+\sigma_b) ~q_i^ku_i \right) ~dxdy\\
\int_{\Omega}&- L(x,y,D^k,\sigma^k,u_1^k,u_2^k)~dxdy-\int_\Omega\sum_{i=1}^2 \left(D^k~\nabla q_i^k\cdot \nabla u_i^k + (\sigma^k+\sigma_b) ~q_i^ku_i^k\right)~dxdy\\
&+\int_\Omega\sum_{i=1}^2 \left(D^k~\nabla q_i^k\cdot \nabla u_i^k + (\sigma^k+\sigma_b) ~q_i^ku_i^k \right)~dxdy.
\end{aligned}
\]

This gives us
\begin{align*}
&J(D,\sigma,u_1,u_2)-J(D^k,\sigma^k,u_1^k,u_2^k)\\
=&\int_{\Omega}  H(x,y,D,\sigma,u_1,u_2,q_1^k,q_2^k) -  H(x,y,D^k,\sigma^k,u_1^k,u_2^k,q_1^k,q_2^k) -\sum_{i=1}^2 \left(D^k~\nabla q_i^k\cdot \nabla u_i+ (\sigma^k+\sigma_b) ~q_i^ku_i \right)\\
&+\sum_{i=1}^2 \left(D^k~\nabla q_i^k\cdot \nabla u_i^k + (\sigma^k+\sigma_b) ~q_i^ku_i^k \right)~dxdy.\\
=&\int_{\Omega}  H(x,y,D,\sigma,u_1^k+\delta u_1,u_2^k + \delta u_2,q_1^k,q_2^k) -  H(x,y,D^k,\sigma^k,u_1^k,u_2^k,q_1^k,q_2^k) \\
&-\sum_{i=1}^2 \left(D^k~\nabla q_i^k\cdot \nabla u_i+ (\sigma^k+\sigma_b) ~q_i^ku_i \right)+\sum_{i=1}^2 \left(D^k~\nabla q_i^k\cdot \nabla u_i^k + (\sigma^k+\sigma_b) ~q_i^ku_i^k \right)~dxdy.\\
=&\int_{\Omega}  H(x,y,D,\sigma,u_1^k,u_2^k ,q_1^k,q_2^k) +\Gamma \alpha \sum_{i=1}^2\left[-(\sigma^k + \sigma_b)[\mathcal{H}(x,y,\sigma^k,u_i^k)-G^\delta_i]\delta u_i
+ \dfrac{1}{2} (\sigma^k + \sigma_b) (\delta u_i)^2 \right]\\
&  -H(x,y,D^k,\sigma^k,u_1^k,u_2^k,q_1^k,q_2^k) -\sum_{i=1}^2 \left(D^k~\nabla q_i^k\cdot \nabla u_i+ (\sigma^k+\sigma_b) ~q_i^ku_i \right)\\
&+\sum_{i=1}^2 \left(D^k~\nabla q_i^k\cdot \nabla u_i^k + (\sigma^k+\sigma_b) ~q_i^ku_i^k \right)~dxdy\\
\end{align*}

\begin{align*}
=&\int_{\Omega}  H(x,y,D,\sigma,u_1^k,u_2^k ,q_1^k,q_2^k) + \epsilon[(\delta D)^2+(\delta \sigma)^2] - H(x,y,D^k,\sigma^k,u_1^k,u_2^k,q_1^k,q_2^k)~dxdy\\
&+\int_{\Omega}  \sum_{i=1}^2\Bigg\lbrace\Gamma\alpha(\sigma^k + \sigma_b)[\mathcal{H}(x,y,\sigma^k,u_i^k)-G^\delta_i]u_i +\left(D^k~\nabla q_i^k\cdot \nabla u_i+ (\sigma^k+\sigma_b) ~q_i^ku_i \right)\Bigg\rbrace~dxdy\\
&-\int_{\Omega}  \sum_{i=1}^2\Bigg\lbrace\Gamma\alpha(\sigma^k + \sigma_b)[\mathcal{H}(x,y,\sigma^k,u_i^k)-G^\delta_i]u_i^k +\left(D^k~\nabla q_i^k\cdot \nabla u_i^k+ (\sigma^k+\sigma_b) ~q_i^ku_i^k \right)\Bigg\rbrace~dxdy\\
&+ \sum_{i=1}^2\int_{\Omega} \dfrac{1}{2} (\sigma^k + \sigma_b) (\delta u_i)^2 -\epsilon  \, [(\delta D)^2+(\delta \sigma)^2]~dxdy\\
=&\int_{\Omega}  H_\epsilon(x,y,D,\sigma,D^k,\sigma^k,u_1^k,u_2^k ,q_1^k,q_2^k) - H(x,y,D^k,\sigma^k,u_1^k,u_2^k,q_1^k,q_2^k)~dxdy\\
&+ \sum_{i=1}^2\int_{\Omega} \dfrac{1}{2} (\sigma^k + \sigma_b) (\delta u_i)^2~dxdy -\int_{\Omega}\epsilon  \, [(\delta D)^2+(\delta \sigma)^2]~dxdy, \qquad \mbox{ (using } \eqref{eq:DA_adj})\\
\leq &\sum_{i=1}^2\int_{\Omega} \dfrac{1}{2} (\sigma^k + \sigma_b) (\delta u_i)^2~dxdy -\int_{\Omega}\epsilon \, [(\delta D)^2+(\delta \sigma)^2]~dxdy, \qquad \mbox{ (using } \eqref{sqh6_6}) 
\end{align*}

Now, we recall the result of Lemma \ref{sqh20} that gives 
the following 
 $$
\|\delta u_i \|_{L^2(\Omega)}^2\leq \bar{C}_1 \, \big( \|\delta D \|_{L^{2}(\Omega)}^2 
+ \|\delta \sigma \|_{L^{2}(\Omega)}^2 \big) ,
 $$
 where $\bar{C}_1=2 \, c_3^2/c_1^2$. Hence, we proceed as follows 
\begin{align*}
J(D,\sigma,u_1,u_2)-J(D^k,\sigma^k,u_1^k,u_2^k)
\leq &\int_{\Omega}-\epsilon\left[(\delta D)^2+(\delta \sigma)^2\right]~dxdy+\sum_{i=1}^2\int_{\Omega} \dfrac{1}{2} (\sigma^k + \sigma_b) (\delta u_i)^2\,dx\\
\leq &-\epsilon\left[\|\delta D\|_{L^2(\Omega)}^2+\|\delta \sigma\|_{L^2(\Omega)}^2\right]+\sum_{i=1}^2\frac{\bar{\sigma}}{2} \|\delta u_i  \|_{L^2(\Omega)}^2\\
\leq&-\epsilon\left[ \|\delta D\|_{L^2(\Omega)}^2+\|\delta \sigma \|_{L^2(\Omega)}^2\right]+\bar{C}_1 \, \bar{\sigma} \, \left[\|\delta D\|_{L^2(\Omega)}^2+\|\delta \sigma \|_{L^2(\Omega)}^2\right] \\
= &-\left( \epsilon - \bar{C}_1 \, \bar{\sigma}  \right) \, \left[\|\delta D\|_{L^2(\Omega)}^2+\|\delta \sigma \|_{L^2(\Omega)}^2\right],
\end{align*}
where $\bar{\sigma}=\sigma_r + \sigma_b$. 
Thus, the theorem is proved with $\theta=\bar{C}_1 \, \bar{\sigma} $. 

\end{appendices}

\end{document}